\documentclass[reqno, 11pt,a4paper]{amsart}
\usepackage[percent]{overpic}
\usepackage{fullpage}
\usepackage{srcltx}
\usepackage{pdfsync}
\usepackage{calc,amsfonts,amsthm,amscd,epsfig,psfrag,amsmath,amssymb,enumerate,paralist,mathrsfs}
\usepackage[numeric,initials,nobysame]{amsrefs}
\usepackage[showlabels,sections,floats,textmath,displaymath]{}

\reversemarginpar
\newlength\fullwidth
\numberwithin{equation}{section}

\DeclareMathSymbol{\leqslant}{\mathalpha}{AMSa}{"36} 
\DeclareMathSymbol{\geqslant}{\mathalpha}{AMSa}{"3E} 
\DeclareMathSymbol{\eset}{\mathalpha}{AMSb}{"3F}     
\renewcommand{\leq}{\;\leqslant\;}                   
\renewcommand{\geq}{\;\geqslant\;}                   
\newcommand{\sumtwo}[2]{\sum_{\substack{#1 \\ #2}}} 
\renewcommand{\b}{\beta}
\newcommand{\bin}{\operatorname{Bin}}
\renewcommand{\restriction}{\mathord{\upharpoonright}}

\def\1{\ifmmode {1\hskip -3pt \rm{I}} \else {\hbox {$1\hskip -3pt \rm{I}$}}\fi}

\newcommand{\si}{\sigma }
\newcommand{\betaR}{\beta_{\textsc{r}}}

\newcommand{\D}{\Delta}

\renewcommand{\b}{\beta}
\renewcommand{\l}{\lambda}
\renewcommand{\L}{\Lambda}

\renewcommand{\l}{\lambda}
\renewcommand{\a}{\alpha}
\renewcommand{\d}{\delta}
\renewcommand{\t}{\tau}

\newcommand{\g}{\gamma}
\newcommand{\G}{\Gamma}
\newcommand{\z}{\zeta}
\newcommand{\e}{\varepsilon}

\renewcommand{\o}{\omega}
\renewcommand{\O}{\Omega}

\newcommand{\tc}{\thinspace |\thinspace}

\newtheorem{theorem}{Theorem}[section]
\newtheorem{lemma}[theorem]{Lemma}
\newtheorem{proposition}[theorem]{Proposition}
\newtheorem{corollary}[theorem]{Corollary}
\newtheorem{remark}[theorem]{Remark}

\newtheorem{claim}[theorem]{Claim}
\newtheorem{definition}[theorem]{Definition}
\newtheorem{maintheorem}{Theorem}

\newtheorem*{question*}{Question}

\newtheorem*{remark*}{Remark}
\newtheorem*{idefinition*}{Definition}

\newcommand{\Z}{\mathbb Z}

\newcommand{\cA}{\ensuremath{\mathcal A}}

\newcommand{\cC}{\ensuremath{\mathcal C}}
\newcommand{\cD}{\ensuremath{\mathcal D}}
\newcommand{\cE}{\ensuremath{\mathcal E}}
\newcommand{\cF}{\ensuremath{\mathcal F}}
\newcommand{\cG}{\ensuremath{\mathcal G}}
\newcommand{\cH}{\ensuremath{\mathcal H}}
\newcommand{\cI}{\ensuremath{\mathcal I}}

\newcommand{\cK}{\ensuremath{\mathcal K}}
\newcommand{\cL}{\ensuremath{\mathcal L}}
\newcommand{\cM}{\ensuremath{\mathcal M}}
\newcommand{\cN}{\ensuremath{\mathcal N}}

\newcommand{\cR}{\ensuremath{\mathcal R}}
\newcommand{\cS}{\ensuremath{\mathcal S}}

\newcommand{\cV}{\ensuremath{\mathcal V}}
\newcommand{\cW}{\ensuremath{\mathcal W}}

\newcommand{\cZ}{\ensuremath{\mathcal Z}}

\newcommand{\bbE}{{\ensuremath{\mathbb E}} }

\newcommand{\bbL}{{\ensuremath{\mathbb L}} }

\newcommand{\bbN}{{\ensuremath{\mathbb N}} }

\newcommand{\bbP}{{\ensuremath{\mathbb P}} }

\newcommand{\bbR}{{\ensuremath{\mathbb R}} }

\newcommand{\bbZ}{{\ensuremath{\mathbb Z}} }

\newcommand{\sC}{{\ensuremath{\mathscr C}}}

\newcommand{\sB}{{\ensuremath{\mathscr B}}}

\newcommand{\wt}{\widetilde }

\newcommand{\Dim}{\textsc{d} }

\begin{document}
\title[SOS surfaces above a wall]
{Scaling limit and cube-root fluctuations in SOS surfaces above a wall}
\author[P. Caputo]{Pietro Caputo}
 \address{P. Caputo\hfill\break Dipartimento di Matematica,
   Universit\`a Roma Tre, Largo S.\ Murialdo 1, 00146 Roma, Italia.}
\email{caputo@mat.uniroma3.it}

\author[E. Lubetzky]{Eyal Lubetzky}
\address{E.\ Lubetzky\hfill\break
Microsoft Research\\ One Microsoft Way\\ Redmond, WA 98052-6399, USA.}
\email{eyal@microsoft.com}

 \author[F. Martinelli]{Fabio Martinelli}
 \address{F. Martinelli\hfill\break Dipartimento di Matematica,
   Universit\`a Roma Tre, Largo S.\ Murialdo 1, 00146 Roma, Italia.}
\email{martin@mat.uniroma3.it}

\author[A. Sly]{Allan Sly}
\address{A. Sly\hfill\break
Department of Statistics\\
UC Berkeley\\
Berkeley, CA 94720, USA.}
\email{sly@stat.berkeley.edu}

 \author[F.L. Toninelli]{Fabio Lucio Toninelli}
 \address{F.L. Toninelli\hfill\break
CNRS and Universit\'e Lyon 1, Institut Camille Jordan,
43 bd du 11 novembre 1918,
69622 Villeurbanne, France
}
\email{toninelli@math.univ-lyon1.fr}

\begin{abstract}
  Consider the classical $(2+1)$-dimensional Solid-On-Solid model above a
  hard wall on an $L\times L$ box of $\bbZ^2$. The model
  describes a crystal surface by assigning a
  non-negative integer height $\eta_x$ to each site $x$ in the box and
  0 heights to its boundary. The probability of a surface
  configuration $\eta$ is proportional to $\exp(-\beta
  \mathcal{H}(\eta))$, where $\beta$ is the inverse-temperature and
  $\mathcal{H}(\eta)$ sums the absolute values of height
  differences between neighboring sites.

  We give a full  description of the shape of the SOS
  surface for low enough temperatures. First we show that with high
  probability the height of almost all sites is concentrated on two
  levels, $H(L)=\lfloor (1/4\beta)\log L\rfloor$ and $H(L)-1$.
  Moreover, for most values of $L$ the height is concentrated on
  the single value $H(L)$.  Next, we study the ensemble of level lines
  corresponding to the heights $(H(L),H(L)-1,\ldots)$. We prove that
  w.h.p.\ there is a unique macroscopic level line for each
  height. Furthermore, when taking a diverging sequence of system
  sizes $L_k$, the rescaled macroscopic level line at height $H(L_k)-n$ has a limiting shape if the fractional parts of $(1/4\beta)\log L_k$
  converge to a noncritical value. The scaling limit is
  an explicit convex subset of the unit square $Q$ and its boundary has a flat component on the
  boundary of $Q$. Finally, the highest macroscopic
  level line has $L_k^{1/3+o(1)}$ fluctuations along the flat part of
  the boundary of its limiting shape.
\end{abstract}


\thanks{This work was supported by the European Research Council through the ``Advanced
Grant'' PTRELSS 228032}

\maketitle
\vspace{-0.5cm}
\section{Introduction}\label{sec:intro}

The $(d+1)$-dimensional \emph{Solid-On-Solid} model is a crystal surface model whose definition goes back to Temperley~\cite{Temperley} in 1952 (also known as the Onsager-Temperley sheet).
At low temperatures, the model approximates the interface between the plus and minus phases in the $(d+1)$\Dim Ising model, with particular interest stemming from the study of 3\Dim Ising.

The configuration space of the model on a finite box $\Lambda\subset\Z^d$ 
with zero boundary conditions is the set 
of all height functions $\eta$ on $\Z^d$ such that $\Lambda\ni x \mapsto \eta_x \in \Z$ whereas $\eta_x=0$ for all $x\notin\Lambda$.
The probability of $\eta$
is given by the Gibbs distribution proportional to
\begin{equation}
  \label{eq-ASOS}
  \exp\bigg(-\beta \sum_{x\sim y}|\eta_x-\eta_y|\bigg)\,,
\end{equation}
where $\beta>0$ is the inverse-temperature and $x\sim y$ denotes a nearest-neighbor bond in $\Z^d$. 

Numerous works have studied the rich random surface phenomena, e.g.\ roughening, localization/delocalization, layering and wetting to name but a few, exhibited by the SOS model and some of its many variants. These include the \emph{discrete Gaussian} (replacing $|\eta_x-\eta_y|$ by $|\eta_x-\eta_y|^2$ for the integer analogue of the Gaussian free field), \emph{restricted} SOS (nearest neighbor gradients restricted to $\{0,\pm1\}$), \emph{body centered} SOS~\cite{vanBeijeren2}, etc.\
(for more on these flavors see e.g.~\cites{Abraham,Baxter,Bolt2}).

Of special importance is SOS with $d=2$, the only dimension featuring a \emph{roughening transition}.
For $d=1$, it is well known (\cites{Temperley,Temperley56,Fisher}) that the SOS surface is \emph{rough} (delocalized) for any $\beta>0$, i.e., the expected height at the origin diverges (in absolute value) in the thermodynamic limit $|\Lambda|\to\infty$. However, for $d\geq 3$
it is known that the surface is \emph{rigid} (localized) for any $\beta>0$ (see~\cites{BFL}), i.e., $|\eta_0|$ is uniformly bounded in expectation. A simple Peierls argument shows that this is also the case for $d=2$ and large enough $\beta$ (\cites{BW,GMM}). That the surface is rough for $d=2$ at high temperatures was established in seminal works of Fr{\"o}hlich and Spencer~\cites{FS1,FS2,FS3}.
Numerical estimates for the critical inverse-temperature $\betaR$ where the roughening transition occurs suggest that $\betaR \approx 0.806$.

%
%
When the $(2+1)$\Dim SOS surface is constrained to stay above a hard wall (or floor), i.e.\ $\eta$ is constrained to be non-negative in~\eqref{eq-ASOS}, Bricmont, El-Mellouki and Fr{\"o}hlich~\cite{BEF} showed in 1986 the appearance of \emph{entropic repulsion}: for large enough $\beta$, the floor pushes the SOS surface to diverge even though $\beta > \betaR$. More precisely, using Pirogov-Sina\"{\i} theory (see the review~\cite{Sinai}), the authors of~\cite{BEF} showed that the SOS surface on an $L\times L$ box rises, amid the penalizing zero boundary, to an average height in the interval $\left[(1/C\b)\log L \;,\; (C/\b)\log L\right]$ for some absolute constant $C>0$,
in favor of freedom to create spikes downwards.
In a companion paper~\cite{CLMST}, focusing on the dynamical evolution of the model, we established that the average height is in fact $(1/4\beta)\log L$ up to an additive $O(1)$-error.

Entropic repulsion is one of the key
features of the physics of random surfaces. This phenomenon
has been rigorously analyzed mainly for some continuous-height variants of the SOS model in which the interaction potential $|\eta_x-\eta_y|$ is replaced by a \emph{strictly convex} potential
$V(\eta_x-\eta_y)$; see, e.g.,~\cites{BDZ,BDG,CV,DG,Bolt,Vel1,Vel2}, and also~\cite{ADM} for a recent analysis of the wetting transition in the SOS model.
It was shown in the companion paper~\cite{CLMST} that entropic repulsion drives the evolution of the surface under the natural single-site dynamics. Started from a flat configuration, the surface rises to an average height of $(1/4\b)\log L-O(1)$ through a sequence of metastable states, corresponding roughly to plateaux at heights $0,1,2,\ldots, (1/4\b)\log L$.

Despite the recent progress on understanding the typical height of the surface, little was known on its actual 3\Dim shape. The fundamental problem is the following:
\begin{question*}
Consider the ensemble of all level lines of the low temperature $(2+1)$\Dim SOS on an $L\times L$ box with floor, rescaled to the unit square.
\begin{enumerate}[(i)]
\item \label{q-item-1} Do these jointly converge to a scaling limit as $L\to\infty$, e.g., in Hausdorff distance?
    \item \label{q-item-2} If so, can the limit be explicitly described?
    \item \label{q-item-3} For finite large $L$, what are the fluctuations of the level lines around their limit?
\end{enumerate}
\end{question*}
In this work we fully resolve parts~\eqref{q-item-1} and~\eqref{q-item-2} and partially answer part~\eqref{q-item-3}. En route, we also establish that for most values of $L$ the surface height concentrates on the single level $\lfloor \frac1{4\beta}\log L\rfloor$.

\subsection{Main results}
We now state our three main results. As we will see, two parameters rule the
  macroscopic behavior of the SOS surface:
  \begin{equation}
  \label{h_l}
  H(L)=\left\lfloor \frac{1}{4\beta}\log(L)\right\rfloor
  \,,\qquad \a(L)=\frac{1}{4\beta}\log(L)- H(L)\,,
  \end{equation}
  namely the
  integer part and the fractional part of $\frac{1}{4\beta}\log(L)$.
  The first result
  states that with probability tending to one as $L\to \infty$  the SOS surface has a large fraction of sites at
  height equal either to $H(L)$ or to $H(L)-1$. Moreover only one of the two
  possibilities holds depending on whether $\a(L)$ is above or below a
critical threshold that can be expressed in terms of an explicit critical
parameter $\l_c=\l_c(\b)$. The second result describes the
macroscopic shape for large  $L$ of any finite collections of level
lines at height $H(L), H(L)-1,\dots$. The third result establishes
cube root fluctuations of the level lines along the
flat part of its macroscopic shape.

In what follows we consider boxes $\L$ of the form
$\L=\L_L=[1,L]\times [1,L]$, $L\in\bbN$. We write $\pi^0_\L$ for
the SOS distribution on $\L$ with floor and
boundary condition at zero, and let $\hat{\pi}^0_\L$ be its analog without a floor. The function $L\mapsto \l(L)\in (0,\infty)$ appearing below is explicitly given in terms of $\a(L)$ (see Section~\ref{sec:proof-ideas} and Definition~\ref{lambdan} below). For large $\b$ it satisfies  $\l(L)e^{-4\b \a(L)}\simeq 1$.

\begin{figure}
\begin{center}
\includegraphics[width=0.4\textwidth]{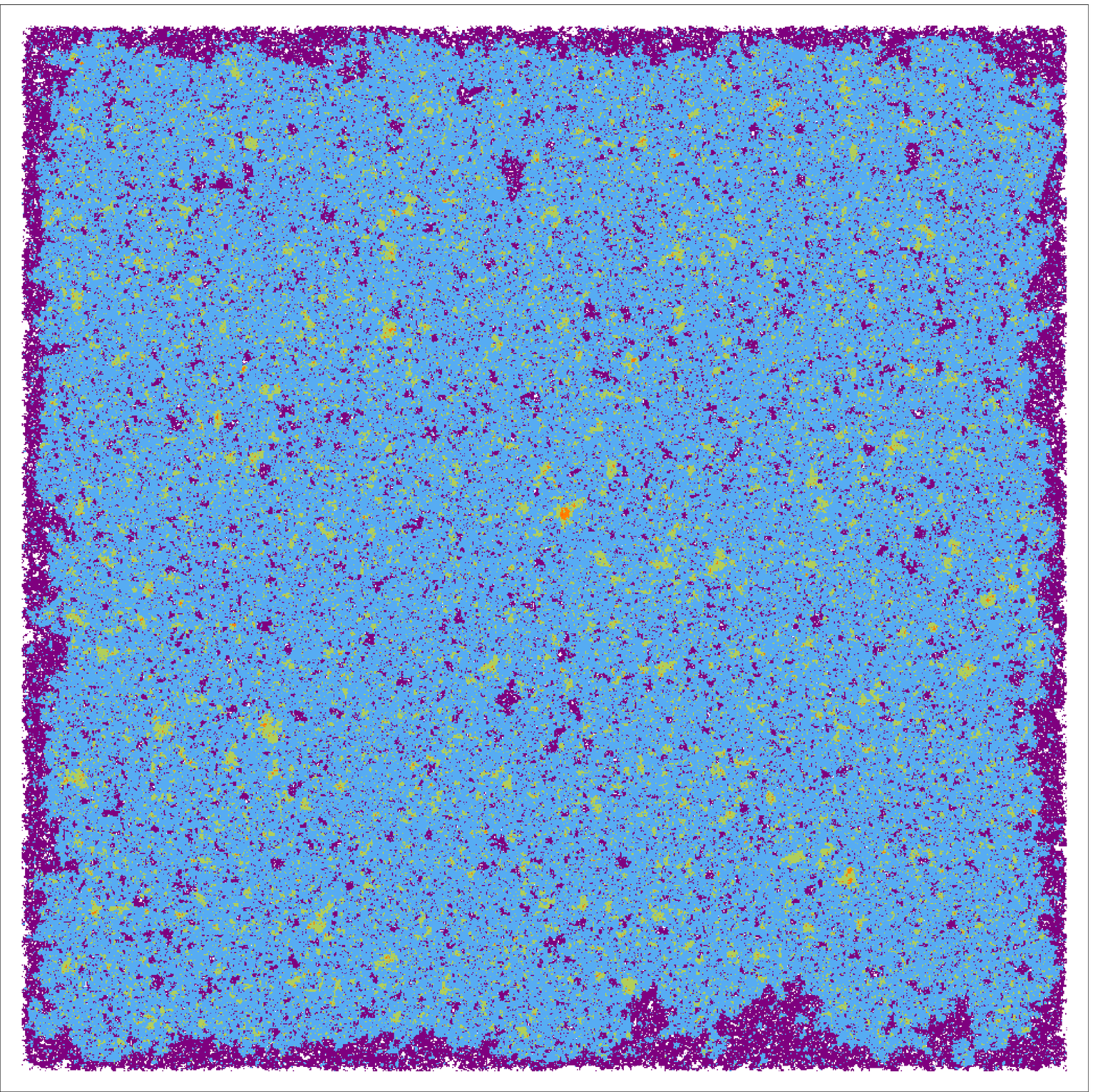}\hspace{0.1in}
\includegraphics[width=0.4\textwidth]{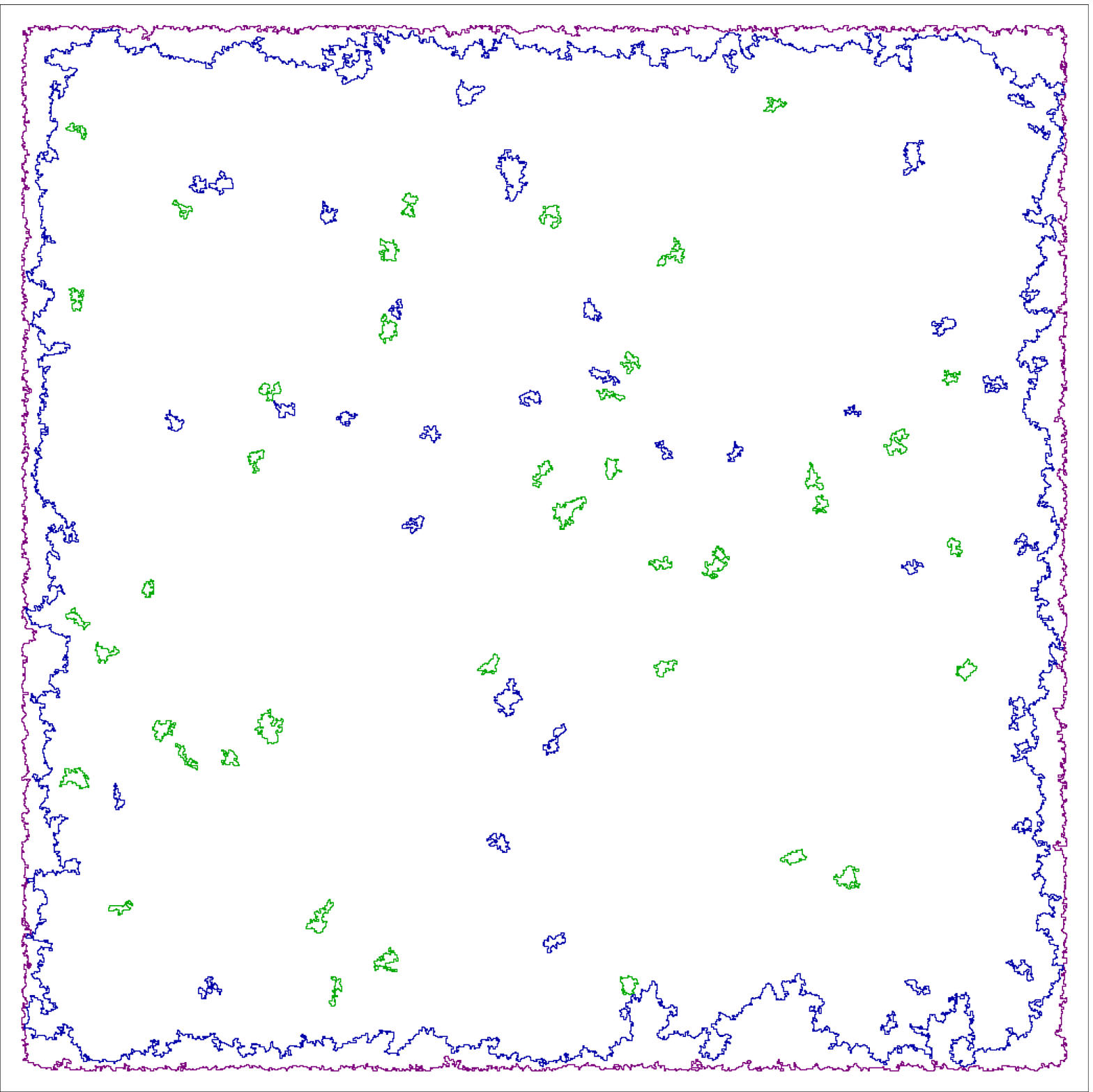}
\end{center}
\vspace{-0.1in}
\caption{Loop ensemble formed by the level lines of an SOS configuration on a box of side-length $1000$ with floor (showing loops longer than 100).}
\label{fig:loops}
\end{figure}

\begin{maintheorem}[Height Concentration]\label{mainthm-1}
Fix $\beta > 0$ sufficiently large and define \[E_h = \left\{ \eta : \#\{x : \eta_x=h\} \geq \tfrac9{10}L^2\right\}.\]
Then the SOS measure $\pi_\L^0$ on the box $\Lambda=\L_L$
with floor, at inverse-temperature $\beta$, satisfies
\begin{equation}
  \label{eq-two-levels}
  \lim_{L\to\infty} \pi^0_\Lambda\left(E_{H(L)  - 1} \cup E_{H(L)}\right) = 1\,.
\end{equation}
Furthermore, the typical height of the configuration is governed by
$L\mapsto \l(L)$
as follows.
  Let $\Lambda_k$ be a diverging sequence of boxes with side-lengths
  $L_k$.
For an explicit constant $\l_c>0$ (given by~\eqref{eq-def-lambdac}) we have:
\begin{enumerate}
  [(i)]
  \item If $\liminf_{k\to\infty}\l(L_k) > \l_c$ then $\lim_{k\to\infty} \pi^0_{\Lambda_k}\left(E_{H(L_k)}\right) = 1$.\\
  \item If $\limsup_{k\to\infty}\l(L_k) < \l_c$ then $\lim_{k\to\infty} \pi^0_{\Lambda_k}\left(E_{H(L_k)-1}\right) = 1$.
\end{enumerate}
\end{maintheorem}

\begin{remark}
The constant $\frac{9}{10}$ in the definition of $E_h$ can be replaced by $1-\epsilon$ for any arbitrarily small $\epsilon>0$ provided that $\beta$ is large enough.
As shown in Remark~\ref{rem:betalargo}, for large enough fixed $\beta$, $\lambda_c\simeq 4 \beta$ whereas $\lambda(L) \simeq e^{4\beta \alpha(L)}$, and hence most values of $L\in\bbN$ will yield $\pi^0_\Lambda(E_{H(L)})=1-o(1)$.
\end{remark}
It is interesting to compare these results to the 2\Dim Gaussian free field (GFF) conditioned to be non-negative, qualitatively akin to high-temperature SOS.
It is known~\cite{BDG} that the height of the GFF  in the box $\L=\L_L$ with floor and zero boundary condition, at any point $x\in\L$ such that ${\rm dist}(x,\partial \L)\geq \d L$, $\d>0$, is asymptotically the same as the maximal height in the unconditioned GFF in $\L$.
On the other hand, our results show that the SOS surface is lifted to height $H(L)$ or $H(L)-1$, which is asymptotically only one half of the SOS unconditioned maximum. Moreover, on the comparison of the maxima of the fields with and without wall we obtain the following:

\begin{corollary}\label{maincor-maximum}
Fix $\beta>0$ large enough, let $X^*_L$ be the maximum of the SOS surface
on the box $\Lambda=\Lambda_L$ with floor, and let $\widehat{X}^*_L$ be its analog in the SOS model without floor. Then for any diverging sequence $\varphi(L)$ one has:
\begin{align}
  \label{heightvsmax}
  \lim_{L\to\infty} \hat\pi^0_\Lambda
  \left(|\widehat{X}^*_L-\tfrac1{2\b}\log L| \ge \varphi(L)
 \right) &= 0\,,\\
  \label{maxvsmax}
  \lim_{L\to\infty} \pi^0_\Lambda
    \left(|{X}^*_L-\tfrac3{4\b}\log L| \ge \varphi(L)
 \right) &= 0\,.
\end{align}
\end{corollary}

We now address the scaling limit of the ensemble of level lines.
The latter is described as follows (see Section~\ref{sec:Notation} for
the full details). As for the 2D Ising model (see
e.g.~\cites{DKS,BIV}),  for the low-temperature SOS model without
floor there is a natural notion of surface tension $\t(\cdot)$
satisfying the  strict convexity property. We emphasise that the
surface tension we consider here is constructed in the usual way,
namely by imposing Dobrushin type conditions (between height zero and
height one) around a box. Consider the associated Wulff shape, namely
the convex body with support function $\t$, and let  $\cW_1$ denote
the Wulff shape rescaled to enclose  area $1$. For a given $s>0$, define the shape $\cL_c(s)$ by taking the union of all possible translates of $\ell_c(s)\cW_1$ within the unit square, with an explicit dilation parameter $\ell_c(s)$ which satisfies $\ell_c(s)\sim \frac{2\b}{s}$ for large $\b$. (Of course, $\cL_c(s)$ is defined only if $\ell_c(s)\cW_1$ fits inside a unit square.) Next, for a fixed $\l_\star>0$, consider the nested shapes $\{\cL_c(\l_\star^{(n)})\}_{n\geq0}$ obtained by taking $s$ equal to $\l_\star^{(n)}:=e^{4\b n}\l_\star$, $n=0,1,\dots$ as shown in Figure~\ref{fig:corners}.
\begin{figure}
\psfig{file=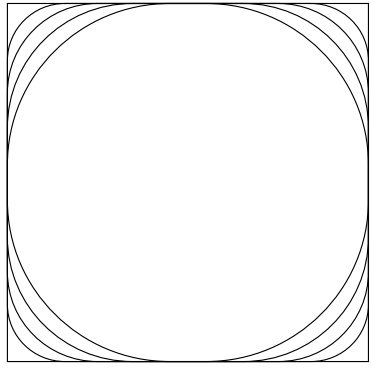,width=0.3\textwidth}
\vspace{-0.1in}
\caption{The nested limiting shapes $\{\cL_c(\l_\star^{(n)})\}$ of the rescaled loop ensemble $\frac{1}{L}\{\Gamma_n\}$.}
\label{fig:corners}
\vspace{-0.25cm}
\end{figure}

The next theorem gives a necessary and sufficient condition for the existence of the scaling limit of ensemble of level lines in terms of the above defined shapes. As a convention, in the sequel we often write ``with high probability'', or w.h.p., whenever the probability of an event is at least $1-e^{-c(\log L)^2}$ for some constant $c>0$.
\begin{maintheorem}[Shape Theorem]\label{mainthm-2}
 Fix $\beta > 0$  sufficiently large and let $L_k$ be a diverging sequence of
 side-lengths. Set $H_k=H(L_k)$. For an SOS surface on the box
 $\Lambda_k$ with side $L_k$, let
 $(\G_0^{(k)},\G_1^{(k)},\ldots)$ be the collections of
 loops with length  at least $(\log L_k)^2$
belonging to the level lines at heights $(H_k,H_k  -1,\ldots)$, respectively. Then:

\begin{enumerate}[(a)]
\item W.h.p.\ 
the level lines of every height $h> H_k $ consist
of loops shorter than  $(\log L_k)^2$, while $\G_0^{(k)}$ is either empty or contains a single loop,
and $\G_n^{(k)}$ consists of exactly one loop for each $n\geq 1$.

\item If $\l_\star:=\lim_{k\to\infty} \l(L_k)$ exists and differs from $\l_c$ (as given by~\eqref{eq-def-lambdac})
     then the rescaled loop ensemble
  $\frac{1}{L_k}(\G_0^{(k)},\G_1^{(k)},\ldots)$ converges to a limit in the Hausdorff distance: for any $\e>0$, w.h.p.\
  \begin{itemize}
  \item\label{it-shape-a} If $\l_\star > \l_c$ then
    \[
    \sup_{n\ge 0} d_\mathcal{H}\big(\tfrac{1}{L_k}\G_n^{(k)},\cL_c(\l_\star^{(n)}) \big)\leq \e
    \,,\]
where $d_\cH$ denotes the Hausdorff distance.
\item\label{it-shape-b} If instead $\l_\star < \l_c$ then $\G^{(k)}_0$ is
  empty while
    \[
    \sup_{n\ge 1} d_\mathcal{H}\big(\tfrac{1}{L_k}\G_n^{(k)},\cL_c(\l_\star^{(n)})\big ) \leq \e\,.\]
  \end{itemize}
\end{enumerate}
\end{maintheorem}
As for the critical behavior, from the above theorem we immediately read that it is possible to
have $\l(L_k) \to \l_c$ without admitting a scaling limit for the loop ensemble (consider a sequence
that oscillates between the subcritical and supercritical regimes). However, understanding the
critical window around $\l_c$ and the limiting behavior there remains an interesting open problem.

The fluctuations of the loop $\G^{(k)}_0$ (the macroscopic plateau at level $H(L_k)$ if
it exists) from its limit $\cL_c(\l_\star)$ along the side-boundaries are
now addressed. As shown in Figure~\ref{fig:corners}, the boundary of the limit shape $\cL_c(\l_\star)$
coincides with the boundary of the unit square $Q$ except for a
neighborhood  of the four corners of $Q$.
Let the interval $[a,1-a]$, $a=a(\l_\star)>0$,  denote the horizontal projection of the
intersection of the shape $\cL_c(\l_\star)$ with the bottom side of the unit square $Q$.

\begin{maintheorem}
  [Cube-root Fluctuations] \label{mainthm-3}
  In the setting of Theorem~\ref{mainthm-2} suppose
 $\l_\star>\l_c$.
Then for any $\epsilon>0$, w.h.p.\
the vertical fluctuation of $\G_0^{(k)}$ from the boundary
interval \[
I^{(k)}_\epsilon=[a(1+\epsilon) L_k,(1-a(1+\epsilon))L_k]
\]
is of order
$L_k^{\frac 13+o(1)}$.
More precisely, let $\rho(x) = \max\{ y\le \frac 12L_k : (x,y)\in
\G_0^{(k)}\}$ be the vertical fluctuation of $\G_0^{(k)}$ from the
bottom boundary of $\L_{k}$ at coordinate $x$.
Then 
w.h.p.
\[L_k^{\frac 13-\epsilon} < \sup_{x\in I^{(k)}_\epsilon} \rho(x)< L_k^{\frac 13+\epsilon} .\]
\end{maintheorem}

\begin{remark}\label{rem:forte}
 We will actually prove the stronger fact that w.h.p.\ a fluctuation of at least $L_k^{\frac 13-\epsilon}$ is
 attained in every sub-interval of $I_\epsilon^{(k)}$ of length
 $L_k^{\frac 23-\epsilon}$ (cf.\ Section~\ref{pfth3}).
\end{remark}

As a direct corollary of Theorem~\ref{mainthm-3} it was deduced in~\cite{CRASS} that the following upper bound on
the fluctuations of \emph{all} level lines $\G_n^{(k)}$ ($n\ge 1$) holds.

\begin{corollary}[Cascade of fluctuation exponents]\label{maincor-4}
In the same setting of Theorem~\ref{mainthm-3}, let $\rho(n,x)$ be the vertical fluctuation of $\G_n^{(k)}$ from the bottom boundary at coordinate $x$.
Let $0<t<1$ and let $n = \lfloor t H_k \rfloor$. Then for any $\epsilon>0$,
\[ \lim_{k\to\infty} \pi^0_{\Lambda_k}\bigg( \sup_{x\in I^{(k)}_\epsilon} \rho(n,x) > L_k^{\frac{1-t}3+\epsilon} \bigg) = 0\,.\]
\end{corollary}


\subsection{Related work}

In the two papers~\cites{ScSh1,ScSh2} Schonmann and Shlosman studied the limiting shape of the low temperature 2\Dim Ising with minus boundary under a prescribed small positive external field, proportional to the inverse of the side-length $L$. The behavior of the droplet of plus spins in this model is qualitatively similar to
the behavior of the top loop $\G_0$ in our case. Here, instead of an external field, it is the entropic repulsion phenomenon which induces the surface to rise to level $ H(L)$ producing the macroscopic loop $\G_0$.
In line with this connection, the shape $\cL_c(s)$ appearing in Theorem~\ref{mainthm-2}
is constructed in the same way as the limiting shape of the plus droplet in the aforementioned works, although with a different Wulff shape.
In particular, as in~\cites{ScSh1,ScSh2}, the shape $\cL_c(s)$ arises as the solution to a variational problem; see Section~\ref{sec:Notation} below.

An important difference between the two models, however, is the fact that in our case there exist $H (L)$ levels (rather than just one), which are interacting in two nontrivial ways. First, by definition, they cannot cross each other. Second, they can weakly either attract or repel one another depending on the local geometry and height. Moreover, the box boundary itself can attract or repel the level lines. A prerequisite to proving Theorem~\ref{mainthm-2} is to overcome these ``pinning'' issues. We remark that at times such pinning issues have been overlooked in the relevant literature.

As for the fluctuations of the plus droplet from its limiting shape, it was argued in~\cite{ScSh1} that these should be normal (i.e., of order $\sqrt{L}$).
However, due to the analogy mentioned above between the models, it follows from our proof of Theorem~\ref{mainthm-3} that these fluctuations are in fact $L^{1/3+o(1)}$ along the flat pieces of the limiting shape, while it seems  natural to conjecture that normal fluctuations appear along the curved portions, where the limiting shapes corresponding to distinct levels are macroscopically separated; see Figure~\ref{fig:corners}.

There is a rich literature of contour models featuring similar cube root fluctuations. In some of these works (e.g.,~\cites{Alexander,FeSp,HV,Vel2}) the phenomenon is induced by an externally imposed constraint (by conditioning on the event that the contour contains a large area and/or by adding an external field); see in particular \cite{Vel2} for the above mentioned case of the $2\Dim$ Ising model in a weak external field, and the recent works \cites{Ham1,Ham2} for refined bounds in the case of FK percolation. In other works, modeling ordered random walks (e.g.,~\cites{CH,Johansson} to name a few), the exact solvability of the model (e.g., via determinantal representations) plays an essential role in the analysis.
In our case, the phenomenon is again a consequence of the tilting of the distribution of contours induced by the entropic repulsion. The lack of exact solvability for the $(2+1)\Dim$ SOS, forces us to  resort to cluster expansion techniques and contour analysis
as in the framework of~\cites{DKS}.
We conclude by mentioning some problems that remain unaddressed by our results. First is to establish
the exponents for the fluctuations of all intermediate level lines from the side-boundaries.
We believe the upper bound in Corollary~\ref{maincor-4} features the correct cascade of exponents.
Second, find the correct fluctuation exponent of the level lines $\{\G_n\}$ around the curved part of their limiting shapes $\{\cL_c(\l_\star^{(n)})\}$. Third, we expect that, as in~\cites{Alexander, Ham1, Ham2}, the fluctuation exponents of the highest level line around its convex envelope is $1/3$,
while, as mentioned above, there should be normal fluctuations around the curved parts of the deterministic limiting shape.

\subsection{On the ensemble of macroscopic level lines}\label{sec:proof-ideas}
We turn to a high-level description of the statistics of the level lines of the SOS interface.
Given a \emph{closed contour} $\g$ (i.e., a closed loop of dual edges as for the standard Ising model), a positive integer $h$ and a surface configuration $\eta$ we say that $\g$ is an $h$-contour (or $h$-level line) for $\eta$, if the surface height jumps from being at least $h$ along the internal boundary of $\g$ to at most $h-1$ along the external boundary of $\g$. Clearly an $h$-contour $\g$ is energetically penalized proportionally to its length $|\g|$ because of the form of the SOS energy function. As in many spin models admitting a contour representation, with high probability contours are either all small, say $|\g|=O((\log L)^2)$, or there exist macroscopically large ones (i.e.\ $|\g|\propto L$),  if  $L$ is the size of the system. An instance of the first situation is the SOS model \emph{without} a wall and zero boundary conditions (see \cite{BW}). On the contrary, macroscopic contours appear in the low temperature 2D Ising model with \emph{negative} boundary conditions and a \emph{positive} external field $\cH$ of the form $\cH=B/L$, $B>0$, as in \cites{ScSh1,ScSh2}. In this case the probabilistic weight of a contour separating the inside plus spins from the outside minus spins is roughly given by
$\exp\big(-\b|\g| + \Psi(\g)+ m_\b^*\cH A(\g)\big)$, where $A(\g)$ denotes the area enclosed by $\g$, $m_\b^*$ is the spontaneous magnetisation and $\Psi(\g)$ is a ``decoration" term which is not essential for the present discussion. If the parameter $B$ is above a certain threshold then the area term dominates the boundary term and a macroscopic contour appears with high probability. Moreover, by simple isoperimetric arguments, the macroscopic contour is unique in this case.

The SOS model with a wall shares some similarities with the Ising example above but has a richer structure that can be roughly described as follows.

Suppose $\{\g_1,\g_2,\dots,\g_n\}$ are macroscopic $h$-contours corresponding to heights $h=1, 2,\dots n$ and that no other macroscopic contour exists. Then necessarily the collection $\{\g_i\}_{i=1}^n$ must consist of \emph{nested} contours, with $\g_n$ and $\g_1$ being the innermost and outermost contour respectively. If we denote by $\L_i$ the region enclosed by $\g_i$ and by $A_i$ the annulus $\L_{i}\setminus \L_{i+1}$, then the partition function of all the surfaces satisfying the above requirements can be written as
\[
Z(\g_1,\dots,\g_n)= \exp\Big(-\b \sum_i |\g_i|\Big)\prod_i Z_{A_i}
\]
where $Z_{A_i}$ is the partition function of the SOS model in $A_i$, with a wall at height zero, boundary conditions at height $h=i$ and restricted to configurations without macroscopic contours\footnote{Strictly speaking one should also require that the height is at most $i$ (at least $i$)  along the inner (outer) boundary of the annulus, but we skip these details for the present discussion.}.

Usually (see, e.g.,~\cite{DKS}) in these cases one tries to exponentiate the partition functions $Z_{A_i}$ using cluster expansion techniques. However, because of the presence of the wall, one cannot apply directly this approach and it is instead more convenient to compare
$Z_{A_i}$ with $\hat Z_{A_i}$, where $\hat Z_{A_i}$ is as $Z_{A_i}$ but \emph{without} the wall. One then observes that the ratio $Z_{A_i}/\hat Z_{A_i}$ is simply the probability that the surface is non-negative computed for the Gibbs distribution of the SOS model in $A_i$ with  boundary conditions at height $h=i$, no wall, and conditioned to have no macroscopic contours. The key point, which was already noted in \cite{CLMST}, is that w.r.t.\ the above Gibbs measure the random variables $\{{\mathbf 1}_{\eta_x\ge 0}\}_{x\in A_i}$  behave approximately as i.i.d.\ with
\[
\bbP(\eta_x\ge 0) \simeq 1- c_\infty e^{-4\b(i+1)},
\]
where is $c_\infty$ a computable constant. Therefore,
\[
Z_{A_i} \simeq \exp\left(-c_\infty e^{-4\b(i+1)} |A_i|\right) \hat Z_{A_i}.
\]
In conclusion, rewriting $|A_i |= |\L_i|-|\L_{i+1}|$,
\[
Z(\g_1,\dots,\g_n)\propto  \exp\left( \sum_i \Bigl[-\b |\g_i| +c_\infty e^{-4\b i}(1-e^{-4\b})|\L_i|\Bigr]\right)\prod_i \hat Z_{A_i}.
\]
The terms proportional to the area encode the effect of the entropic repulsion and play  the same role as the magnetic field term in the Ising example. Cluster expansion techniques can now be applied to the partition functions $\hat Z_{A_i}$ without wall. As in many other similar cases their net result is the appearance of a decoration term $\Psi(\g_i)=O(\e_\b |\g_i|)$ for each contour, $\e_\b$ small for large $\b$, and an effective many-body \emph{interaction} $\Phi(\g_1,\dots, \g_n)$ among the contours which however is very rapidly decaying with their mutual distance. Thus the probability of the above macroscopic contours should then be proportional to
\begin{equation}
\label{ensemble}
\exp\left(\sum_i \Bigl[-\b |\g_i| +c_\infty e^{-4\b i}(1-e^{-4\b})|\L_i|+\Psi(\g_i)\Bigr] +\Phi(\g_1,\dots, \g_n)\right).
\end{equation}
Note that in each term of the above sum the area part is dominant up to height $i\simeq (1/4\b)\log(L)$. In other words, macroscopic $h$-contours are sustained by the entropic repulsion up to a height $h\simeq (1/4\b)\log(L)$ while higher contours are exponentially suppressed. More precisely, if we measure heights relatively to $H(L)=\lfloor (1/4\b)\log(L)\rfloor$, then the $i^{\rm th}$ area term can be rewritten as
\[
\l\,\frac{ e^{4\b (H(L)-i)}}{L} |\L_i|,\quad \text{with}\quad \l=\l(L)=c_\infty e^{4\b \a(L)}(1-e^{-4\b}).
\]
The quantity $\l(L)$ is  exactly the key parameter appearing in the main theorems. Notice that the loop $\G_0$ appearing in Theorem~\ref{mainthm-2}  would correspond to the contour $\g_n$ if  $n=H(L)$.

Summarizing, the macroscopic contours behave like nested random loops with an area bias and with some interaction potential $\Phi$.
While the latter is in many ways a weak perturbation, in principle delicate pinning effects may occur among the different level lines, as  emphasized earlier.

Although one could  try to implement directly the above line of reasoning, we found it more convenient to combine the above ideas with  monotonicity properties of the model (w.r.t.\ the height of the boundary conditions and/or the height of the wall) to reduce ourselves always to the analysis of one single macroscopic contour at a time. That allowed us to partially overcome the above mentioned pinning problem. On the other hand, a stronger control of the interaction between level lines seems to be crucial in order to determine the correct fluctuation exponents of all the intermediate level lines from the side-boundaries (see the open problems discussed above).

\section{General tools}
In this section we collect some preliminary definitions together with basic results which will be used several times throughout the paper. Once combined together with  standard cluster expansion methods  they quantify precisely the effect of the entropic repulsion from the floor.
\subsection{Preliminaries}
In order to formulate our first tools we need a bit of extra notation.
\subsubsection{Boundary conditions and infinite volume limit.} Given a height function $\bbZ^2\ni x\mapsto \tau_x\in \Z$ (the boundary condition) and  a finite set $\L\subset \bbZ^2$ we denote by $\O_\L^{(\tau)}$ (resp.\ $\hat \O_\L^{(\tau)}$) all the height functions
$\eta$ on $\Z^2$ such that $\Lambda\ni x \mapsto \eta_x \in \Z_+$ (resp.\ $\Lambda\ni x \mapsto \eta_x \in \Z$) whereas $\eta_x=\tau_x$ for all $x\notin\Lambda$.  The corresponding Gibbs measure given by~\eqref{eq-ASOS} will be denoted by $\pi_\L^\t$ (resp.\ $\hat \pi_\L^\t$). In other words $\pi_\L^\t$ describes the SOS model in $\L$ with boundary conditions $\t$ and floor at zero while
$\hat\pi_\L^\t$ describes the SOS model in $\L$ with boundary
conditions $\t$ and no floor. The corresponding partition functions
will be denoted by $Z_\L^{\t}$ and $\hat Z_\L^\t$ respectively. If
$\tau$ is constant and equal to $j\in \bbZ$ we will simply replace $\tau$
by $j$ in all the notation.
We will denote by $\hat \pi$ the
infinite volume Gibbs measure obtained as the thermodynamic limit of the measure $\hat\pi_\L^0$ along an
increasing sequence of boxes. The limit exists and does not depend on the sequence of boxes; see \cite{BW}.

 \subsubsection{Contours and level lines}
The level lines of the SOS surface, and the corresponding loop ensemble they give rise to, are formally defined as follows.
\begin{definition}[Geometric contour]
\label{contourdef}
We let ${\Z^2}^*$ be the dual lattice of $\Z^2$ and we call a {\sl bond}
any segment joining two neighboring sites in ${\Z^2}^*$.
Two sites $x,y$ in $\Z^2$ are said to be {\sl separated by a bond $e$} if
their distance (in $\bbR^2$) from $e$ is $\tfrac12$. A pair of orthogonal
bonds which meet in a site $x^*\in {\Z^2}^*$ is said to be a
{\sl linked pair of bonds} if both bonds are on the same side of the
forty-five degrees line across $x^*$. A {\sl geometric contour} (for
short a contour in the sequel) is a
sequence $e_0,\ldots,e_n$ of bonds such that:
\begin{enumerate}
\item $e_i\ne e_j$ for $i\ne j$, except for $i=0$ and $j=n$ where $e_0=e_n$.
\item for every $i$, $e_i$ and $e_{i+1}$ have a common vertex in ${\Z^2}^*$
\item if $e_i,e_{i+1},e_j,e_{j+1}$ intersect at some $x^*\in {\Z^2}^*$,
then $e_i,e_{i+1}$ and $e_j,e_{j+1}$ are linked pairs of bonds.
\end{enumerate}
We denote the length of a contour $\gamma$ by $|\gamma|$, its interior
(the sites in $\bbZ^2$ it surrounds) by $\Lambda_\gamma$ and its
interior area (the number of such sites) by
$|\Lambda_\gamma|$. Moreover we let $\Delta_{\gamma}$ be the set of sites in $\Z^2$ such that either their distance
(in $\bbR^2$) from $\gamma$ is $\tfrac12$, or their distance from the set
of vertices in ${\Z^2}^*$ where two non-linked bonds of $\gamma$ meet
equals $1/\sqrt2$. Finally we let $\Delta^+_\gamma=\Delta_\gamma\cap
\L_\g$ and $\Delta^-_\gamma = \Delta_\gamma\setminus \Delta^+_\gamma$.
\end{definition}

\begin{definition}[$h$-contour]
\label{def:levels}
Given a contour $\gamma$ we say that $\gamma$ is an \emph{$h$-contour} (or an $h$-level line)
for the configuration $\eta$ if
\[
\eta\restriction_{\Delta^-_\gamma}\leq h-1, \quad \eta\restriction_{\Delta^+_\gamma}\geq h.
\]
We will say that $\g$ is a contour for the configuration $\eta$ if
there exists $h$ such that $\g$ is an $h$-contour for $\eta$.
Contours
  longer than $(\log L)^2$ will be called \emph{macroscopic contours}\footnote{This convention is slightly abusive, since the term {\em macroscopic} is usually reserved to objects with size comparable to the system size $L$. However, as we will see, it is often the case in our context that with overwhelming probability, there are no contours
  at intermediate scales between $(\log L)^2$ and $L$.}.
Finally $\sC_{\gamma,h}$ will denote the event that $\gamma$ is an $h$-contour.
\end{definition}
\begin{definition}[Negative $h$-contour]
We say that a closed contour $\gamma$ is a negative $h$-contour if the external boundary $\gamma$ is at least $h$ whereas its internal boundary is at most $h-1$. That is to say, denoting this event by $\sC_{\gamma,h}^-$, we have that $\eta \in \sC_{\gamma,h}^-$ iff $\eta\restriction_{\Delta^+_\gamma}\leq h-1$ and $\eta\restriction_{\Delta^-_\gamma}\geq h$. \end{definition}


\subsubsection{Entropic repulsion parameters}
In order to define key parameters measuring the entropic repulsion we first need the following Lemma
whose proof is postponed to Appendix~\ref{sec:pbeta-pf}.
\begin{lemma}
\label{pbeta}For $\beta$ large enough the limit $c_\infty:= \lim_{h\to \infty}e^{4\b
  h}\hat \pi(\eta_0 \ge h)$ exists and
  \[
  |c_\infty -e^{4\b h}\hat \pi(\eta_0 \ge h)|= O(e^{-2\b h}).
  \]
 Moreover $\lim_{\b\to \infty}c_\infty=1$.
\end{lemma}
\begin{definition}
\label{lambdan}Given an integer $L>1$ we define
\begin{equation}
  \label{eq-def-lambda}
\lambda:=\lambda(L)=e^{4\beta\, \a(L)}
c_\infty(1-e^{-4\beta})
\end{equation}
where $\a(L)$ denotes the fractional part of $\frac{1}{4\beta}\log L$.
Also, for $n\ge 0$, we let $\l^{(n)}:=\l^{(n)}(L)=\l  e^{4\b n}$.
\end{definition}

\subsection{An isoperimetric inequality for contours}
 The following simple lemma will prove useful in establishing the existence of macroscopic loops.
    \begin{lemma}\label{le:isop}
    For all $\delta'>0$ there exists a $\delta>0$ such the following holds.
 Let $\{\gamma_i\}$ be a collection of closed contours with areas $A(\g_i)$
   satisfying $A(\gamma_1)\geq A(\gamma_2)\geq \dots$, and suppose that
   \[
   \sum_i|\g_i|\le (1+\d)4L\,,\quad \text{and}\quad\sum_iA(\g_i)\ge (1-2\d)L^2.
   \]
   Then
  the interior of $\g_1$ contains a sub-square of area at least $(1-\delta')L^2$.
\end{lemma}
\begin{proof}
Define
  $\a_i= A(\g_i)\left[(1-2\d)L^2\right]^{-1}$ so that $\sum_i\a_i\ge 1$. Then  (see, e.g.,~\cite{DKS}*{Section~2.8}) $\sum_i\sqrt{\a_i}\ge \frac{1}{\sqrt{\a_1}}$.
 Using 
 the isoperimetric bound $|\g_i|\ge 4\sqrt{A(\g_i)}$, it follows that
  \begin{gather*}
    (1+\d)4L\ge \sum_i|\g_i|\ge
    4\sqrt{(1-2\d)L^2}\sum_i\sqrt{\a_i}\\
\ge 4\sqrt{(1-2\d)L^2}\frac{1}{\sqrt{\a_1}}=4(1-2\d)L^2\frac{1}{\sqrt{A(\g_1)}}.
  \end{gather*}
This implies, for $\d$ small enough,
\[
A(\g_1)\ge \frac{(1-2\d)^2}{(1+\d)^2} L^2\ge (1-8\d)L^2.
\]
Noting that the unit square is the unique shape with area at least 1 and $L^1$-boundary length at most $4$ it follows by continuity that for all $\delta'>0$ there exists $\delta>0$ such that if a curve has length at most $(1+\delta)4$ and area at least $1-8\d$ then it contains a square of side-length $1-\delta'$ implying the last assertion of the lemma.
\end{proof}

\subsection{Peierls estimates and entropic repulsion}
\label{pee}
Our first result is a upper bound  on the probability of encountering a given $h$-contour and it is  a refinement of \cite{CLMST}*{Proposition 3.6}. Recall the definition of the height $H(L)$, of the parameters $\l^{(n)}$ and of the events $\sC_{\gamma,h},\, \sC^-_{\gamma,h}$.
\begin{proposition}
\label{bdgma}
Fix $j\ge 0$ and consider the SOS model in  a finite connected subset
$V$ of $\bbZ^2$ with floor at height $0$ and boundary
conditions at height $j\ge 0$. 
There exists $\d_h$ and $\varepsilon_\b$ with $\lim_{h\to \infty}\d_h=\lim_{\b\to \infty}\varepsilon_\b=0$ such that, for all $h\in\bbN$:
\begin{align}
\label{e:contourFloorBound}
\pi^j_V\left( \sC_{\gamma,h} \right) &\leq
\exp\left(-\beta|\gamma|+c_\infty(1+\d_h)|\Lambda_{\gamma}|e^{-4\beta
    h}\right)
\exp\left( \varepsilon_\b\,e^{-4\beta
    h}|\g|\log(|\g|)\right), \\
\label{neg-cont}  \pi^j_V\big( \sC^-_{\gamma,h} \big)   &\le e^{-\b |\g|}.
\end{align}
\end{proposition}
\begin{remark}
Notice that if $h=H(L)-n$ then $c_\infty e^{-4\b h}= (1-e^{-4\b})^{-1}\l^{(n)}/L$.
\end{remark}
\begin{proof}[Proof of~\eqref{e:contourFloorBound}] Let $Z_{\rm in}^{+,n}$ (resp.\ $\hat Z_{\rm in}^{+,n})$ be the partition function of
the SOS model in $\L_\g$
 with floor at height $0$ (resp.\ no floor), b.c.\ at height $n$ and
 $\eta\restriction_{\D_\g^+}\ge n$. Similarly call $Z_{\rm out}^{-,n}$ be the partition function of
the SOS model in $V\setminus \L_\g$
 with floor at height $0$, b.c.\ at height $j$ along $\partial V$,
 at height $n$ along $\g$ and satisfying
 $\eta\restriction_{\D_\g^-}\le n$.  One has
 \begin{align*}
  \pi^j_\Lambda\left( \sC_{\gamma,h} \right)&= e^{-\b |\g|}\ \frac{Z_{\rm
      out}^{-,h-1}Z_{\rm in }^{+,h}}{Z_V^j} \le e^{-\b |\g|}\ \frac{Z_{\rm
      out}^{-,h-1}Z_{\rm in }^{+,h}}{Z_{\rm out}^{-,h-1}Z_{\rm in
    }^{+,h-1}}
=   e^{-\b |\g|}\ \frac{Z_{\rm in }^{+,h}}{Z_{\rm in  }^{+,h-1}}\\ &\le
e^{-\b |\g|}\ \frac{\hat Z_{\rm in }^{+,h}}{Z_{\rm in  }^{+,h-1}}= e^{-\b |\g|}\ \frac{\hat Z_{\rm in }^{+,h-1}}{Z_{\rm in  }^{+,h-1}},
 \end{align*}
where in the last equality we used the fact that $\hat Z_{\rm in
}^{+,n}$ is independent of $n$.

Let now $\hat \pi^{n}_{\L_\g}$ be the Gibbs measure in $\L_\g$ with
b.c.\ at height $n$ and no floor. Then, using first monotonicity and then the FKG inequality:
\begin{align}
\frac{Z_{\rm in }^{+,h-1}}{\hat Z_{\rm in  }^{+,h-1}}&=\hat
\pi^{h-1}_{\L_\g}\left(\eta\restriction_{\L_\g}\ge
  0
  \,\,\big|\,\, \eta\restriction_{\D_\g^+}\ge h-1\right)\nonumber\\
&\ge \hat
\pi^{h-1}_{\L_\g}\left(\eta\restriction_{\L_\g} \ge 0\right)  
\ge \prod_{x\in \L_\g}\hat \pi^{h-1}_{\L_\g}\left(\eta_x\ge 0\right)=
\prod_{x\in \L_\g}\left[1-\hat \pi^{0}_{\L_\g}\left(\eta_x\ge h\right)\right].
\label{giacinto}
 \end{align}
It follows from \cite{CLMST}*{Proposition 3.9} that $\max_{x\in \L_\g}\hat
\pi^{0}_{\L_\g}\left(\eta_x\ge h\right)\le c \exp(-4\b h)$ for some
constant $c$ independent of $\beta$. Moreover, using the exponential decay of correlations of
the SOS measure without floor (cf.~\cite{BW}), we obtain
\begin{gather*}
\hat \pi^{0}_{\L_\g}\left(\eta_x\ge h\right)
\le \begin{cases}
c\ e^{-4\beta h}& \text{ if ${\rm dist}(x,\g) \le \varepsilon_\b\log(|\L_\g|)$}\\
\hat \pi\left(\eta_x\ge
  h\right)| + 1/|\L_\g|^2& \text{otherwise}
\end{cases}
  \end{gather*}
with $\lim_{\beta \to \infty}\varepsilon_\b=0$. If we now use Lemma~\ref{pbeta} to write $\hat \pi\left(\eta_x\ge
  h\right)= c_\infty(1+\d_h)e^{-4\b h}$, $\lim_{h\to \infty}\d_h=0$, we get
\begin{gather*}
\prod_{x\in \L_\g}\left[1-\hat \pi^{0}_{\L_\g}\left(\eta_x\ge
    h\right)\right]= \!\!\!\!\!\!\!\!\!\!\prod_{\substack{x\in \L_\g \\ {\rm dist}(x,\g)\le \varepsilon_\b\log(|\L_\g|)}}\!\!\!\!\!\!\!\!\!\!\left[1-\hat \pi^{0}_{\L_\g}\left(\eta_x\ge
    h\right)\right]\times \!\!\!\!\!\!\!\!\!\!\prod_{\substack{x\in \L_\g \\ {\rm dist}(x,\g)> \varepsilon_\b\log(|\L_\g|)}}\!\!\!\!\!\!\!\!\!\!\left[1-\hat \pi^{0}_{\L_\g}\left(\eta_x\ge
    h\right)\right]\\
\ge \exp\left(- c\, \varepsilon_\b\,e^{-4\beta
    h}|\g|\log(|\L_\g|)\right)\exp\left(-c_\infty(1+\d_h)e^{-4\beta
    h}|\L_\g|\right).
  \end{gather*}
The proof is complete using $|\L_\g|\le |\g|^2/16$.
\end{proof}
\begin{proof}[Proof of~\eqref{neg-cont}]
With the same notation as before we write
\begin{align*}
 \pi^j_\Lambda\left( \sC^-_{\gamma,h} \right)&= e^{-\b |\g|}\ \frac{Z_{\rm
      out}^{+,h}Z_{\rm in }^{-,h-1}}{Z_V^j} \le e^{-\b |\g|}\ \frac{Z_{\rm
      out}^{+,h}Z_{\rm in }^{-,h-1}}{Z_{\rm out}^{+,h}Z_{\rm in
    }^{-,h}} \le e^{-\b |\g|}\,.
    \qedhere
\end{align*}
\end{proof}
The next result is a simple geometric criterion to exclude certain large contours.
\begin{lemma}
  \label{prop:nocontornoni}
Fix $n\in \bbZ$ and consider the measure $\pi_V^{h-1}$ of the SOS model in  a finite connected subset
$V$ of $\bbZ^2$, with floor at height $0$ and boundary
conditions at height $h-1$ where $h:=H(L)-n$. Let $c_0=2\log(3)$. If
\begin{equation}
\label{eq:card}
|V|\le \left[\frac{4(\b-c_0) (1-e^{-4\b})L}{\l^{(n)}}\right]^2,
\end{equation}
then w.h.p.\ there are no macroscopic contours.
\end{lemma}
An immediate consequence of the above bound is that for any sufficiently large $\beta$ one can exclude the
existence of macroscopic $(H(L)+1)$-contours.
\begin{corollary}\label{cor:H+1-contours}
Let $\Lambda$ be the square of side-length $L$ and let $\beta$ be large enough.
Then  w.h.p.\ the SOS measure $\pi_\Lambda^0$ does not admit any $(H(L)+1)$-contours of
length larger than $\log(L)^2$.
\end{corollary}
\begin{proof}[Proof of the Corollary]
The statement is just a special case of Lemma~\ref{prop:nocontornoni} with $n=-1$ and $V=\Lambda$. In this case
the inequality~\eqref{eq:card} is obvious since $\l^{(-1)}=e^{-4\b}\l\le c_\infty \le 2$ for $\b$ large.
\end{proof}
\begin{proof}[Proof of Lemma~\ref{prop:nocontornoni}]
The statement is an easy consequence of Proposition~\ref{bdgma}, applied with $j=h-1$. Let us first show that w.h.p.\ there are no macroscopic $h$-contours. First of all we observe that the error term
$\exp\left( \varepsilon_\b\,e^{-4\beta
    h}|\g|\log(|\g|)\right)$ in the r.h.s.\ of~\eqref{e:contourFloorBound} is at most $\exp\left(c'L^{-1}|\g|\log L\right)$ for some constant $c'=c'(\b,n)$ because $|\g|\le |V|\le c L^2$. Hence it is negligible w.r.t.\ the main term
    $\exp\left(-\beta|\gamma|+c_\infty(1+\d_h)|\Lambda_{\gamma}|e^{-4\beta
    h}\right) $. If we now use the inequality $|\L_\g|\le |V|^{1/2}|\g|/4$ we get immediately that, under the stated assumption on the cardinality of $V$, the area term satisfies
    \begin{gather*}
     c_\infty(1+\d_h)|\Lambda_{\gamma}|e^{-4\beta h}= \frac{(1+\d_h)\l^{(n)}}{(1-e^{-4\b})L} |\Lambda_{\gamma}|     \leq (1+\d_h)(\b-c_0)|\g|.
     \end{gather*}
 Hence the probability that a macroscopic $h$-contour exists can be bounded from above by
 \[
 \sum_{\g: |\g|\ge \log(L)^2}e^{-(c_o-c'L^{-1}\log L+\b\d_h)|\g|}=O(e^{-c\log(L)^2})
 \]
 for some constant $c$. Clearly, if no macroscopic $h$-contour exists then there is no macroscopic $j$-contour for $j\ge h$. It remains to rule out macroscopic $j$-contours with $j\le h-1$. However the existence of such a contour would imply the existence of a negative macroscopic contour  and such an event has probability  $O(e^{-c\log(L)^2})$ because of Proposition~\ref{bdgma}.
\end{proof}

Fix $n\in \bbZ$ and consider the SOS model in  a finite connected subset
$V$ of $\bbZ^2$, with floor at height $0$ and boundary
conditions at height $h-1$ where $h:=H(L)-n$. Let $\partial _*V$ denote the set of  $y\in V$ either at distance $1$ from $\partial V$
or at distance $\sqrt 2$ from $\partial V$ in the south-west or north-east direction. In particular, if $V$ is the set $\L_\g$ corresponding to a contour $\g$, then $\partial_* V=\D_\g^+$.
For a fixed $U\subset \partial_*V$, define
the partition function $Z^{h-1,+}_{V,U}$ (resp. $Z^{h-1,-}_{V,U}$) of the SOS model on $V$ with boundary condition $h-1$ on $\partial V$, with floor at height $0$ and with the further constraint that $\eta_y\geq h-1$ (resp. $\eta_y\leq h-1$), for all $y\in U$.
We write
$\hat Z^{h-1,\pm}_{V,U}$ for the same partition functions \emph{without} the floor constraint. By translation invariance, $\hat Z^{h-1,\pm}_{V,U}$ does not depend on $h$.
We let $\pi_{V,U}^{h-1,\pm}$ and $\hat \pi_{V,U}^{h-1,\pm}$ be the Gibbs measures associated to the partition functions $Z^{h-1,\pm}_{V,U}$ and $\hat Z^{h-1,\pm}_{V,U}$ respectively.

\begin{remark}
\label{rem:generalizzo1}
Exactly the same argument given above shows that Lemma~\ref{prop:nocontornoni}
applies as it is to the measures $\pi_{V,U}^{h-1,\pm}$ for any $U\subset\partial_* V$.
\end{remark}

The next proposition quantifies the effect of the floor constraint.
\begin{proposition}
  \label{prop:grigliatona} In the above setting, fix $\e \in (0,1/10)$ and assume that 
  $|\partial V|\le L^{1+\e}$. Then 
  \begin{equation}
\label{grill-lwb}
Z^{h-1,\pm}_{V,U}\geq \hat Z^{h-1,\pm}_{V,U}\exp\Bigl(-c_\infty e^{-4\b h}|V| + O(L^{\frac12+2\e})\Bigr).
\end{equation}
If, in addition,~\eqref{eq:card} holds,
then
\begin{equation}
\label{grill-upb}
Z^{h-1,\pm}_{V,U}\leq \hat Z^{h-1,\pm}_{V,U}\exp\Bigl(-c_\infty e^{-4\b h}|V| + O(L^{\frac12+c(\b)})\Bigr),
\end{equation}
where $c(\b)\to 0$ as $\b\to\infty$.
\end{proposition}
\begin{remark}
\label{rem:generalizzo}
In Section~\ref{US-section} we will apply the above result to sets $V$ with area of order $L^2$. In this case the error terms in~\eqref{grill-lwb}--\eqref{grill-upb} will be negligible (recall that $e^{-4\b h}\propto L^{-1}$). In Section~\ref{gadget} we will instead apply it to sets with area of order $L^{4/3}$ and then it will be necessary to refine it and show that, in this case, the error term becomes $o(1)$.
\end{remark}

The core of the argument is to show that, w.r.t.\ the measure
$\hat \pi_{V,U}^{h-1,\pm}$, the Bernoulli variables
$\{{\mathbf 1}_{\eta_x\geq 0}\}_{x\in
  V}$ behave  essentially as i.i.d.\ random variables with
$\bbP({\mathbf 1}_{\eta_x\ge 0}=1)\approx 1-\hat\pi(\eta_0\ge h)$ where $\hat\pi$
is the infinite volume SOS model without floor.

\begin{proof}[Proof of~\eqref{grill-lwb}]
From the FKG inequality
\[
\frac{Z^{h-1,\pm}_{V,U}}{\hat Z^{h-1,\pm}_{V,U}}=\hat \pi_{V,U}^{h-1,\pm}(\eta_x\ge 0,\,\forall x\in V)
\ge \prod_{x\in V}\hat \pi_{V,U}^{h-1,\pm}(\eta_x\ge 0).
\]
At this point one can proceed exactly as in the proof of Proposition
\ref{bdgma} (see~\eqref{giacinto} and its sequel). Indeed, using
$|V|\leq |\partial V|^2\leq L^{2+2\e}$, $\d_h=O(e^{-2\beta h})=O(L^{-1/2})$, see Lemma~\ref{pbeta}, and $e^{-4\b h}=O(L^{-1})$,  one sees that
\begin{equation*}
\max\left(\d_he^{-4\b h}|V|, e^{-4\beta
    h}|\partial V|\log(|V|)\right)= O(L^{\frac12 + 2\e}).
\qedhere
\end{equation*}\end{proof}

\begin{proof}[Proof of~\eqref{grill-upb}]
The upper bound is more involved and it is here that the area
constraint plays a role. Without it, the entropic repulsion could
push up the whole surface and the product
$\prod_{x\in V}{\mathbf 1}_{\eta_x\ge 0}$ would no longer behave
(under $\hat \pi_{V,U}^{h-1,\pm}$) as a
product of i.i.d variables.

Let $\cS$ denote the event that there are no macroscopic contours and
use the identity
\[
\hat \pi_{V,U}^{h-1,\pm}\left(\eta_x\ge 0 \  \forall x \in V\right)=
\frac{\hat \pi_{V,U}^{h-1,\pm}(\cS)}{\pi_{V,U}^{h-1,\pm}(\cS)}\;
\hat \pi_{V,U}^{h-1,\pm}\left(\eta_x\ge 0 \  \forall x\in V\tc \cS\right).
\]
Thanks to Lemma~\ref{prop:nocontornoni} (see Remark~\ref{rem:generalizzo1}) and our area constraint, one has
$\pi_{V,U}^{h-1,\pm}(\cS)=1-o(1)$. Hence, it is enough to show that
\begin{eqnarray}
  \label{eq:condizionata}
\hat \pi_{V,U}^{h-1,\pm}\left(\eta_x\ge 0 \ \forall x\in V \tc
  \cS\right)\le \exp\Bigl(-\hat\pi(\eta_0\ge h)|V| + O(L^{\frac12+c(\b)})
  \Bigr),
\end{eqnarray}
where $c(\b)$ is a constant that can be made small if $\b$ is large, since then one can appeal to Lemma~\ref{pbeta} to write
$\hat\pi(\eta_0\ge h)|V| =c_\infty e^{-4\b h}|V| +O(L^{-3/2}|V|)$.
The estimate~\eqref{eq:condizionata} has been essentially already proved in \cite{CLMST}*{Section~7}. For the reader's convenience, we give the details in the Appendix~\ref{a44}. 
\end{proof}

\begin{remark}
For technical reasons, later in the proofs we will need Proposition~\ref{bdgma}, Lemma~\ref{prop:nocontornoni} and  Proposition~\ref{prop:grigliatona}  in a slightly more general case, in the sequel referred to as the \emph{``partial floor setting"}, in which the SOS model in $V$ has the floor constraint $\eta_x\ge 0$ only for those vertices $x$ inside a certain subset $W$ of $V$.
Exactly the same proofs show that in this new setting the very same statements hold with
$\L_\g$ replaced by
$|\L_\g\cap W|$ in~\eqref{e:contourFloorBound}, and with $|V|$ replaced by $|V\cap W|$ in~\eqref{eq:card} and in the exponent at~\eqref{grill-lwb}--\eqref{grill-upb}.
\end{remark}
We conclude by describing a monotonicity trick to upper bound
   the probability of an increasing event $A$, under the SOS measure
   $\pi^0_\L$ in some domain $\L$ with zero-boundary conditions and
   floor at height zero.
   \begin{lemma}[Domain-enlarging procedure]
\label{rem:mallargo}
 Let $(\Lambda\cup\partial \L)\subset V\subset \Lambda'$, let
   $\tau$ be a non-negative (but otherwise arbitrary) boundary
   condition on $\partial \L'$ and let  $\pi^\tau_{\L',V}$
   denote the SOS measure on $\L'$, with b.c.\ $\tau$ and floor at
   zero in $V$. Let $A$ be an increasing event in $\O_\L$. Then,
\begin{eqnarray}
  \label{eq:mallargo}
  \pi^0_\L(A)\le \pi^\tau_{\L',V}(A).
\end{eqnarray}
\end{lemma}
\begin{proof}
Note first of all that $ \pi^{\tau'}_{\L',V}(A)\le
\pi^\tau_{\L',V}(A)$ where $\tau'$ is obtained from $\tau $ by setting
$\tau'_x=0$ for every $x\in \partial \L'\cap \partial \L$. Then,
$\pi^0_\L$ can be seen as the marginal in $\L$ of the measure
$\pi^{\tau'}_{\L',V}$ conditioned on the decreasing event that
$\eta=0$ on $\partial \L$. By FKG, removing the conditioning can only
increase the probability of $A$.
\end{proof}

\subsection{Cluster expansion}
In order to write down precisely the law of certain macroscopic contours we shall use a cluster expansion for partition functions of the SOS with partial or no floor.
Given a finite connected set $V\subset \bbZ^2$ and $U\subset \partial_* V$ (the set $ \partial_* V$ has been defined before Proposition~\ref{prop:grigliatona}), we write
$\hat Z_{V,U}$ for the SOS partition function with the sum over $\eta$  restricted to those
$\eta\in\hat \O_V^0$ such that $\eta_x\geq 0$ for all $x\in U$. Notice that
 $\hat Z_{V,U}$ coincides with the partition function $\hat Z^{h,+}_{V,U}$ appearing in Proposition~\ref{prop:grigliatona}
 (the latter does not depend on $h$).
We refer the reader to \cite{CLMST}*{Appendix A} for a proof of the following expansion.
\begin{lemma}
\label{lem:dks}
There exists $\b_0>0$ such that for all $\b\geq \b_0$, for all finite connected $V\subset\bbZ^2$ and $U\subset \partial_* V$:
\begin{equation}\label{zetaexp}
\log \hat Z_{V,U}=\sum_{V'\subset V}\varphi_{U}(V'),
\end{equation}
where the potentials $\varphi_{U}(V')$ satisfy
\begin{enumerate}[(i)]
\item $\varphi_{U}(V')=0$ if $V'$ is not connected.
\item $\varphi_{U}(V')=\varphi_0(V')$ if ${\rm dist}(V', U)\neq 0$, for
  some shift invariant  potential $V'\mapsto \varphi_0(V')$
that is
\[
\varphi_0(V')=
\varphi_0(V'+x)\quad \forall \,x\in\bbZ^2\ .
\]
\item For all $V'\subset V$:
\[
\sup_{U\subset\partial_* V}|\varphi_{U}(V')|\le
  \exp(-(\b-\b_0)\, d(V'))
\]
where $d(V')$ is the cardinality of the smallest connected set of bonds of $\bbZ^2$ containing
  all the boundary bonds of $V'$ (i.e., bonds connecting $V'$ to its complement).
\end{enumerate}
\end{lemma}

\section{Surface tension and variational problem}
\label{sec:Notation}
In this section we first collect all the necessary information about surface
tension and associated Wulff shapes. We then consider the variational problem
of maximizing a certain functional which will play a key role
in our main results and describe its solution.

We begin by defining the surface tension of the SOS model
\emph{without} the wall (see also Appendix~\ref{app:ce}). We assume $\beta$ is large enough in order to enable
cluster expansion techniques~\cites{BW,DKS}.
\begin{definition}
\label{deftau}
Let
$\L_{n,m}= \{-n, \dots ,n\}\times \{-m,\dots,m\}$ and let
$\xi(\theta)$, $\theta\in
[0,\pi/2)$, be the boundary condition given by
\begin{equation*}
\xi(\theta)_y  =
\begin{cases}
+1, & \qquad \text{if} \quad  \vec{n} \cdot y \geq 0, \\
0, & \qquad \text{if} \quad  \vec{n} \cdot y <  0  \end{cases}
\qquad \forall y\in \partial \L_{n,m}
\end{equation*}
where $\vec n$ is the unit vector orthogonal to the line forming an angle $\theta$
with the horizontal axis.

The surface tension $\tau(\theta)$ in the direction $\theta$ is defined by
\begin{equation}
\label{surfacetension}
\tau(\theta) =  \lim_{n\to \infty}\lim_{m\to \infty} \, - \frac{\cos(\theta)}{2\beta n}
\,
\log\left( \frac{\hat Z_{\L_{n,m}}^{\xi(\theta)}}{\hat Z_{\L_{n,m}}^0}\right) \,
. \end{equation}
Using the symmetry of the SOS model  we
finally extend $\tau$ to an even, $\pi/2$-periodic function on
$[0,2\pi]$. Finally, if one extends
$\tau(\cdot)$ to $\bbR^2$ as
 $x\mapsto \tau(x):=|x|\tau(\theta_x)$, $\theta_x$ being the direction
 of $x$, then $\tau(\cdot)$ becomes  (strictly)
 convex and analytic. See  \cite{DKS}*{Ch. 1 and 2} for additional
 information\footnote{Strictly speaking,
  \cite{DKS} deals with the
  nearest-neighbor two-dimensional Ising model, but their proofs are
  immediately extended to our case. Also in the following, whenever a
  result of \cite{DKS} can be adapted straightforwardly to our
  context, we just cite the relevant chapter without an explicit
  caveat.}, and Appendix~\ref{app:ce} for an equivalent definition
 of $\tau(\cdot)$ in the cluster expansion language.
\end{definition}
Next we proceed to define the Wulff shape.
\begin{definition}
Given a closed rectifiable curve $\gamma$ in $\bbR^2$, let $A(\gamma)$
be the area of its interior and let $W(\gamma)$ be the Wulff
functional $\gamma\mapsto \int_{\gamma}\tau(\theta_s)ds$, with $\theta_s$ the direction of the normal with respect to the curve
$\gamma$ at the point $s$ and $ds$ the length element.
The convex body with support function $\tau(\cdot)$ (see
e.g.~\cite{convexanalysis}) is denoted by $\mathcal W_\t$. The rescaled set
\[
\mathcal W_1=\sqrt{\frac{2}{W(\partial \cW_\t)}}\times\cW_\t
\]
is called the \emph{Wulff} shape and it has unit area (see
e.g.~\cite{DKS}*{Ch. 2}). $\mathcal W_1$ is also the subset of $\bbR^2$
of unit area that minimizes the Wulff
functional.  We set $w_1:=W(\partial
\mathcal W_1)$.
\end{definition}

Now, given $\lambda>0$, consider the problem of
maximizing the functional
\begin{equation}\label{funct}
\gamma\mapsto \cF_\l(\g):= -\beta W(\gamma)+\lambda  A(\gamma)
\end{equation}
among all curves contained in the square $Q=[0,1]\times [0,1]$. In order
to solve this variational problem we proceed as follows.

We first observe that, if $\ell_\t$ denotes the side of the smallest
square with sides parallel to the coordinate axes into which $\cW_1$ can fit, then one has
  \[\ell_\t=  2 \sqrt{\frac{2}{W(\partial \cW_\t)}}\tau(0)=4\frac{\tau(0)}{w_1}.\]
\begin{remark}
\label{Wulff->square}
As $\beta$ tends to $\infty$, one has $\tau(\theta)\to
|\cos(\theta)|+|\sin(\theta)|$ (analyticity is lost in this limit)
and the Wulff shape converges to the unit square.
\end{remark}
We now set
 \begin{equation}
   \label{eq:16}
\hat \l =2\b \tau(0)\quad \text{;}\quad
\ell_c(\l)=\b w_1 /(2\l).
 \end{equation}
\begin{definition}
\label{Lstorto}
For $r,t,\l$ such that $0<t\ell_c\ell_\t\le 1$ and $r\in (-1,1)$
we define the convex
body $\cL(\l,t,r)$ as the $(1+r)$-dilation of the set formed by the union of all possible translates of $t\ell_c \mathcal
W_1$ contained inside $Q$. When $t=1$ and $r=0$ we write $\cL_c(\l)$ for $\cL(\l,1,0)$.
\end{definition}
\begin{remark}
\label{key-property}
We point out two properties of the parameters $\ell_c$ and $\hat \l$ that are useful to keep in mind.
The first one is that, by construction, the rescaled droplet $\ell_c \cW_1$ can fit inside the unit square $Q$ iff $\l\ge \hat \l$. The second one, as shown in Section~\ref{growth}, goes as follows. Consider the SOS model with floor in a box of side $L$ with zero boundary conditions and assume the existence of an  $(H(L)-n)$-contour containing the rescaled
Wulff body $L\ell_c(\l^{(n)}) \cW_1$. Necessarily that requires $\l^{(n)}\ge \hat \l$. Then w.h.p.\ the $(H(L)-n)$-contour actually contains the
whole region $L\cL_c(\l^{(n)})$ up to $o(L)$ corrections.
\end{remark}
\begin{claim}
\label{lambdac}
Set
\begin{equation}
  \label{eq-def-lambdac}
\l_c= \inf\{\l\ge \hat \l:\  \cF_\l(\cL_c(\l))>0\}
\,.
\end{equation}
Then $\l_c= \hat \l+ \b w_1/2$.
\end{claim}
\begin{proof}
Using the
  definitions of $\ell_c$, $\cL_c$ and $\ell_\t$, we can write
  \begin{align*}
W(\cL_c(\l))&= \ell_c w_1 +4\tau(0)(1-\ell_c\ell_\t)= \frac \b {2\l} w_1^2
+4\tau(0)\left(1-2\b\frac{\tau(0)}{\l}\right);\\
A(\cL_c(\l))&=1+ \frac{\b^2w_1^2}{4\l^2} -\frac{4\b^2\tau(0)^2}{\l^2}.
  \end{align*}
Hence
\[
\cF_\l(\cL_c(\l))= -4\b \tau(0)+\l -\frac{\b^ 2 w_1^2}{4\l} +4\frac{\b^2\tau(0)^2}{\l}.
\]
Solving the quadratic equation $\cF_\l(\cL_c(\l))= 0$ gives the
solutions
\begin{equation*}
\l_\pm=2\b\tau(0)\pm \b w_1/2=\hat\l \pm \b w_1/2.
\qedhere
\end{equation*}\end{proof}
\begin{remark}
\label{rem:betalargo}
In the limit $\b\to \infty$ we have: $\hat \l/\b\to 2$,
$\l_c/\b\to 4$ and $\ell_c(\l_c)\to 1/2$.
\end{remark}

\begin{claim}
\label{clamotto}
 Going back to the variational problem of maximizing $\cF_\l(\gamma)$,
the following holds \cite{ScSh2}:
\begin{enumerate}[(i)]
\item  if $\lambda<\lambda_c$ then the supremum
corresponds to a sequence of curves $\gamma_n$ that shrinks to a point, so that
$\sup_\g\cF_\l(\gamma)=0$; moreover for any $\d>0$ there exists $\epsilon>0$ such that $\cF_\l(\g)\leq  -\epsilon$ for any curve $\g$ enclosing an area larger than $\d$.
\item  if $\l>\l_c$ then the maximum is attained for $\gamma=\partial \cL_c(\l)$ and $\cF_\l(\partial \cL_c(\l))>0$.
\end{enumerate}
The area (or perimeter) of the optimal curve has therefore a
discontinuity at $\lambda_c$.
\end{claim}
We conclude with a last observation on the geometry of the Wulff
shape $\cW_1$ which will be important in the proof of Theorems~\ref{mainthm-2} and~\ref{mainthm-3}.
\begin{lemma}
\label{th:wulff_height}
Fix $\theta\in [-\pi/4,\pi/4]$ and $d\ll 1$. Let $I(d,\theta)$ be the
segment  of length $d$ and angle $\theta$ w.r.t.\ the
$x$-axis such that its endpoints lie on the boundary of
the Wulff shape $\cW_1$. Let $\D(d,\theta)$ be the vertical distance between the midpoint
of $I(d,\theta)$ and $\partial \cW_1$. Then
\[
\D(d,\theta)= \frac{w_1}{16 \left(\tau(\theta)+\tau''(\theta)\right)\cos(\theta)}d^2(1 +O(d^2))\quad
\text{as}\quad d\to 0.
\]
\end{lemma}
\begin{proof}
Let $x$ be the midpoint of $I(d,\theta)$ and let $h$ be the distance
between $x$ and $\partial \cW_1$. Clearly
$\D(d,\theta)=\frac{h}{\cos(\theta)}(1+O(h))$. From elementary considerations, as
$d\to 0$,
\[
h= \frac{d^2}{8R(\theta)}(1+O(d^2))
\]
where $R^{-1}(\theta)$ is the curvature of the Wulff shape $\cW_1$ at
angle $\theta$. It is known that (see, e.g.,~\cite{convexanalysis})
\begin{equation*}
R(\theta)= \sqrt{\frac{2}{W(\partial\cW_\tau)
  }}\left(\tau(\theta)+\tau''(\theta)\right)= \frac{2}{w_1}\left(\tau(\theta)+\tau''(\theta)\right)\,.
  \qedhere
\end{equation*}
\end{proof}

\section{Proof of Theorem~\ref{mainthm-1}}
\label{US-section}

\subsection{An intermediate step: existence of a supercritical $(H(L)-1)$-contour}\label{sec:H(L)-1-droplet}
Our first goal is to show that w.h.p.\ there exists a large droplet at level $H(L)-1$.

\begin{proposition}
  \label{prop:(H-1)-starting-point}
Let $\Lambda$ be a square of side-length $L$. If $\beta$ is large enough, the SOS measure $\pi_\Lambda^0$ admits an $(H(L)-1)$-contour $\gamma$ whose interior contains a square of side-length $\frac9{10} L$  w.h.p.
\end{proposition}
\begin{proof}[Proof of Proposition~\ref{prop:(H-1)-starting-point}]
The first ingredient  is a  bound addressing the contribution of microscopic contours to the height profile.
\begin{lemma}\label{lem:mass-at-H-from-small-contours}
Let $V\subset \Lambda$ where $\Lambda$ is a square of side-length $L$ with boundary condition $\xi\leq h-1$, where $h = H(L) - n$ for some fixed $n\geq 0$. Denote by
$\sB_h$ the event that there is no $h$-contour of length at least $\log^2 L$.
Then for any $\delta>0$ there are constants $C_1,C_2>0$ such that for any $\beta \geq C_1$,
\begin{equation}
  \label{eq:mass-at-H-from-small-contours}
 \pi_\Lambda^\xi\left( \#\{v : \eta_v \geq h\} > \delta  L^2 \,,\, \sB_h \right) \leq \exp(- C_2 \log^2 L) \,,
 \end{equation}
and for any closed contour $\gamma$
\begin{equation}
  \label{eq:mass-at-H-from-small-neg-contours}
\pi_\Lambda^\xi\left( \#\{v\in \Lambda_\gamma : \eta_v \leq h-1\} > \delta  L^2 \, \mid \sC_{\gamma, h} \right) \leq \exp(- C_2 \log^2 L) \,.
\end{equation}
\end{lemma}
\begin{proof}

For a configuration $\eta$ let $\mathcal{N}_k(\eta)$ denote the number
of $h$-contours of length $k \leq \log^2 L$.  As there are at most
$L^2 4^k$ possible such contours, by Proposition~\ref{bdgma} we have that for some constant $C_0>0$, for any $m$,
\begin{align*}
\pi_\Lambda^\xi (\mathcal{N}_k(\eta) \geq m) &\leq \sum_{r\geq m} \binom{4^k L^2}{r} e^{r(-\beta k+C_0e^{-4\beta h}k^2 )}
 \leq \sum_{r\geq m} \binom{4^k L^2}{r} e^{-r\beta k/2}\\
 &\leq \frac{\mathbb{P}(\bin(4^k L^2,e^{-\beta k/2})\geq m)}{(1 - e^{-\beta k/2})^{4^k L^2}} \leq
 \exp\left(2 e^{-\beta k/2} 4^k L^2\right) \mathbb{P}(\bin(4^k L^2,e^{-\beta k/2})\geq m)\,,
\end{align*}
where we used the fact that $1-x \geq e^{-2x}$ for $0 \leq x \leq \frac12$ as well as that $e^{-\beta k/2} \leq \frac12$ for $\beta$ large.
For each $1 \leq k \leq \log^2 L$ we now wish to apply the above inequality for a choice of
\[
m(k) = 7 \cdot 4^k L^2 e^{-\beta k/2} + \log^2 L\,.
\]
 By the well-known fact that $\mathbb{P}(X\geq \mu+t) \leq \exp[-t^2/(2(\mu+t/3))]$ for any $t>0$ and  binomial variable $X$ with mean $\mu$, which in our setting of $t \geq 6 \mu$ implies a bound of $\exp(-t)$, we get
 \[ \pi_\Lambda^\xi (\mathcal{N}_k(\eta) \geq m) \leq \exp\left(- 4  e^{-\beta k/2} 4^k L^2 - \log^2 L\right) \leq e^{-\log^2 L}\,.\]
Each $h$-contour counted by $\mathcal{N}_k(\eta)$ encapsulates at most $k^2$ sites of height larger than $h$, thus setting $M(L) = \sum_{k=1}^{\log^2 L}k^2 m(k)$ we get
\[ \pi_\Lambda^\xi\left( \#\{v : \eta_v \geq h\} > M(L) \,,\, \sB_h \right) \leq e^{-(1-o(1)) \log^2 L} \,.\]
The proof is concluded by the fact that $M(L) = O(e^{-\beta/2} L^2) + L^{1+o(1)} $ for any $\beta > 4\log 2$, where the $O(L^2)$-term is easily less than $\delta L^2$ for large enough $\beta$.

To prove~\eqref{eq:mass-at-H-from-small-neg-contours}, observe that by monotonicity
\[ \pi_\Lambda^\xi\left( \#\{v\in \Lambda_\gamma : \eta_v \leq h-1\} > \delta L^2 \, \mid \sC_{\gamma, h} \right) \leq
\pi_{\Lambda_\gamma}^h\left( \#\{v\in \Lambda_\gamma : \eta_v \leq
  h-1\} > \delta L^2 \,  \right) \,\]
(in the inequality we removed the constraint that the heights are at
least $h$ on $\Delta^+\gamma$).
Thus, if no large negative contours are present, the argument shown above for establishing~\eqref{eq:mass-at-H-from-small-contours} will imply~\eqref{eq:mass-at-H-from-small-neg-contours}. On the other hand, Proposition~\ref{bdgma} and a simple Peierls bound immediately imply that w.h.p.\ there exists no macroscopic negative contour.
\end{proof}

We now need to introduce the notion of external $h$-contours.
\begin{definition}
Given a configuration $\eta\in \O_\L$ we say  that $\{\g_i\}_{i=1}^n$ forms the collection of the external $h$-contours of $\eta$ if  every $\g_i$ is a macroscopic $h$-contour and there exists no other $h$-contour $\gamma'$ containing it.
\end{definition}
With this notation we have
\begin{lemma}\label{prop:4-beta-L} Let $h=H(L)-1$ and $\d>0$. If $\beta$ is sufficiently large then the collection $\{\gamma_i\}$ of external $h$-contours satisfies
\begin{align}
      \label{eq:18}
\pi_\Lambda^{0}\bigg(\sum_i|\g_i|\le (1+\d)4L\bigg)&\ge 1-e^{-\beta \delta L/2}\,.
\end{align}
\end{lemma}
\begin{proof}
 Let $A=\cup \Lambda_{\gamma_i}$ and let $R = \sum_i |\gamma_i|$. Let $U_A:\Omega\to\Omega$ denote the map that increases each $v\notin A$ by $1$ (retaining the remaining configuration as is), we see that $U_A$ increases the Hamiltonian by at most $|\partial \Lambda| - R$ and so
\[
\pi^0_\Lambda(U_A \eta) \geq \exp\left(-4\beta L + \beta R\right)\pi^0_\Lambda(\eta)\,.
\]
Since $U_A$ is bijective we get that the probability of having a given configuration of external contours $\{\g_i\}$ is bounded by $e^{-\beta( R - 4L)}$. Given $R=\ell$, the number of possible external contours is at most $\ell/\log(L)^2$, and the number of their arrangements is easily bounded from above by $C^\ell$ for some  constant $C>0$, for $L$ large enough.
Therefore, if we sum over configurations for which $R \geq (1+\delta)4L$ we have
\begin{align*}
\pi_\Lambda^0(R \geq (1+\delta)4L) &\leq \sum_{\ell \geq (1+\delta)4L}
C^\ell \,
e^{-\beta( \ell - 4L)}\leq e^{-\beta \delta L /2}
\end{align*}
for large enough $\beta$.
\end{proof}

The next ingredient in the proof of  Proposition~\ref{prop:(H-1)-starting-point} is to establish that most of the sites have height at least $H(L)-1$ with high probability.

\begin{lemma}\label{lem:total-mass-below-H(L)}
Let $\Lambda$ be the square of side-length $L$.
For any $\delta>0$ there exists some constants $C_1,C_2>0$ such that for any $\beta\geq C_1$,
\[ \pi_\Lambda^0\left( \#\{v : \eta_v \leq H(L)-2\} > \delta L^2 \right) \leq \exp(- C_2 L) \,.\]
\end{lemma}
\begin{proof}
 Let $\cS_h(\eta)=\{v\in \Lambda:\eta_v=h\}$ for $h=H(L)-k$.  Define $U_A:\Omega\to\Omega$ for each $A\subseteq \cS_h(\eta)$ as
\[
(U_A \eta)_v = \begin{cases}
\eta_v+1 &v \not\in A\\
0 & v\in A.
\end{cases}
\]
Since $U_A$ is equivalent to increasing each height by $1$ followed by decreasing the sites in $A$ by $h+1$,  the Hamiltonian is increased by at most $|\partial \Lambda| + 4(h+1)|A|$ and so
\[
\pi^0_\Lambda(U_A \eta) \geq \exp\left(-4\beta L - 4\beta(h+1)|A|\right)\pi^0_\Lambda(\eta)\,.
\]
Therefore,
\begin{align*}
\sum_{A\subseteq \cS_h(\eta)} \pi^0_\Lambda(U_A \eta) &\geq \exp(-4\beta L) \left(1+e^{-4\beta(h+1)}\right)^{|\cS_h(\eta)|}\pi^0_\Lambda(\eta),\\
&\geq \exp\left(-4\beta L+(1-o(1)) e^{-4\beta(h+1)}|\cS_h(\eta)|\right)\pi^0_\Lambda(\eta)\,,
\end{align*}
as $1+x \geq e^{x/(1+x)}$ for $x\geq 0$ and here the factor $1/(1+x)$ is $1-O(e^{-4\beta(h+1)})=1-o(1)$ since $h$ diverges with $L$ (namely, $h \asymp \log L$ by the assumption on $k$).
By definition $U_A \eta \neq U_{A'} \eta$ for any $A \neq A'$ with $A,A'\subseteq \cS_h(\eta)$. In addition, if $A\subseteq \cS_h(\eta)$ and $A'\subseteq \cS_h(\eta')$ for some $\eta\neq \eta'$ then $U_A \eta \neq U_{A'} \eta'$ (one can read the set $A$ from $U_A \eta$ by looking at the sites at level 0, and then proceed to reconstruct $\eta$).
Using the fact that $e^{-4\beta(h+1)} \geq e^{4\beta(k-1)}/L$ we see that
\begin{align*}
1 &\geq \sum_{\eta\,:\, |\cS_h(\eta)|\geq \delta e^{-2\beta (k-1)} L^2} \sum_{A\subseteq \cS_h(\eta)} \pi^0_\Lambda(U_A \eta)\\
&\geq \exp\left(-4\beta L+ (\delta-o(1)) e^{ 2\beta(k-1)} L \right)\pi^0_\Lambda(|\cS_h(\eta)|\geq \delta e^{-2\beta (k-1)} L^2),
\end{align*}
and so, for $k\geq 1$
\begin{align*}
\pi^0_\Lambda(|\cS_h(\eta)|\geq \delta e^{-2\beta (k-1)} L^2) &\leq  \exp\left(4\beta L - (\delta-o(1)) e^{2\beta(k-1)} L \right)\,.
\end{align*}
Summing over $k\geq 2$ establishes the required estimate for any sufficiently large $\beta$.
\end{proof}

We now complete the proof of Proposition~\ref{prop:(H-1)-starting-point}.
Fix $0<\delta\ll 1$. By Lemma~\ref{lem:total-mass-below-H(L)}, the number of sites with height less than $H(L)-1$ is at most $\delta L^2$. Condition on the external macroscopic $(H(L)-1)$-contours $\{\gamma_i\}$ and consider the region obtained by deleting those contours as well as their interiors and immediate external neighborhood, i.e., $V = \Lambda \setminus  \bigcup_i(\Lambda_{\gamma_i} \cup \Delta^-_{\gamma_i})$. An application of Lemma~\ref{lem:mass-at-H-from-small-contours} to $\pi^\xi_V$ where $\xi$ is the boundary condition induced by $\partial\Lambda$ and $\{\gamma_i\}$ (in particular at most $H(L)-1$ everywhere) shows that w.h.p.\ there are at most $\delta L^2$ sites of height larger than $H(L)-1$ in $V$. Altogether,
\[ \sum_i |\L_{\gamma_i}| \geq (1-2\delta)L^2\,,\]
and therefore, by an application of Lemma~\ref{prop:4-beta-L} followed by Lemma~\ref{le:isop}, we can conclude that w.h.p.\
one of the $\g_i$ contains a square with side-length at least $\frac9{10}L$ as required.
\end{proof}

\subsection{Absence of macroscopic $H(L)$-contours when $\lambda <
  \lambda_c$}
\label{sec:l<l}

In this section we prove:
\begin{proposition}
\label{prop:scifondo}
Fix $\delta>0$ 
and assume that
$\lambda<\lambda_c-\delta$. W.h.p., there
are no macroscopic $H(L)$-contours.
\end{proposition}
\begin{proof}
The strategy of the proof is the following:
\begin{itemize}
\item Step 1: via a simple isoperimetric argument, we show that if a
  macroscopic $H(L)$ contour exists, then it must contain a square of
  area almost $L^2$;
\item Step 2: using the ``domain-enlarging procedure'' (see Lemma~\ref{rem:mallargo}) we reduce the proof of the non-existence of
  macroscopic $H(L)$-contour as in Step 1 to the proof of the same fact in a larger square $\L'$
  of size $5L$ with boundary conditions $H(L)-1$. That allows us to
  avoid any pinning issues with the boundary of the original square
  $\L$. Using Proposition~\ref{prop:grigliatona} we write precisely the
  law of such a contour (assuming it exists) and we
  show that it satisfies a certain ``regularity property'' w.h.p.;
\item Step 3: using the exact form of the law of the macroscopic
  $H(L)$-contour in $\L'$ we are able to bring in the functional $\cF_\lambda$
  defined in Section~\ref{sec:Notation} and to show, via a precise area vs. surface tension
  comparison, that the probability that an $H(L)$-contour contains
  such square
is exponentially (in $L$) unlikely. This implies that no macroscopic
$H(L)$ contour exists and Proposition~\ref{prop:scifondo} is proven.
\end{itemize}

For lightness of notation throughout this section we will write $h$ for $H(L)$.

\subsubsection*{Step 1.}
We apply Proposition~\ref{bdgma} with $V=\Lambda$ and
$j=0$. Noting that $e^{-4\beta h}\log (|\gamma|)=O(\log
L/L)$ and recalling Definition~\ref{eq-def-lambda} of $\lambda$, we
have
\begin{eqnarray}
  \label{eq:20}
\pi^0_\Lambda(\mathcal C_{\gamma,h})\le
\exp\left(-(\beta+o(1))|\gamma|+(1+o(1))\frac
\lambda{L(1-e^{-4\beta})} |\Lambda_\gamma|\right)
\end{eqnarray}
where $o(1)$ vanishes with $L$.
This has two easy consequences. From $|\Lambda_\gamma|\le L^2$ we 
see
that w.h.p.\ there are no $h$-contours with
$|\gamma|\ge a_1 L:=(1+\epsilon_\beta)L\lambda/\beta$.
 Here and in the following,
$\epsilon_\beta$ denotes
some positive constant (not necessarily the same at each occurrence) that vanishes for
$\beta\to\infty$ and does not depend on $\delta$.
From $|\Lambda_\gamma|\le
|\gamma|^2/16 $ (isoperimetry) together with standard Peierls counting
of contours we see that w.h.p.\ there are no $h$-contours with
\begin{eqnarray}
  \label{eq:a2}
(\log L)^2\le |\gamma|\le a_2L:=\frac{16}\lambda\beta
L(1-\epsilon_\beta).
\end{eqnarray}
If $
\lambda<4\beta (1-\epsilon_\beta)$
then $a_1<a_2$ and we have excluded the occurrence of $h$-contours
longer than $(\log L)^2$: Proposition~\ref{prop:scifondo} is proven.
The remaining case is
\begin{eqnarray}
  \label{eq:17}
 4\beta (1-\epsilon_\beta)\le
\lambda<\lambda_c-
\delta
\end{eqnarray}
and it remains to exclude $h$-contours with
\begin{eqnarray}
  \label{eq:19}
\frac{16}\lambda\beta
L(1-\epsilon_\beta)\le |\gamma|\le L(1+\epsilon_\beta)\frac\lambda\beta.
  \end{eqnarray}
Recall from Remark~\ref{rem:betalargo} that $\lambda_c/(4\beta)$ tends to $1$ for $\beta$ large so that
under condition~\eqref{eq:17} we have
that
$4\beta(1-\epsilon_\beta)\le \lambda\le
4\beta(1+\epsilon_\beta)$. Then, condition~\eqref{eq:19} implies
\[
4L(1- \epsilon_\beta)\le |\gamma|\le 4L(1+\epsilon_\beta).\]
For all such $\gamma$, Eq.~\eqref{eq:20} implies that $\mathcal
C_{\gamma,h}$ is extremely unlikely, unless $|\Lambda_\gamma|\ge
L^2(1-\epsilon_\beta)$. But, as in the proof of Lemma~\ref{le:isop}, a contour in $\Lambda$ that has perimeter at most
$4L(1+\epsilon_\beta)$ and area at least $L^2(1-\epsilon_\beta)$ necessarily contains a square of area
$(1-\epsilon_\beta)L^2$, for a different value of $\epsilon_\beta$.

\subsubsection*{Step 2.}
We are left with the task of proving
\begin{gather}
  \label{eq:dapro}
 \pi^0_\Lambda(A):=
 \pi^0_\Lambda(\,\exists h\text{-contour}  \text{
   containing a square $Q\subset \L$ with area }
 (1-\epsilon_\beta)L^2)
 \leq e^{-c (\log L)^2}.
\end{gather}
Observe that the event $A$ is increasing. We
apply Lemma~\ref{rem:mallargo}, with $V=\Lambda\cup\partial \L$, $\Lambda'$ a square
of side $5L$ and concentric to $\Lambda$ and boundary condition $h-1$, to write
$\pi^0_\Lambda(A)\le \pi^{h-1}_{\Lambda',V}(A)$.
From now on, for lightness of notation,
we write $\tilde\pi^{h-1}_{\L'}$ instead of
$\pi^{h-1}_{\Lambda',V}$

Let $\gamma$ denote a contour
enclosing
a square $Q\subset \L$ of area $(1-\epsilon_\beta)L^2$.
As in the proof of~\eqref{e:contourFloorBound},
\begin{eqnarray}
  \label{eq:21}
 \tilde \pi^{h-1}_{\L'}(\mathcal
  C_{\gamma,h})=e^{-\beta|\gamma|}\frac{Z_{\rm out}^{-,h-1}
    Z_{\rm in}^{+,h}}{\tilde Z^{h-1}_{\L'}}.
\end{eqnarray}
Here, $\tilde Z^{h-1}_{\Lambda'}$ is the partition function
corresponding to the Gibbs measure $\tilde\pi^{h-1}_{\L'}$. 
 In the partition functions
$Z^{-,h-1}_{\rm out}, Z^{+,h}_{\rm in}$, and $\tilde Z^{h-1}_{\L'}$,
it is implicit the floor constraint that imposes non-negative heights in
$\L\cup\partial\L$.

Now we can apply Proposition~\ref{prop:grigliatona} (see also Remark
\ref{rem:generalizzo}) to the two
partition functions in the numerator.
For $Z^{+,h}_{\rm in}$, we have $V=\L_\g$ (as usual $\L_\gamma$ is the interior of $\gamma$ and
$\L_\gamma^c=\L'\setminus
\L_\gamma$), $W=\L\cap\partial \L $ and $n=-1$ (recall that $\lambda^{(n)}=\lambda
e^{4\beta n}$ and that $\lambda$ is around $4\beta$ by~\eqref{eq:17}). Since
\[|\L_\g\cap \L|\le L^2\ll
\left(\frac{4\beta L}{\lambda e^{-4\beta}}\right)^2\approx
L^2 e^{8\beta},\]
condition~\eqref{eq:card} is satisfied
and, for some $a\in(0,1)$ we have\footnote{In principle we should have
$|\L_\g\cap(\L\cup\partial\L)|$ instead of $|\L_\g\cap\L|$, but since $|\partial \L|/L=O(1)$ the
difference can be absorbed into the error $O(L^a)$.}
\begin{gather}
  \label{eq:24}
Z^{+,h}_{\rm in}=\hat Z^{+,h}_{\rm in}
\exp\left[-\frac{c_\infty}L e^{4\beta\alpha(L)} e^{-4\beta} |\L_\g\cap\L|+O(L^a)\right].
  \end{gather}
To expand $Z^{-,h-1}_{\rm out}$ we apply the same argument on the region
$\L'\setminus \L_\g$. Since by assumption $\gamma$ contains a square $Q\subset \L$ with area
$ (1-\epsilon_\beta)L^2$, we have
\[
|\L\setminus \L_\g|\le \epsilon_\beta L^2\ll \left(\frac{4\beta}\l
  L\right)^2\approx L^2.
\]
Therefore,
\begin{gather}
  \label{eq:25}
Z^{-,h-1}_{\rm out}=\hat Z^{-,h-1}_{\rm out}
\exp\left[-\frac{c_\infty}L e^{4\beta \alpha(L)}|\L\setminus\L_\g|+O(L^a)\right].
  \end{gather}
As for the denominator
$\tilde Z^{h-1}_{\Lambda'}$, via~\eqref{grill-lwb} we get
\begin{gather}
  \label{eq:23}
  \tilde Z^{h-1}_{\Lambda'}\ge \hat Z^{h-1}_{\Lambda'}\exp\left[
-\frac{c_\infty}Le^{4\beta\alpha(L)}|\L|+O(L^a)\right].
\end{gather}
Putting together~\eqref{eq:24},~\eqref{eq:25} and \eqref{eq:23} and
recalling
that $\l=c_\infty e^{4\beta\a(L)}(1-e^{-4\b})$, we
get
\begin{gather}
  \label{eq:26}
   \tilde\pi^{h-1}_{\L'}(\mathcal
  C_{\gamma,h})=e^{-\beta|\gamma|}\frac{\hat Z^{-,h-1}_{\rm out}\hat Z^{+,h}_{\rm in}
}{\hat
  Z^{h-1}_{\Lambda'}}\exp\left[\frac{\l}L|\L\cap\L_\g|+O(L^a)\right]
.
\end{gather}
Finally, the partition functions $\hat Z^{-,h-1}_{\rm out},\hat
Z^{+,h}_{\rm in}$ and $\hat
  Z^{h-1}_{\Lambda'}$ can be expanded using Lemma
\ref{lem:dks}. The net result is that
\begin{gather}
 \label{eq:22}
  \frac{\hat Z^{-,h-1}_{\rm out}\hat Z^{+,h}_{\rm out}
}{\hat
  Z^{h-1}_{\Lambda'}}=\exp(\Psi_{\L'}(\gamma))
\end{gather}
where, for every $V\subset \bbZ^2$ and $\gamma$ contained in $V$,
\begin{gather}
  \label{eq:28}
  \Psi_V(\g)=-
\sumtwo{W\subset V}{W\cap \g\neq
  \emptyset}
  \varphi_0(W)+
\sumtwo{W\subset \L_{\g}}{W\cap \g\neq
  \emptyset}
  \varphi_{\D_\g^+}(W)+
\sumtwo{W\subset V\setminus\L_{\g} }{W\cap \g\neq
  \emptyset}
  \varphi_{\D_\g^-}(W)
\end{gather}
(see also \cite{CLMST}*{App. A.3}). Here the notation $\g\cap W\ne\emptyset$
means $W\cap(\Delta^+_{\g}\cup \Delta^-_{\g})\ne\emptyset$.

Altogether, we have obtained
\begin{gather}
  \label{eq:27}
   \tilde\pi^{h-1}_{\L'}(\mathcal
  C_{\gamma,h})=\exp\left[-\beta|\g|+\Psi_{\L'}(\g)+
\frac{\l}L|\L\cap\L_\g|+O(L^a)
\right].
\end{gather}
Let $\Sigma$ denote the collection of all possible contours that enclose
a square $Q\subset \L$ with area $(1-\epsilon_\beta)L^2$.

A first observation is that the event that there exists
an $h$-contour $\gamma\in\Sigma $ that has distance less than $(\log
L)^2$ from $\partial\L'$ (the boundary of the square of side $5L$) has
negligible probability. Indeed, such contours have necessarily
$|\g|\ge 5L$. Then, the area term $\lambda |\L\cap\L_\g|/L\le \lambda
L\approx 4\b L$ cannot compensate for $-\beta|\g|$, and from the
properties
of the potentials $\varphi$ in Lemma~\ref{lem:dks}, we see that
$|\Psi_{\L'}(\g)|\le \epsilon_\beta |\gamma|$.
As a consequence, we can safely replace
$\Psi_{\L'}(\g)$ with $\Psi_{\mathbb Z^2}(\g)$ in~\eqref{eq:27}: indeed,
thanks to Lemma~\ref{lem:dks} point (iii), one has $|
\Psi_{\L'}(\g)-\Psi_{\bbZ^2}(\g)|\le \exp(-(\log L)^2)$ if $\g$ has
distance at least $(\log L)^2$ from $\partial\L'$.

Secondly, we want to exclude contours with long ``button-holes''.
Choose $a'\in(a,1)$.  For any contour
$\g$ and any pair of bonds $b,b'\in \g$ we let $d_\g(b',b)$ denote the
number of bonds in $\G$ between $b$ and $b'$ (along the shortest of
the two portions of $\g$ connecting $b,b'$). Finally, we define the
set of ``contours with button-holes'' as the subset $\Sigma '\subset
\Sigma$ such that
there exist $b,b'\in\g$ with $d_\g(b,b')\ge L^{a'}$ and
$|x(b)-x(b')|\le (1/2)d_\g(b,b')$,
where $x(b),x(b')$ denote the centers of $b,b'$ and $|\cdot|$ is the $\ell^1$ distance.
The next result states that contours with button-holes are unlikely:
\begin{lemma}
For any $c>0$ and $\b$ large enough
\[
\tilde\pi^{h-1}_{\L'}(\exists\  \g\in\Sigma' \text{ such that } \mathcal C_{\g,h} \text{ holds })\le e^{-c\,L^{a'}}.
\]
\end{lemma}
\begin{proof}
The proof is based on standard arguments \cite{DKS}, so we will be
extremely concise.
 Suppose that $\g\in\Sigma'$:
  that implies the existence  of two bonds $b,b'\in \g$, with $d_\g(b,b')\ge L^{a'}$ and
$|x(b)-x(b')|\le (1/2)d_\g(b,b')$. One
can then short-cut the button-hole, to obtain a new contour $\g'$ that
is at least $(1/2)d_\g(b,b')\ge (1/2)L^{a'}$ shorter than $\g$ and at the same time
contains the same large square $Q\subset \L$ of area
$(1-\epsilon_\beta)L^2$. The basic observation is then that the area
variation satisfies
\[\left| |\L_\g\cap\L|-|\L_{\g'}\cap \L| \right |\le
\min\left(
d_\g(b,b')^2,\epsilon_\beta L^2
\right)
\]
so that
\[
-\beta|\g|+\Psi_{\bbZ^2}(\g)+\frac\l L|\L\cap\L_\g|\le
-\beta|\g'|+\Psi_{\bbZ^2}(\g')+\frac\l L|\L\cap\L_\g'|-(\beta/4)L^{a'}.
\]
At this point, Eq.~\eqref{eq:27}
together with routine Peierls arguments implies the claim (recall that
$a'>a$).
\end{proof}

The important property of contours without button-holes is that
the interaction between two portions of the contour is at most of
order $L^{a'}$:
\begin{claim}
  \label{clamobutto}
If $\gamma$ has no button-holes, then for every decomposition of
$\gamma$ into a concatenation of $\gamma=\gamma_1\circ
\gamma_2\circ \dots \circ\g_n $ we have\footnote{strictly speaking, in
\eqref{eq:28} we have defined $\Psi_\L(\g)$ for a closed contour. For
an open portion $\g'$ of a closed contour $ \g$, one can define for instance
\[
 \Psi_{\bbZ^2}(\g')=-
\sumtwo{W\subset \bbZ^2}{W\cap \g'\neq
  \emptyset}
  \varphi_0(W)+
\sumtwo{W\subset \L_{\g}}{W\cap \g'\neq
  \emptyset}
  \varphi_{\D_\g^+}(W)+
\sumtwo{W\subset \bbZ^2\setminus\L_{\g} }{W\cap \g'\neq
  \emptyset}
  \varphi_{\D_\g^-}(W).
\]
}
 $|\Psi_{\bbZ^2}(\g)-\sum_{i=1}^n
\Psi_{\bbZ^2}(\g_i)|\le n L^{a'}$.
\end{claim}
\begin{proof}
Just use the representation~\eqref{eq:28} and the decay properties
of the potentials $\varphi(\cdot)$, see Lemma~\ref{lem:dks} point (iii).
\end{proof}

\subsubsection*{Step 3.}
We are now in a position to conclude the proof of Proposition~\ref{prop:scifondo}.
Let $\cM$ denote the set of contours in $\L'$, of length at most
$5L$, that do not come too close to the boundary of $\L'$, that
include a square $Q\subset \L$ of side $(1-\epsilon_\beta)L^2$
and finally that have no buttonholes. In view of the previous discussion, it will be sufficient to upper bound the  $\tilde \pi^{h-1}_{\L'}$-probability of the event $\cup_{\g\in\cM}\sC_{\gamma,h}$.
Let $\mathcal{V}_s=\{\underline{v}=(v_1,\ldots,v_s,v_{s+1}=v_1): v_i \in
\Lambda'\}$ denote a sequence of points in $\Lambda'$.  We say that
$\gamma \in {\mathcal{M}}_{\underline{v}}$ if $\gamma \in
\mathcal{M}$, all the $v_i$ appear along $\gamma$ in that order, and
for each $i \geq 2$, $v_i$ is the first point $x$ on $\gamma$ after $v_{i-1}$ such that $|x-v_{i-1}| \geq \epsilon L$.  Note that since we are considering $|\gamma| \leq 5L$ we have that $s\leq 5/\epsilon$.
\begin{align*}
&\tilde\pi^{h-1}_{\L'}(\exists\  \gamma\in\mathcal{M}\,,\,  \sC_{\gamma,h})
 \leq \sum_{s=1}^{5/\epsilon} \sum_{\underline{v} \in \mathcal{V}_s} \sum_{\gamma \in {\mathcal{M}}_{\underline{v}}}  \tilde\pi^{h-1}_{\L'}\left( \sC_{\gamma,h}\right)\\
& \leq \sum_{s=1}^{5/\epsilon} \sum_{\underline{v} \in \mathcal{V}_s}
\sum_{\gamma \in {\mathcal{M}}_{\underline{v}}}
\exp\left(-\beta|\gamma| + \Psi_{\bbZ^2}(\gamma)+\frac{\lambda}L
    |\Lambda_\gamma\cap \Lambda|+O(L^a)\right),
\end{align*}
where we used
\eqref{eq:27} (with $\Psi_{\L'}$ replaced by $\Psi_{\mathbb Z^2}$).
Now let $K_{\underline{v}}$ denote the convex hull of the set of
points $\underline{v}$.  Since the contour $\gamma$ is never more than at
distance $\epsilon L$ from a point in $\cV$ (by definition of
$\cM_{\underline v}$) we have
\begin{align*}
|\Lambda_\gamma\cap\Lambda| &\leq |K_{\underline{v}}\cap\Lambda| + s \epsilon^2 L^2
\leq |K_{\underline{v}}\cap\Lambda| + 5 \epsilon L^2.
\end{align*}
Also, from Claim~\ref{clamobutto} we have, if $\g_{i,{i+1}}$
is the portion of $\g$ between $v_i$ and $v_{i+1}$,
\[
|\Psi_{\mathbb{Z}^2}(\gamma) - \sum_{i=1}^s
\Psi_{\mathbb{Z}^2}(\gamma_{{i},{i+1}})| \leq  s L^{a'}.
\]
Now note that, by  standard estimates of~\cite{DKS}
\begin{align*}
\sum_{\gamma \in {\mathcal{M}}_{\underline{v}}}  \exp(-\beta|\gamma| +
\Psi_{\bbZ^2}(\gamma)) &\leq e^{O( L^{a'})}\prod_{i=1}^s \sum_{\gamma_{i,i+1}} e^{-\beta |\gamma_{i,i+1} |+ \Psi_{\mathbb{Z}^2}(\gamma_{{i},{i+1}})}\\
&\leq e^{O(L^{a'})} \prod_{i=1}^s \exp\left(-(\beta+o(1))\tau(v_{i+1}-v_i)\right)\\
&= \exp\Bigl(-(\beta+o(1))\int_{\gamma_{\underline{v}}}\tau(\theta_s)ds + O(L^{a'})\Bigr)
\end{align*}
with $o(1)$ vanishing as $L\to\infty$,  the sum is over all
contours $\gamma_{i,{i+1}}$ from $v_i$ to $v_{i+1}$,
$\gamma_{\underline{v}}$ denotes the piecewise linear curve joining
$v_1,v_2,\dots,v_1$ and we applied Appendix~\ref{app:ce} to reconstruct
the surface tension $\tau(v_{i+1}-v_i)$ from the
sum over $\g_{i,i+1}$ (cf.~Definition~\ref{deftau}).

By convexity of the surface tension,
\[
\int_{\gamma_{\underline{v}}}\tau(\theta_s)ds \geq \int_{\partial
  {K_{\underline{v}}}}\tau(\theta_s)ds\ge
\int_{\partial[
  {K_{\underline{v}}}\cap \L]}\tau(\theta_s)ds
\]
and so combining the above inequalities we get
\begin{align*}
\sum_{\gamma \in {\mathcal{M}}_{\underline{v}}}  \exp\left(-\beta|\gamma| + \Psi_{\bbZ^2}(\gamma)+\frac{\lambda}L |\Lambda_\gamma\cap\Lambda|\right)
&\leq \exp\left( - \beta \int_{\partial[K_{\underline{v}}\cap \L]}\tau(\theta_s)ds + \frac{\lambda}{L} |K_{\underline{v}}\cap\Lambda| + c\,\epsilon L \right)\,,
\end{align*}
for some constant $c>0$ and $L$ large enough.
After rescaling $K_{\underline{v}}\cap\L$ to the unit square we have a
shape with area at least $(1-\epsilon_\beta)$.  Since $\lambda < \lambda_c$ we
have from Claim~\ref{clamotto}
that for all curves $\gamma^\dagger$ in $[0,1]^2$ enclosing such an area
\[
\mathcal{F}_\lambda(\gamma^\dagger)=- \beta\int_{\gamma^\dagger}\tau(\theta_s)ds + \lambda A(\gamma^\dagger) \leq - \theta(\lambda)<0.
\]
Hence we have that
\begin{align*}
&\sum_{\gamma \in {\mathcal{M}}_{\underline{v}}}  \exp\left(-\beta|\gamma| + \Psi_{\bbZ^2}(\gamma)+\frac{\lambda}L |\Lambda_\gamma\cap\Lambda|\right) \leq \exp(-\theta L /2)
\end{align*}
provided $\epsilon>0$ is sufficiently small and $L$ is sufficiently large.  Now since $s\leq 5/\epsilon$ we have that $|\mathcal{V}_s| \leq |\Lambda'|^{5/\epsilon}$ and so
\begin{align*}
\widetilde{\pi}_{\Lambda'}^{h-1}(\exists \gamma\in\mathcal{M}\,,\,  \sC_{\gamma,h}) &\leq (5/\epsilon) |\Lambda'|^{5/\epsilon} \exp(-\theta L /2) \leq c_1 e^{-c_2 \log^2 L}\,,
\end{align*}
which completes the proof of Proposition~\ref{prop:scifondo}.
\end{proof}

\subsection{Existence of a macroscopic $H(L)$-contour when $\lambda >
\lambda_c$}
\label{sec:l>l}

In the special case where $\lambda \geq (1+a)\lambda_c$ for some
arbitrarily small absolute constant $a>0$ (say, $a=0.01$),
one can prove the existence of a macroscopic $H(L)$-contour by
following (with some more care) the same line of arguments used to
establish a supercritical $H(L)-1$ droplet in
Section~\ref{sec:H(L)-1-droplet}.  To deal with the more delicate case
where $\lambda$ is arbitrarily close to $\lambda_c$ we provide the
following proposition.

\begin{proposition}\label{p:bigHContour}
  Let  $\beta$ be sufficiently large. For any  $\delta>0$ there exist constants $c_1,c_2$ such
  that if $(1+\delta) \lambda_c \le \lambda \le \lambda_c(1+a) $ then
\[
\pi^0_{\Lambda}\left(\exists \gamma: \sC_{\gamma,H(L)}\,,\, |\L_\g|\ge (9/10) L^2 \right) \geq 1 -c_1 e^{-c_2 \log^2 L}\,.
\]

\end{proposition}
We emphasize that the difference between the parameters $a$ and $\delta$ is
that $a$ is small but fixed, while $\delta$ can be arbitrarily small with $\beta$.
\begin{proof}[Proof of Proposition~\ref{p:bigHContour}]
  First of all, from~\eqref{eq:20} and~\eqref{eq:a2} we see that, if
$(1+\delta) \lambda_c \le \lambda \le \lambda_c(1+a) $
(and recalling that $\lambda_c/\beta\sim 4$), w.h.p.\ there are no $H(L)$-contours of length at least $(\log L)^2$ and area at most $(9/10)L^2$.
Let $\mathcal{S}_0$ denote the event that there does not
exist an $H(L)$-contour $\gamma$ of area larger than $(9/10)L^2$.
Thus, on  the event $\mathcal{S}_0$ w.h.p.\ the largest $H(L)$-contour has length at most $\log^2 L$.

By Proposition~\ref{prop:(H-1)-starting-point}  and
 Theorem~\ref{th:growth0}
we have that w.h.p.\  for any $\epsilon>0$
there exists an external $(H(L)-1)$-contour $\Gamma$ containing
 $(1-\epsilon)L\mathcal{L}_c(\lambda)$.
We condition on this
$\Gamma$.   Thus, by monotonicity,
\[
\pi^0_{\Lambda}(\mathcal{S}_0\mid \Gamma) \leq \pi^{H(L)-1}_{\Lambda_\Gamma}(\mathcal{S}_0)
\]
so it suffices to work under $\pi^{H(L)-1}_{\Lambda_\Gamma}$.
Let $\mathcal{S}$ denote the event that there are no macroscopic contours (of any height, positive or negative).
Observe that w.r.t.\ $\pi^{H(L)-1}_{\Lambda_\Gamma}$ w.h.p.\
there are no macroscopic contours on the event $\cS_0$ (cf.\ e.g.\ the proof of Lemma~\ref{prop:nocontornoni}). Thus, $\pi^{H(L)-1}_{\Lambda_\Gamma}(\mathcal{S}^c\cap \cS_0)$ is negligible and it suffices to upper bound the probability $\pi^{H(L)-1}_{\Lambda_\Gamma}(\mathcal{S})$.

To this end, we compare $\pi^{H(L)-1}_{\Lambda_\Gamma}(\mathcal{S})$ with
the probability of a specific contour $\gamma$ approximating the
optimal curve and then sum over the choices of $\gamma$.  Let
\begin{align}
\label{cicappa}
  \cK=\partial \left[(L(1-\epsilon)-L^{3/4})\mathcal{L}_c(\lambda)\right]
\end{align}
 be
the suitably dilated solution to the variational problem of maximizing
$\mathcal{F}_\lambda$.  By our choice of dilation factor, $\cK$ is
at distance at least $L^{3/4}$ from $\Gamma$.  Then for some $s$
growing slowly to infinity with $L$ let $v_1,\ldots,v_s$ be a sequence
of vertices in clockwise order along $\cK$ with $3L/s\leq
|v_i-v_{i+1}| \leq 5L/s$ for $1\leq i \leq s$ where $v_{s+1}=v_1$.

Let $W$ be the bounded region delimited
by the two curves
\[
x\mapsto \xi^\pm(x):=\pm \left(x(1-x)
\right)^{3/5}, x\in [0,1].
\]
We define the cigar shaped region $W_{i}$ between points $v_i$ and $v_{i+1}$ as in~\cite{MT}*{Section~1.4.6}
to be given by $W$ modulo a translation/rotation/dilation
that brings $(0,0)$ to $v_i$ and $(1,0)$ to $v_{i+1}$.  Now let
$\gamma = \gamma_{1}\circ\ldots\circ\gamma_s$ be a closed contour
where each $\gamma_i$ is a curve from $v_i$ to $v_{i+1}$ inside the
region $W_{i}$.  Note  that  by construction $|\Lambda_\gamma| \geq
|\Lambda_{\cK}| - s(5L/s)^2 = |\Lambda_{\cK}| - o(L^2)$ and $\gamma$
is at least at distance $\frac12 L^{3/4}$ from $\partial
\Lambda_\Gamma$.

In analogy with~\eqref{eq:21} we have that
\begin{align*}
\frac{\pi^{H(L)-1}_{\Lambda_\Gamma}( \sC_{\gamma,H(L)})}{\pi^{H(L)-1}_{\Lambda_\Gamma}(\mathcal{S})}
= e^{-\beta|\gamma|}\frac{{Z}^{+,H(L)}_{\rm in} {Z}^{-,H(L)-1}_{\rm out}}{{Z}^{H(L)-1}_{\L_\G}(\cS)}
\end{align*}
where: ${Z}^{+,H(L)}_{\rm in}$ (resp. ${Z}^{-,H(L)-1}_{\rm out}$) is the partition function in $\L_\g$
(resp. $\L_\G\setminus \L_\g$) with floor at zero,
b.c.\ $H(L)$ (resp. $H(L)-1$) and constraint $\eta\ge H(L)$ in $\Delta^+\g$
(resp. $\eta\le H(L)-1$ on $\Delta^-_\g$); ${Z}^{H(L)-1}_{\L_\G}(\cS)$ is the
partition function in $\L_\G$, b.c.\ $H(L)-1$, floor at zero and constraint $\eta\in\cS$.
As in Section~\ref{sec:l<l}, one can apply Proposition~\ref{prop:grigliatona} to the numerator to get
\begin{align*}
  {Z}^{+,H(L)}_{\rm in} {Z}^{-,H(L)-1}_{\rm out}=
\hat {Z}^{+,H(L)}_{\rm in} \hat {Z}^{-,H(L)-1}_{\rm out}
\exp\left[-\frac{c_\infty}Le^{4\b \a(L)}\left(
e^{-4\b}|\L_\g|+|\L_\G\setminus\L_\g|
\right)
+o(L)
\right]
\end{align*}
where the partition functions with the ``hat'' have no floor.
As for the denominator,
\begin{align*}
  {Z}^{H(L)-1}_{\L_\G}(\cS)\le \hat {Z}^{H(L)-1}_{\L_\G}
\hat\pi^{H(L)-1}_{\L_\G}(\eta\restriction_{\L_\G}\ge 0|\cS) \leq
\hat {Z}^{H(L)-1}_{\L_\G}\exp\left[
-\frac{c_\infty}Le^{4\b \a(L)}|\L_\G|
+o(L)
\right]
\end{align*}
where we applied~\eqref{eq:condizionata} in the second step. Together with
\eqref{eq:22}, the Definition~\ref{lambdan} of $\l$ and the fact that
$|\L_\g|\ge |\L_{\cK}|-o(L^2)$, this yields
\begin{align*}
\frac{\pi^{H(L)-1}_{\Lambda_\Gamma}( \sC_{\gamma,H(L)})}{\pi^{H(L)-1}_{\Lambda_\Gamma}(\mathcal{S})}&\ge \exp\left(-\beta|\gamma| +\Psi_{\mathbb{Z}^2}(\gamma)+\frac\l L
|\L_\cK|+o(L)\right),
\end{align*}
where we replaced $\Psi_{\L_\G}(\g)$ with $\Psi_{\bbZ^2}(\g)$, cf.\ the discussion after~\eqref{eq:27}, since by construction
$\g$ stays at distance at least $(1/2)L^{3/4}$ from $\partial \L_\G$.

At this point we can sum over $\g$, with the constraint that each portion
$\g_{i,i+1}$ from $v_i$ to $v_{i+1}$ is in $W_i$ as specified before.
Since the cigar $W_i$ is close to $W_{i\pm1}$ only
at its tips  we have, from the decay properties of the potentials $\varphi$
that define $\Psi$, that \cite{DKS}
\[
|\Psi_{\mathbb{Z}^2}(\gamma) - \sum_{i=1}^s \Psi_{\mathbb{Z}^2}(\gamma_{i})| =O(s)\,.
\]
Also, by  Appendix~\ref{app:ce} 
\begin{align*}
\sum_{\gamma_{i}\in W_i} e^{-\beta |\gamma_{i,i+1}| + \Psi_{\mathbb{Z}^2}(\gamma_{i,i+1})}
&=  \exp(-(\beta+o(1))\tau(v_{i+1}-v_i))\,.
\end{align*}
Summing over all such contours we then have that
\begin{align*}
\frac{\sum_\gamma \pi^{H(L)-1}_{\Lambda_\Gamma}( \sC_{\gamma,H(L)})}{\pi^{H(L)-1}_{\Lambda_\Gamma}(\mathcal{S})}
&\qquad= e^{\frac{\lambda}{L}|\Lambda_{\cK}|+o(L)} \prod_{i=1}^s \sum_{\gamma_{i,i+1}}   \exp(-\beta|\gamma_{{i},{i+1}}| + \Psi_{\mathbb{Z}^2}(\gamma_{v_{i},v_{i+1}}))\\
&\qquad= e^{\frac{\lambda}{L}|\Lambda_{\cK}|+o(L)} \prod_{i=1}^s   \exp(-\beta\tau(v_{i}-v_{i+1}))\\
&\qquad= \exp\left( -\beta \int_{\partial{\cK}}\tau(\theta_s)ds + \frac{\lambda}{L}|\Lambda_{\cK}| +o(L) \right)\\
&\qquad
= \exp\left( L\mathcal{F}_\lambda(L^{-1} \cK) +o(L) \right),
\end{align*}
with $\cF_\l(\cdot)$ the functional in~\eqref{funct}.  Since $L^{-1}
\cK$ is a close approximation to $\mathcal{L}_c(\lambda)$ by Claim
\ref{clamotto} (ii) it follows that $\mathcal{F}_\lambda(L^{-1} \cK)>0$.  Hence
$\pi^{H(L)-1}_{\Lambda_\Gamma}(\mathcal{S}) \leq e^{-c L}$, which
concludes the proof.
\end{proof}

\subsection{Conclusion: Proof of Theorem~\ref{mainthm-1}}
Assume that $\l(L_k)$ has a limit (otherwise it is sufficient to work on converging sub-sequences).
The results established thus far yield that w.h.p.:
\begin{compactitem}[\indent$\bullet$]
  \item By 
Proposition~\ref{prop:(H-1)-starting-point}
 there exists an
$(H(L_k)-1)$-contour whose area is at least $(9/10) L_k^2$.
\item By
Corollary~\ref{cor:H+1-contours} there are no macroscopic
$(H(L_k)+1)$-contours.
\item When $\lim_{k\to\infty} \lambda(L_k)<\lambda_c$, by
Proposition~\ref{prop:scifondo} there is no macroscopic $H(L_k)$-contour.
\item When $\lim_{k\to\infty}
\lambda(L_k) > \lambda_c$, by 
Proposition~\ref{p:bigHContour}
there
exists an $H(L_k)$-contour whose area is at least $(9/10) L_k^2$.
\end{compactitem}
Combining these statements
with~\eqref{eq:mass-at-H-from-small-contours}
and~\eqref{eq:mass-at-H-from-small-neg-contours} completes the proof when
$\lim_{k\to\infty} \lambda(L_k)\ne\lambda_c$. Whenever $\lambda(L_k)\to\lambda_c$
we want to prove that $\pi^0_{\L_k}(E_{H(L_k)-1}\cup E_{H(L_k)})\to1$, i.e., we want to exclude, say, that half of the sites have height $H(L_k)$ and the other half
have height $H(L_k)-1$. This is a simple consequence of~\eqref{eq:20} and~\eqref{eq:a2}  that say that, when $\l\approx 4\beta$, either there are no macroscopic
$H(L_k)$-contours, or there exists one of area $(1-\epsilon_\beta)L^2$. The proof is then concluded also for $\lim_{k\to\infty} \lambda(L_k)=\lambda_c$, invoking again~\eqref{eq:mass-at-H-from-small-contours}
and~\eqref{eq:mass-at-H-from-small-neg-contours}.
\qed

\subsection{Proof of Corollary~\ref{maincor-maximum}}
Consider first the case with no floor.
With a union bound the probability that $\widehat{X}^*_L \ge \varphi(L)+\frac1{2\b}\log L$
can be bounded by $L^2\hat\pi_\L^0(\eta_x\geq \varphi(L)+\frac1{2\b}\log L)$, which is
$O(e^{-4\b\varphi(L)})$ by Lemma~\ref{pbeta}. For the other direction,
let $\cA$ denote the set of $x\in\L$ belonging to the even sub-lattice of $\bbZ^2$ and such that $\eta_y=0$ for all neighbors $y$  of $x$. Then, by conditioning on $\cA$, and using the Markov property, one finds that
the probability of $\widehat{X}^*_L \le - \varphi(L)+\frac1{2\b}\log L$ is bounded by
the expected value $\hat\pi_\L^0(\exp(-e^{4\b\varphi(L)}|\cA|/L^2))$. Using Chebyshev's bound and the
exponential decay of correlations \cite{BW} it is easily established that the event $|\cA|<\d L^2$ has vanishing $\hat\pi_\L^0$-probability as $L\to\infty$ if $\d$ is small enough. Since $\varphi(L)\to\infty$ as $L\to\infty$, this ends the proof of~\eqref{heightvsmax}.
%
%

For the proof of~\eqref{maxvsmax} we proceed as follows.
Consider the $\pi_\L^0$-probability  that ${X}^*_L \le \frac3{4\b}\log L -\varphi(L)$.
Condition on the largest  $(H(L)-1)$-contour $\gamma$, which contains a square of side-length $\frac{9}{10}L$  w.h.p.\ thanks to Proposition~\ref{prop:(H-1)-starting-point}.
By monotonicity we may remove the floor and fix the height of the internal boundary condition on $\L_\gamma$ to $H(L)-1$. At this point the argument given above for the proof of~\eqref{heightvsmax}
yields that
\[
\pi_\L^0\left({X}^*_L \le \frac3{4\b}\log L -\varphi(L)\right)=o(1),
\]
since $H(L)+\frac1{2\b}\log L=\frac3{4\b}\log L+O(1)$.
To show that $\pi_\L^0\left({X}^*_L \ge \frac3{4\b}\log L +\varphi(L)\right)=o(1)$, recall that w.h.p.\ there are no macroscopic $(H(L)+1)$-contours thanks to Corollary~\ref{cor:H+1-contours}. Condition therefore on $\{\gamma_i\}$, all the external microscopic $(H(L)+1)$-contours. The area term in~\eqref{e:contourFloorBound} is negligible for these, thus it suffices to treat each $\gamma_i$ without a floor and with an external boundary height $H(L)$. The probability that a given $x\in \Lambda_{\gamma_i}$ sees an additional height increase of $k$ is then at most $c e^{-4\beta k}$ and a union bound completes the proof.
\qed

\section{Local shape of macroscopic contours}
\label{gadget}

In this section we establish the following result. Given $n\in
\bbZ_+$, consider the
SOS model in a domain of linear size
$\ell=L^{\frac 23 +\epsilon}$, with floor at zero and Dobrushin's
boundary conditions around it at height $\{j-1, j\}$, where
$j=H(L)-n$. We show that the entropic repulsion from the floor forces the
unique open $j$-contour to have height
$(1+o(1))c(j,\theta)\ell^{1/2+3\epsilon/2}$ above the straight
line $\bbL$
joining its end points. The constant $c(j,\theta)$ is explicitly
determined  in terms of the contour index $j$ and of the surface
tension computed at the angle $\theta$
describing the tilting of $\bbL$
w.r.t.\ the coordinate axes. This result will be the key element in proving the
scaling limit for the level lines as well as the $L^{\frac 13}$-fluctuations
around the limit.

\subsection{Preliminaries}
\label{chains}
We call
\emph{domino} any rectangle in $\bbZ^2$ of short
and long sides $(\log L)^2$ and $2(\log L )^2$ respectively.
A subset
$\cC=\{x_1,x_2,\dots,x_k\}$ of the domino
will be called a \emph{spanning chain} if
\begin{enumerate}[(i)]
\item $x_i\neq x_j$ if $i\neq j$;
\item ${\rm dist}(x_i,x_{i+1})=1$ for all $i=1,\dots k-1$;
\item $\cC$ connects the two opposite short sides of the domino.
\end{enumerate}
\begin{remark}
\label{remA1}Let $R_k$ be a rectangle with short side $2(\log L )^2$
and long side $k (\log L )^2$, $k\in \bbN$, $k\geq 2$.
Consider a covering of $R_k$ with horizontal dominos and a covering of $R_k$ with vertical dominos,
and fix a choice of a
spanning chain for each domino in these coverings. By joining these spanning chains, one obtains a chain $\wt \cC=\{y_1,y_2,\dots,y_n\}\subset R_k$ connecting the
opposite short sides of $R_k$. A chain constructed this way will
be called a \emph{regular chain}.
\end{remark}
Given $j\ge 0$ consider now the SOS model in a subset $V$ of the $L\times L$ box $\L$ with
$j$-boundary conditions and floor at height $0$.
\begin{definition}\label{def:dominos}
Given a SOS-configuration $\eta$, a domino entirely contained in $V$ will be called
of \emph{positive type},  if there exists a spanning chain $\cC$ inside it such that
$\eta_{x}\ge j$ for all $x\in \cC$. Similarly, if there exists a
spanning chain $\cC$ such that $\eta_{x}\le j$ for all $x\in \cC$, then the domino will be said to be of \emph{negative type}.
\end{definition}
\begin{lemma}
\label{quasi-rect.1}
W.h.p.\ all
dominos in $V$ are of positive type.
\end{lemma}
\begin{proof}
A given domino is not of positive type iff there exists a *-chain
$\{y_1,\dots,y_n\}$ connecting the two long opposite sides and such
that $\eta_{y_i}<j$ for all $i$.
Such
an event is decreasing and therefore its probability is bounded from
above by the probability w.r.t.\ the SOS model \emph{without}
the floor. Moreover, the above event
implies the existence of a $(j-1)$-contour larger than $(\log L)^2$. The
standard cluster expansion shows that the probability of the latter is
$O(e^{-c (\log L)^2})$. A union bound over all possible choices
of the domino completes the proof.
\end{proof}
Under the assumption that $|V|$ is not too large depending on $j$, we can
also show that all dominos are of negative type. Recall the definition
of $c_\infty$ and $\d_j$ from Lemma~\ref{pbeta}.
\begin{lemma}
\label{quasi-rect.2}
In the same setting of Lemma~\ref{quasi-rect.1} with $j=H(L)-n$ for fixed $n$, assume
$|V|^{\frac 12}\le 2e^{4\b (j+1)}\left[(1+\d_j)c_\infty\right]^{-1}$. Then w.h.p.\ all dominos  in $V$ are of negative type.
\end{lemma}
\begin{proof}
It follows
from Lemma~\ref{bdgma} that
\[
\pi_V^j\left(\sC_{\g,j+1}\right) \le e^{-\b|\g| + (1+\d_j)
  \,c_\infty e^{-4\beta (j+1)}|\L_\g|}\ e^{\varepsilon_\b e^{-4\b
    (j+1)}|\g| \log(|\g|)},
\]
with $\lim_{\b\to
  \infty}\varepsilon_\b=0$. Clearly $|\L_\g|\le
|V|^{1/2}|\g|/4$ and $|\g|\le |V|$. Hence,
\begin{align*}
(1+\d_j) \,c_\infty e^{-4\beta (j+1)}|\L_\g|&\le \b |\g|/2
\\
\varepsilon_\b e^{-4\b (j+1)}|\g|
\log(|\g|)&\ll \b |\g|
\end{align*}
and a standard Peierls
bound proves that w.h.p.\ there are no macroscopic
$(j+1)$-contours. Hence w.h.p.\ all dominos are of negative type.
\end{proof}

\subsection{Main Result}
\begin{definition}[Regular circuit $\cC_*$]
\label{reg-circuit}Let $0<\epsilon\ll1$ and let $Q$ (resp. $\tilde Q$)
be the rectangle of  horizontal side
$L^{\frac 23 +\epsilon}$ and vertical side $2L^{\frac 23 +\epsilon}$
(resp. $L^{\frac 23 +\epsilon}+4(\log L)^2$ and $2L^{\frac 23 +\epsilon}+4(\log L)^2$)
centered at the origin. Write $\tilde Q\setminus Q$ as the union of
four 
thin rectangles
(two vertical and two horizontal) of shorter side $2(\log L)^2$ and
pick a regular chain
for each one of them as in Remark \ref{remA1}. Consider the shortest (self-avoiding) circuit $\cC_*$ surrounding the
square $Q$, contained in the union of the four chains. We call $\cC_*$ a \emph{regular circuit}.
\end{definition}

\begin{definition}[Boundary conditions on $\cC_*$]
\label{bcc} Given a regular circuit $\cC_*$ and integers $a,b,j$ with
$j>0$ and $-L^{\frac 23 +\epsilon}\le a\le b\le
L^{\frac 23 +\epsilon}$,
we define a height configuration $\xi=\xi(\cC_*,j,a,b)$ on $\cC_*$ as
follows. Choose a  point $P$ in $\cC_*$, with zero horizontal coordinate and
positive vertical coordinate. Follow the circuit anti-clockwise (resp. clockwise) until you hit for
the first time the vertical coordinate $a$ (resp. $b$), and call $A$
(resp. $B$) the
corresponding point of $\cC_*$. Set $\xi_x=j-1$ on the portion of  $\cC_*$
between $A$ and $B$ that includes $P$, and set $\xi_x=j$ on
the rest of the circuit; see Figure \ref{fig:ghiro}. It is easy to check that, given the
regularity properties (cf.~Remark~\ref{remA1}) of the four chains composing the circuit, the
construction of $\xi$ is independent of the choice of $P$ as
above.
\end{definition}
\begin{figure}[h]
\psfrag{j}{$j-1$}
\psfrag{j-1}{$j$}
\psfrag{A}{$A$}\psfrag{B}{$B$}
\psfrag{Q}{$Q$}\psfrag{Qt}{$\tilde Q$}
\psfrag{ga}{$\G$}
\psfig{file=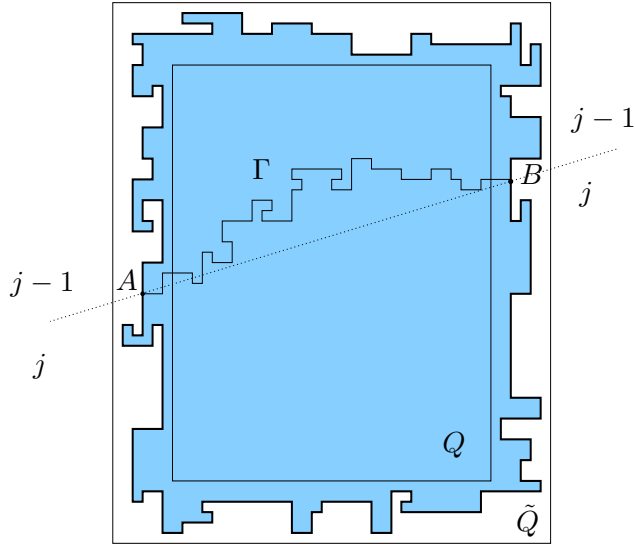,width=0.5\textwidth,height=0.45\textwidth}
\caption{The region $\L$ delimited by the circuit $\cC_*$, with the boundary conditions from Definition \ref{bcc}, and the open $j$-contour $\G$ joining $A$ and $B$ at height $a,b$ respectively.
}
\label{fig:ghiro}
\end{figure}
With the above definitions let $\pi_\L^{\xi}$ be the SOS measure in the finite
subset $\L$ of
$\tilde Q$ delimited by $\cC_*$, with boundary condition $\xi$ on
$\cC_*$ and floor at height $0$.
Note that the boundary conditions $\xi$ induce a unique open
$j$-contour $\Gamma$ from $A$ to $B$. 
Let also
$\theta_{A,B}\in [0,\pi/4]$ be the angle formed with the horizontal
axis by the segment $AB$, $\ell_{A,B}$  be the Euclidean distance
between $A,B$ and
$d_{A,B}=x_B-x_A$ where $x_A,x_B$ denote the horizontal coordinates of
$A,B$ respectively.
\begin{theorem}
  \label{th:napalla} Recall the Definition~\ref{lambdan} of
  $\l^{(n)}$.
 For every $n\ge0$ and every $x\in [x_A,x_B]$ such that $(x-x_A)\wedge
 (x_B-x)\ge \frac 1{10}d_{A,B}$, let $X^\pm (n,x)$ be the points with horizontal
coordinate $x$ and vertical coordinate
\[Y^\pm(n,x)=Y(n,x)\pm \sigma(x,\theta_{A,B})L^{\epsilon},
\]
where
\begin{gather}
Y(n,x)=\frac{a(x-x_A)+b(x_B-x)}{d_{A,B}}+\frac{\l^{(n)}\,(x-x_A)(x_B-x)}{2\b L
  (\tau(\theta_{A,B})+\t''(\theta_{A,B}))(\cos(\theta_{A,B}))^3}\label{eq:defY}
  \end{gather}
\begin{gather}\sigma^2(x,\theta_{A,B})=\frac{1}{\b (
 \t(\theta_{A,B})+\t''(\theta_{A,B}))(\cos(\theta_{A,B}))^3}\frac{(x-x_A)
 (x_B-x)}{d_{A,B}}.\nonumber
 \end{gather}
If the integer $j$ that
enters the definition of the boundary condition $\xi$ is equal to
$H(L) -n$, then:
\begin{enumerate}[(1)]
\item if $- \frac 12 L^{\frac 23 +\epsilon}\le a\le b\le L^{\frac 23 +\epsilon}-L^{\frac 13 +3\epsilon}$ then
  w.h.p.\ the point $X^-(n,x)$ lies below $\Gamma$.
\medskip

\item if $ -L^{\frac 23 +\epsilon}+L^{\frac 13 +3\epsilon}\le a\le b\le
  \frac 12 L^{\frac 23 +\epsilon}$ then w.h.p.\ the point $X^+(n,x)$ lies above $\Gamma$.
\end{enumerate}
\end{theorem}
\begin{remark}
In the above the fraction $1/10$ could be replaced by an arbitrarily
small constant independent of $L$. The core of the argument behind
Theorem~\ref{th:napalla} is the fact that
the height of $\G$ above $x$ is approximately a
Gaussian $\cN(Y(n,x),\sigma^2(x,\theta_{A,B}))$. Not
surprisingly $\sigma^2(x,\theta_{A,B})$ has the form of the variance of a Brownian
bridge.
In the concrete applications (see Section~\ref{GR}) we will only need
the above statement for $\bar x=\frac{x_A+x_B}{2}$. Note that, while
$Y(n,\bar x)- \frac{a+b}{2}$ is of order $L^{\frac 13 +2 \epsilon}$, one has that
the fluctuation term $\sigma(\bar
  x,\theta_{A,B})L^{\epsilon}$ is only $O(L^{\frac 13 +\frac 32 \epsilon})$.

\end{remark}


Following \cite{CLMST}*{Sect 7 and App. A}, we begin by deriving an
expression for the law of the open contour $\G$. We refer to~\eqref{eq:28}
for the definition of the decoration term $\Psi_\L(\G)$ (see also \cite{CLMST}*{App.\ A.3}).
\begin{lemma}
\label{Law of Gamma}
In the setting of Theorem~\ref{th:napalla}:
\begin{enumerate}[(i)]
\item $\pi_\L^\xi\left(|\G|\ge 2L^{\frac 23+\epsilon}\right)\le
  e^{-c L^{\frac 23 +\epsilon}}$.
\item Assume $|\G|\le 2 L^{\frac 23+\epsilon}$. Then
\begin{equation*}
\label{aexpan1}
\pi^{\xi}_\L(\G) \propto \exp\left(-\b |\G| +
\Psi_{\L}(\G)+\frac{\l^{(n)}}{L}|\L_-| +\varepsilon_n(L)\right),
\end{equation*}
where $|\L_-| $ denotes  the number of sites in $\L$ below $\G$ and $\varepsilon_n(L)=o(1)$ for any given $n$.
\end{enumerate}
\end{lemma}

\begin{proof}[Proof of the Lemma]
We first establish (i).
Denote by $\hat \pi^{\xi}_\L$ the SOS measure in $\L$ with b.c.\ $\xi$
and no floor. Then
\begin{gather*}
  \pi_\L^\xi\left(|\G|\ge 2L^{\frac 23+\epsilon}\right)\le \frac{\hat
    \pi_\L^\xi\left(|\G|\ge 2L^{\frac 23+\epsilon}\right)}{\hat
    \pi^{\xi}_\L\left(\eta_x\ge 0 \ \forall x\in \L\right)}.
\end{gather*}
The FKG\ inequality together with the simple bound $\min_{x\in
  \L}\hat\pi^{\xi}_\L(\eta_x\ge 0)\ge 1- c/L$,
imply that the denominator is larger than $\exp(-c |\L|/L)$ for some constant $c=c(n)>0$. A simple Peierls estimate gives instead that the
numerator is smaller than $\exp(-c L^{\frac 23 + \epsilon})$. Since
$|\L|/L\le c L^{\frac 13 +2\epsilon}$ the
result follows.

We now turn to part (ii).
Given $\G$, the region $\L$ is partitioned into two connected
regions $\L_+,\L_-$ separated by $\G$ (say that $\L_-$ is the one
below $\G$). Thus, if $Z^\xi_\L(\G)$ denotes the partition function
restricted to all surfaces whose open contour is $\G$, we have
\begin{equation}
\label{agamma}
Z^\xi_\L(\G)=
e^{-\b|\G|} Z^{(j)}_{\L_-}Z^{(j-1)}_{\L_+},
\end{equation}
where $Z^{(j)}_{\L_-}$ is the partition function of the SOS model in
$\L_-$ with $0$-b.c., floor at height $-j$ and the additional
constraint that $\eta_x\ge 0$ for all $x\in \L_-$ adjacent to $\G$,
$Z^{(j-1)}_{\L_+}$ is defined similarly except that $\eta_x\le 0$ for
all $x\in \L_+$ adjacent to $\G$.

Let $\hat Z_{\L_-}$ be defined as $Z^{(j)}_{\L_-}$ but without the
floor and similarly for $\hat Z_{\L_+}$. From Proposition~\ref{griglia} we know
that
\[
\frac{Z^{(j)}_{\L_-}Z^{(j-1)}_{\L_+}}{\hat Z_{\L_-}\hat Z_{\L_+}}=
\exp\left(-\hat \pi(\eta_0 > j)|\L_-|-\hat \pi(\eta_0 \ge j)|\L_+| +\varepsilon_n(L)\right).
\]
with $\varepsilon_n(L)=o(1)$ for any finite $n$.
Since $j=H(L)-n$, using Lemma~\ref{pbeta}:
\[
\hat \pi(\eta_0 \ge j)-\hat \pi(\eta_0 > j)= \frac{\l^{(n)}}{L} (1+O(L^{-1/2})).
\]
In conclusion,~\eqref{agamma} can be rewritten as
\[
Z^{\xi}_\L(\G)\propto
\exp\left(-\b |\G| +\frac {\l^{(n)}}{L} |\L_-| + \varepsilon_n(L)\right)\left(\frac{
\hat Z_{\L_-}\hat Z_{\L_+}}{\hat Z_\L}\right)\hat Z_\L.
\]
where $|\L_-|$ denotes the cardinality of $\L_-$ and $\hat Z_\L$ is
the partition function in $\L$ with no floor and $0$-b.c. Using Lemma~\ref{lem:dks}, as in~\eqref{eq:22}:
\[
\frac{
\hat Z_{\L_-}\hat Z_{\L_+}}{\hat Z_\L}=\exp(\Psi_\L(\G)),
\]  and the result follows.
\end{proof}

\begin{proof}[\textbf{\emph{Proof of Theorem~\ref{th:napalla}}}]
The fact that the circuit $\cC_*$ enclosing $\L$ is wiggled introduces
a number of inessential technical nuisances. In order not to hide the
main ideas we will prove the theorem in the case when $\L$ coincides
with $Q$ and we refer to
Appendix~\ref{app:1} for a discussion covering the general case. In
what follows we will drop the suffix $A,B$ from $\ell_{A,B}$ and
$\theta_{A,B}$. For simplicity we will only discuss the case $x=\frac
12 (x_A+x_B)$ and we will drop $x$ from $Y(n,x), Y^\pm (n,x)$. The
case of general $x$ at distance at least $d_{A,B}/10$ from $x_A,x_B$ can be treated similarly.

\begin{proof}[Proof of (1)]
We observe that the  event, in the sequel denoted by $U$, that the point $X^-(n)$ is above $\Gamma$
is decreasing. Thus, by FKG, if $G_+$ denotes the decreasing event that $\G$ does
not touch a $(\log L)^2$-neighborhood of the top side of $\L$ then
\[
 \pi_\L^\xi(U)\le \frac{\pi_\L^\xi(U;G_+)}{\pi_\L^\xi(G_+)}.
\]
The reason for conditioning on $G_+$ will be explained at the end of the
proof. Thanks to Lemma~\ref{Law of Gamma} we can write
\begin{equation}
  \label{eq:9}
\frac{\pi_\L^\xi(U;G_+)}{\pi_\L^\xi(G_+)} = \frac{\sum_{\G\in U\cap G_+} e^{-\b |\G| +
\Psi_{\L}(\G)+\frac{\l^{(n)}}{L}A_-(\G)}}{\sum_{\G\in G_+} e^{-\b |\G| +
\Psi_{\L}(\G)+\frac{\l^{(n)}}{L}A_-(\G)}}\times \Bigl(1+o(1)\Bigr)
\end{equation}

We observe that $A_-(\G)$ is, apart from an additive constant, the
\emph{signed area} $A(\G)$ of the contour $\G$ w.r.t.\ the straight
line joining $A,B$. Thus we can safely replace $A_-(\G)$ with
$A(\G)$ in the above ratio.

\subsubsection{Upper bound of the numerator}
Let $G_-$ denotes the event that $\G$ does
not touch a $(\log L)^2$-neighborhood of the bottom side of $\L$. A
simple Peierls argument shows that
\[
\sum_{\G\in U\cap G_+} e^{-\b |\G| +
\Psi_{\L}(\G)+\frac{\l^{(n)}}{L}A(\G)}= (1+o(1))\times
\sum_{\G\in U\cap G_+\cap G_-} e^{-\b |\G| +
\Psi_{\L}(\G)+\frac{\l^{(n)}}{L}A(\G)},
\]
since getting close to the bottom side of $\L$ implies an anomalous
contour excess length.
If now $S$ denotes the infinite vertical strip through the points
$A,B$, for any $\G\in G_+\cap G_-$ the decoration term
$\Psi_\L(\G)$ satisfies
\[
|\Psi_\L(\G)-\Psi_S(\G)|=o(1)
\]
thanks to~\eqref{eq:28} and Lemma~\ref{lem:dks}.
Therefore we can upper bound the
numerator by
\[
(1+o(1))\times  \sum_{\G\in U} e^{-\b |\G| +
\Psi_{S}(\G)+\frac{\l^{(n)}}{L}A(\G)}.
\]
Although we replaced the finite volume decorations $\Psi_\L$ with the
decorations $\Psi_S$ associated to the infinite strip $S$ we emphasize that the
contours will always be constrained within the original box $\L$.

It will be convenient in what follows to define, for an arbitrary
event $E$,
\[
Z_{\l^{(n)}}(E):=\sum_{\G\in E} e^{-\b |\G| +
\Psi_{S}(\G)+\frac{\l^{(n)}}{L}A(\G)};\quad
Z_{\l^{(n)}}:=\sum_{\G} e^{-\b |\G| +
\Psi_{S}(\G)+\frac{\l^{(n)}}{L}A(\G)}.
\]

Let now $Y\in [-L^{\frac 23 +\epsilon},L^{\frac 23 +\epsilon}]$ be the
height of the first (following $\Gamma$ from $A$) horizontal bond of $\G$
crossing the middle vertical line of $S$ and let us write
$U=U_1\cup U_2$ where
\[
U_1=U\cap\{Y \ge \hat Y^-(n)\},\quad  U_2=U\cap\{Y
\le \hat Y^-(n)\}
\]
and $\hat Y^-(n)=Y(n)-\frac 12 \sigma(x,\theta)L^{\epsilon}$. In order to estimate $Z_{\l^{(n)}}(U_1), Z_{\l^{(n)}}(U_2) $ we will apply the bounds of Section~\ref{iterative} below with the choice of the parameter $\mu=\l^{(n)}$.

We start with $Z_{\l^{(n)}}(U_1)$.
Multiplying and dividing
by $Z_0$ gives
\begin{gather}  \label{eq:6}
Z_{\l^{(n)}}(U_1)=
Z_{0} \times\bbE_{\l=0}\left(U_1;e^{\frac{\l^{(n)}}{L}A(\G)}\right)
\le Ce^{-\b \tau(\theta) \ell}\, \left(\frac{Z_{2\l^{(n)}}}{Z_0}\right)^{1/2}\, \sqrt{\bbP_{0}(U_1)}
\end{gather}
where we used \cite{DKS}*{Sect. 4.12} (or Corollary~\ref{iteration}
below at $\mu=0$) to upper bound $Z_{0}$.
Again by Corollary~\ref{iteration} below
\[
\frac{Z_{2\l^{(n)}}}{Z_0}\le e^{c L^{3\epsilon}}.
\]
Finally we observe that  the event $U_1$ implies that the contour $\G$ touches the
middle vertical line in two points separated by a distance larger than
$\frac 12 \sigma(x,\theta)L^{\epsilon} \geq c L^{\frac 13
  +\frac 32 \epsilon}$. From \cite{DKS}*{Ch.~4} such an event has
probability smaller than $\exp(-c L^{\frac 13})$ under the $\l=0$ measure.
In conclusion
\begin{equation}
  \label{eq:7}
  Z_{\l^{(n)}}(U_1)\le Ce^{-\b \tau(\theta) \ell}e^{- c L^{\frac 13}}.
\end{equation}
We now turn our attention to the term $Z_{\l^{(n)}}(U_2)$. We first decompose $\G=\G_1\circ\G_2$ into
the piece $\G_1$ from $A$ to $C=(0,Y)$ and $\G_2$ from $C$ to
$B$. Then we write
\[
A(\G)= A_{0}+A_1(\G_1)+A_2(\G_2)
\]
where $A_{0},A_1(\G_1),A_2(\G_2)$ are the signed areas of the
triangle $(A C B)$ and of the contours $\G_1,\G_2$ w.r.t.\ the segments
$AC, CB$ respectively.
Thus
\begin{align*}
Z_{\l^{(n)}}(U_2)&\le \sum_{y\le \hat Y^-(n)}e^{\frac{\l^{(n)}}{L}A_0}\times \sum_{\G:\ Y=y}e^{-\b |\G_1| +
\Psi_{S}(\G_1)+\frac{\l^{(n)}}{L}A_1(\G_1)}e^{-\b |\G_2| +
\Psi_{S}(\G_2)-\D\Psi_S(\G_1,\G_2)+\frac{\l^{(n)}}{L}A_2(\G_2)}\\
&=:\sum_{y\le \hat Y^-(n)}e^{\frac{\l^{(n)}}{L}A_0}
\sum_{\G_1:\ Y=y}e^{-\b |\G_1| +
\Psi_{S}(\G_1)+\frac{\l^{(n)}}{L}A_1(\G_1)}Z_{\l,y,\G_1}
 \end{align*}
where
\begin{equation}
  \label{eq:14}
\D\Psi_S(\G_1,\G_2)=\Psi_S(\G_1)+ \Psi_S(\G_2) - \Psi_S(\G_1\circ\G_2)
\end{equation}
It now follows from Corollary~\ref{iteration} that
\begin{align*}
\sup_{\G_1}Z_{\l^{(n)},y,\G_1}&\le \exp(\cG_{\l^{(n)}}(\ell_2,\theta_2)+L^{3\epsilon/2})\\
\sum_{\G_1:\ Y=y}e^{-\b |\G_1| +
\Psi_{S}(\G_1)+\frac{\l^{(n)}}{L}A_1(\G_1)} &\le  \exp(\cG_{\l^{(n)}}(\ell_1,\theta_1)+L^{3\epsilon/2})
\end{align*}
where
$\ell_1,\ell_2$ denote the distances between $A, C$ and $C,B$
respectively, $\theta_1,\theta_2$ the angles w.r.t.\ the horizontal
direction of the segments $AC$ and $CB$ and we define
\begin{equation}
  \label{eq:8}
\cG_{\l}(\ell,\theta)=-\beta\tau(\theta)\ell+\frac{\lambda^2\,\ell^3}{24\beta(\tau(\theta)+\tau''(\theta))L^2}.
\end{equation}
Putting all together we get
\begin{align}
\label{eq:7bis}
  Z_{\l^{(n)}}(U_2)&\le e^{2L^{3\epsilon/2}}\sum_{y\le
    \hat Y^-(n)}\exp\left[\frac{\l^{(n)}}{L}A_0+\cG_{\l^{(n)}}(\ell_1,\theta_1)+\cG_{\l^{(n)}}(\ell_2,\theta_2)\right]\nonumber\\
&=:e^{2L^{3\epsilon/2}}\left[\Sigma_1 +\Sigma_2\right]
\end{align}
where
\begin{align}
  \Sigma_1&= \sum_{y\le (a+b)/2 - L^{\frac
      13+3\epsilon}}\exp\left[\frac{\l^{(n)}}{L}A_0
    +\cG_{\l^{(n)}}(\ell_1,\theta_1)+\cG_{\l^{(n)}}(\ell_2,\theta_2)\right],\;\;\Sigma_2=\Sigma-\Sigma_1
\label{Sigma2}
\end{align}
Using  $\cG_{\l^{(n)}}(\ell,\theta)\le -\b
\tau(\theta) \ell +O(L^{3\epsilon})$ together with the strict
convexity of the surface tension \cite{DKS}, we get that the first sum is upper
bounded by
\[
\Sigma_1\le
\exp\left(-\b \tau(\theta)\ell- c L^{5\epsilon}\right)
\]
for some constant $c>0$ where we used that $A_0\le 0$ for $y\le (a+b)/2$.

In order to bound $\Sigma_2$ we observe that for all $y\in [(a+b)/2 -
L^{\frac 13 +3\epsilon},\hat Y^-(n)]$
\[
\varphi:=\theta_1-\theta=
O(L^{-\frac 13 +2\epsilon}).
\]
Thus it suffices to expand $\cG_{\l^{(n)}}(\ell_i,\theta_i)$ in
$\varphi$ up to
second order. A little trigonometry
shows that
\[
\ell_1= L^{\frac 23 +\epsilon}/\cos(\theta +\varphi),\quad  \ell_2=
L^{\frac 23 +\epsilon}/\cos(\theta -\psi(\varphi)),\quad \psi(\varphi)=\varphi +
2\tan(\theta)\varphi^2 +O(\varphi^3).
\]
Moreover
\[
\frac{1}{24}\ell_i^3\frac{(\lambda^{(n)})^2}{\beta(\tau(\theta_i)+\tau''(\theta_i))L^2}=\frac
18\ell^3\,
\frac{(\lambda^{(n)})^2}{24\beta(\tau(\theta)+\tau''(\theta))L^2}+o(1), \qquad
i=1,2
\]
while
\[
\t(\theta_1)\ell_1
+\t(\theta_2)\ell_2=\t(\theta)\ell+2\left(\t(\theta)+\t''(\theta)\right)\frac{[y-(a+b)/2]^2\cos(\theta)^2}{\ell}+ o(1).
\]
In conclusion
\begin{gather*}
\cG_{\l^{(n)}}(\ell_1,\theta_1)+\cG_{\l^{(n)}}(\ell_2,\theta_2)\\= \cG_{\l^{(n)}}(\ell,\theta)
-\frac{1}{32}\frac{(\l^{(n)})^2\,\ell^3 }{\b (\t(\theta)+\t''(\theta))L^2} -2\b \left(\t(\theta)+\t''(\theta)\right)\frac{[y-(a+b)/2]^2\cos(\theta)^2}{\ell}
+o(1)
\end{gather*}
and
\begin{gather}
\nonumber
\Sigma_2 \le
(1+o(1))\exp\left(\cG_{\l^{(n)}}(\ell,\theta)
-\frac{1}{32}\frac{(\l^{(n)})^2\,\ell^3 }{\b (\t(\theta)+\t''(\theta))L^2}
\right)\\
\label{dc}
\times \sum_{y\in [(a+b)/2 -
L^{\frac 13+3\epsilon},\hat Y^-(n)]}\exp\left(\frac{\l^{(n)}}{L}A_0   -2\b \left(\t(\theta)+\t''(\theta)\right)\frac{[y-(a+b)/2]^2\cos(\theta)^2}{\ell}                  \right).
 \end{gather}
Since $A_0= \frac 12 L^{\frac 23 +\epsilon}(y-(a+b)/2)$ and $\ell
\cos(\theta)=L^{\frac 23 +\epsilon}$,
one finds that
\begin{align}
& \frac{\l^{(n)}}{L}A_0   -2\b \left(\t(\theta)+\t''(\theta)\right)\frac{[y-(a+b)/2]^2\cos(\theta)^2}{\ell}     \nonumber \\
&\qquad \qquad \qquad= -\frac{(y-Y(n))^2}{2\bar \si^2} + \frac{1}{32}\frac{(\l^{(n)})^2\,\ell^3 }{\b (\t(\theta)+\t''(\theta))L^2},
\label{csq}
\end{align}
where \[
Y(n) = \frac{a+b}2 + \frac{1}{8}\frac{\l^{(n)} L^{\frac13+2\epsilon} }{\b (\t(\theta)+\t''(\theta))(\cos\theta)^3},
\]
as in~\eqref{eq:defY} for $x=(x_A+x_B)/2$, and
\[
\bar\si^2 = \frac{\ell}{4\b\left(\t(\theta)+\t''(\theta)\right)\cos(\theta)^2}
\]
 Therefore, apart from the factor $\exp(\cG_{\l^{(n)}}(\ell,\theta))$, the summand in~\eqref{dc} is proportional to a Gaussian density with mean
$Y(n)$ and variance $\bar\si^2$.  Using $\hat Y^-(n)=Y(n)-\frac12
\si(x,\theta_{A,B})L^\epsilon\le Y(n)- cL^{\frac13+\frac32 \epsilon}$, and $\bar\si^2=O(L^{\frac 23 +\epsilon})$, one finds
for some $c>0$ \[
\Sigma_2\le \exp(\cG_{\l^{(n)}}(\ell,\theta) -c L^{2\epsilon}).
\]
In conclusion, using~\eqref{eq:7} and~\eqref{eq:7bis},  the numerator $Z_{\l^{(n)}}(U;G_+)$ appearing in the r.h.s.\ of~\eqref{eq:9} satisfies
\begin{equation}
  \label{eq:11}
Z_{\l^{(n)}}(U;G_+)\le \exp(\cG_{\l^{(n)}}(\ell,\theta) -c L^{2\epsilon})
\end{equation}
for some new constant $c>0$.
\medskip

\subsubsection{Lower bound on the denominator}
We consider the restricted class of contours defined as the set of $\G$ that stay within the
neighborhood of size $\big(L^{\frac 23 +\epsilon}\big)^{1/2 +
  \epsilon/3}$ around the optimal curve $\G^{\l^{(n)}}_{\rm opt}$
defined by
\begin{equation}
  \label{eq:parabola}
\G^\mu_{\rm opt}(x)= \G^{(1)}_{\rm opt}(x)+\G^{(2)}_{\rm opt}(x),
\quad x\in [x_A,x_B]
\end{equation}
where $x\mapsto \G^{(1)}_{\rm opt}(x)$ describes the
straight segment $AB$ and
\[
\G^{(2)}_{\rm opt}(x)=\left(\frac{\mu \ell_{A,B}^3}{2\b L
  (\tau(\theta_{A,B})+\t''(\theta_{A,B}))d_{A,B}^3}\right)(x-x_A)(x_B-x),\quad
x\in [x_A,x_B],
\]
where $d_{A,B}=x_B-x_A$.
Note that,
thanks to the assumption $a\le b\le
L^{\frac 23 +\epsilon}-L^{\frac 13 +3\epsilon}$, the curve $\G^{\l^{(n)}}_{\rm opt}$
is well within the domain $Q$.
For such contours $\G$ one has
\[
A(\G)= A(\G^{\l^{(n)}}_{\rm opt})+ O(L^{1+\frac{11}{6}\epsilon}).
\]
Thus
\[
Z_{\l^{(n)}}(G_+)\ge e^{O(L^{\frac {11}6 \epsilon})}\,e^{\frac {\l^{(n)}}{L} A(\G^{\l^{(n)}}_{\rm
    opt})}\!\!\!\!\!\!\!\!\sum_{\G:\ {\rm dist}(\G,\G^{\l^{(n)}}_{\rm opt})\le \big(L^{\frac 23 +\epsilon}\big)^{1/2 + \epsilon/3}}e^{-\b|\G|+\Psi_{\L}(\G)}.
\]
As in the proof of Lemma A.6 of \cite{MT}, the latter sum is lower
bounded by
\[
\exp\left(-\b \int_{\Gamma^{\l^{(n)}}_{\rm opt}} ds\, \tau(\theta_s)
\right) \exp(-(\log L)^c ).
\]
Using~\eqref{eq:12} we
finally get
\begin{equation}
  \label{eq:10}
Z_{\l^{(n)}}(G_+)\ge \exp\left(\cG_{\l^{(n)}}(\ell,\theta) + O(L^{\frac{11}{6}\epsilon})\right)
\end{equation}
where $\cG_{\l^{(n)}}(\ell,\theta)$ is as in~\eqref{eq:8}.

\subsubsection{Conclusion}
By combining~\eqref{eq:11} with~\eqref{eq:10} we finally get
point (1) of the theorem.
\end{proof}

\begin{proof}[Proof of (2)]
The proof of point (2) follows exactly the same pattern. Using FKG one first conditions on $G_-$ and
then, using Peierls, one restricts to paths in $G_-\cap G_+$.
\end{proof}
This concludes the proof of Theorem~\ref{th:napalla}.\end{proof}

\subsection{Iterative upper bound on the partition function}
\label{iterative}
This is a key technical section whose main object is a certain contour
  partition function $Z_{A,B}$ for open contours
  joining two points $A,B$ at distance $\sim L^{\frac 23 + \epsilon}$. The exponential weight of a contour
  contains, besides the familiar length term with decorations as in
  \cite{DKS},  an additional term
  proportional to $L^{-1}\times$~the signed area of the contour w.r.t.\ the segment
  $AB$. The main output is a precise upper bound on
  $Z_{A,B}$. The bound indicates that the main contribution to $Z_{A,B}$ comes from
  contours close to a deterministic curve (approximately a parabola)
  joining $A,B$ and satisfying a variational principle (cf.~\eqref{funct1}).
\subsubsection{Setting}
Recall that $Q$ is a rectangle of horizontal side $D:=L^{\frac 23
  +\epsilon}$ and vertical side $2D$ centered
at the origin.
Set $\ell_0=L^{\frac 23}$ and $\delta = 1/10$ and define $\cR_n$ as the set
of pairs $(A,B)$, $A,B\in Q\cap (\bbZ^2)^*$, such that their Euclidean distance $\ell_{A,B}$
satisfies  $1\le \ell_{A,B}\le 2^{n(1-\delta)}\ell_0$. Call  $\theta_{A,B}$ the angle formed with the horizontal axis by the straight
line through $A,B$ and assume that $\theta_{A,B}\in[0,\pi/4]$. Without
loss of generality we assume that $x_A<x_B$ if $x_A,x_B$ denote the horizontal coordinates of $A,B$.

Choose two open contours $\G_{\rm left},\G_{\rm right}$ such that
  $\G_{\rm left}$ joins $A$ with the left vertical side of $Q$
    without ever going to the right of $A$ and $\G_{\rm
      right}$ joins $B$ with the right vertical side of $Q$ without
    ever going to the left of $B$.

Define now the contour ensemble $\Xi$ consisting of all
contours $\G$ joining $A,B$ within $Q$ such that the concatenation
$\G_{\rm left}\circ\G\circ\G_{\rm
      right}$ is an admissible open contour. Notice that the signed
    area $A(\G)$  of $\G$ w.r.t.\ the segment $AB$
is unambiguously defined.
Fix a parameter $\mu\ge 0$ and
to each $\G\in \Xi$ assign the weight
\begin{eqnarray}
  \label{eq:ueit!}
w(\G)=\exp\left(-\b |\G| + \Psi_{\bbZ^2}(\G) + e^{-\b}|\G\cap Q_{A,B}|+\frac{\mu}{L}A(\G)\right)
\end{eqnarray}
where $\Psi_{\bbZ^2}(\G)$ has been defined in~\eqref{eq:28},
and $Q_{A,B}\subset Q$  consists of all those points whose horizontal
coordinate $x$ satisfies either $x\le x_A+(\log L)^2$ or $x\ge x_B-(\log
L)^2$. The
term $e^{-\b}|\G\cap Q_{A,B}|$ has been added only for technical
convenience and, in practice, it will be $O((\log L)^2)$ for the ``relevant'' contours.
\begin{definition}
Let
$
Z_{A,B}:=\sum_{\G\in \Xi}w(\G).
$
We say that statement $H_n$ holds if
\[
\sup_{\G_{\rm left}, \G_{\rm right}}Z_{A,B}\le z_n e^{\mathcal
  \cG_\mu(\ell_{A,B},\theta_{A,B})}\quad \forall A,B\in \cR_n
\]
where
\[
\mathcal G_\mu(\ell,\theta)=-\beta\tau(\theta)\ell+\ell^3\frac{\mu^2}{24\beta(\tau(\theta)+\tau''(\theta))L^2}
\]
and $z_1 =e^{c(\log L)^2},\ z_{n}=\left(Lz_1\right)^{2^{n-1}}$, $n\ge
  2$.
\end{definition}
With the above notation the following holds.
\begin{proposition}
\label{iteration-bis}

For any large $L$, $H_1$ holds. Moreover, for all $n\le n_f\equiv
\epsilon(1-\d)^{-1}\log_2(L)$, $H_n$ implies $H_{n+1}$. In particular
$H_n$ holds for all $n\le n_f$.
\end{proposition}
\begin{remark}
\label{zn}
It is important to observe that $z_{n_f}=O\left(\exp(L^{3\epsilon/2})\right)$.
\end{remark}
The reason why $H_n$ holds is that the main contribution to the
partition function $Z_{A,B}$ comes from contours $\G$ which are close
to the curve from $A$ to $B$ maximizing the functional
\begin{equation}
  \label{funct1}
\cC\mapsto -\b \int_{\cC}\tau(\theta_s)\, ds + \frac{\mu}{L}A(\cC).
\end{equation}
By expanding
the functional up to second order around the straight line from
$A$ to $B$  one
finds easily that such optimal curve is approximately the parabola
given by~\eqref{eq:parabola}.
A short computation shows that
\begin{equation}
  \label{eq:12}
-\b \int_{\G^\mu_{\rm opt}}\tau(\theta_s)\, ds +
\frac{\mu}{L}A(\G^\mu_{\rm opt})=\cG_\mu(\ell_{A,B},\theta_{A,B}) +o(1).
\end{equation}
Before proving the proposition let us state a simple corollary which
formalizes a useful consequence of the result.
\begin{corollary}
\label{iteration}
Consider the two partition functions
$
Z^{(i)}_{A,B}:=\sum_{\G\in \Xi}w^{(i)}(\G)
$, $i=1,2$,
corresponding to the weights
\begin{align*}
w^{(1)}(\G) &=\exp\left(-\b |\G| + \Psi_{S}(\G)+\frac{\mu}{L}A(\G)\right)\\
w^{(2)}(\G) &=\exp\left(-\b |\G| + \Psi_{S}(\G)- \D\Psi_S(\G,\G_{\rm left})  +\frac{\mu}{L}A(\G)\right),
\end{align*}
where $S$ is any vertical strip containing the strip through the points $A,B$ and
$\D\Psi_S(\G,\G_{\rm left})$ has been defined  in~\eqref{eq:14}. Then, uniformly in $\G_{\rm left}, \G_{\rm right}$:
\begin{align*}
  \max\left(Z_{A,B}^{(1)}, Z_{A,B}^{(2)} \right)\le
  \exp\left(\cG_\mu(\ell_{A,B},\theta_{A,B})+O(L^{3\epsilon/2})\right).  
.
\end{align*}
\end{corollary}
\begin{proof}[Proof of the Corollary]
It follows immediately from Proposition~\ref{iteration-bis} together
with Remark~\ref{zn} and the bounds
\begin{equation*}
|\Psi_{\bbZ^2}(\G)-\Psi_{S}(\G)|
+ \sup_{\G_{\rm left}}|
  \D\Psi_S(\G,\G_{\rm left})|
  \le e^{-\b}|\G\cap Q_{A,B}|.\qedhere
\end{equation*}
\end{proof}
\begin{proof}[Proof of Proposition~\ref{iteration-bis}]\
\subsubsection{Proof of the base case $H_1$}\ \\
Fix $(A,B)\in \cR_{1}$ with $x_A<x_B$, mutual distance $\ell\le 2\ell_0=2L^{2/3}$ and angle
$\theta$, together with $\G_{\rm left}, \G_{\rm right}$ and denote by $h_\G$ the maximal height (w.r.t.\ the segment $AB$) reached by the
contour $\G$. Then
\begin{gather*}
Z_{A,B} = \hat Z_{A,B}\ \hat \bbE_{A,B}\left(e^{\frac \mu L A(\G) +
    e^{-\b}|\G\cap Q_{A,B}|}\right)
\end{gather*}
where $\hat Z_{A,B}$ is defined as $Z_{A,B}$ but with modified weights
$\hat w(\G)$ in which the area
parameter $\mu$ is equal to zero and the term $ e^{-\b}|\G\cap Q_{A,B}|$
is absent. The area $A(\G)$ clearly satisfies $A(\G)\le |\G| h_\G$.
Because of \cite{DKS}*{Ch.\ 4},
\begin{align}
\hat \bbP_{A,B}\left(h_\G=j\right) &\le c e^{-\min(j,\ j^2/\ell)/c}\nonumber\\
\hat Z_{A,B}&\le c\,e^{-\b\t(\theta)\ell}\nonumber \\
\hat \bbP_{A,B}\left( |\G\cap Q_{A,B}|\ge q\right) &\le e^{- q +c(\log L)^2}
\label{***}
\end{align}
for suitable constant $c$ and $\b$ large enough.
From~\eqref{***} it follows that
\[
 \hat \bbE_{A,B}\left(e^{2 e^{-\b}|\G\cap Q_{A,B}|}\right)\le e^{c'(\log L)^2},
\]
for some constant $c'>0$.
Moreover, using Peierls, the excess length $(|\G|-2\ell)^+$ has
exponential tail with parameter $\b-O(1)$. This, combined with
the first estimate in~\eqref{***} proves that
\[
\hat \bbE_{A,B}\left(e^{2\frac \mu L |\G|h_\G}\right)\le c' ,
\]
for some constant $c'>0$.
 The claim is proved with $z_1:=e^{c(\log L)^2}$.

\subsubsection{Proof of the inductive step $H_n\Rightarrow H_{n+1}$}\ \\
Fix $(A,B)\in \cR_{n+1}$ with $x_A<x_B$, mutual distance $\ell$ and angle
$\theta$, together with $\G_{\rm left}, \G_{\rm right}$.
Let $C$ be the midpoint between $A$ and $B$, and
define $\bbL$ the vertical line
through $C$. Write $\Gamma$ as $\G=\Gamma_1\circ\Gamma_2$ where $\Gamma_1$ is the contour from $A$ until the first contact $X$ with $\bbL$
and $\Gamma_2$ is the remaining part. Let also $\ell_{AX},\theta_{AX}$
be the length and angle of $AX$ and similarly for $\ell_{XB},\theta_{XB}$.

Call $j$ the vertical coordinate of $X$ minus the vertical
coordinate of $C$.
We distinguish two cases.
\medskip

\emph{Case 1: $|j|\ge L^{\frac 13+3\epsilon}$.}
With the same notation as in the proof of $H_1$ and using~\eqref{***}
\begin{eqnarray}
\sum_{\Gamma: |j|\ge L^{\frac 13+3\epsilon}}w(\gamma)\le  c e^{-\b
  \tau(\theta)\ell -c'\, L^{5\epsilon}}
\label{caso1}
\end{eqnarray}
for some positive $c,c'$
which is clearly negligible compared to the target $z_{n+1}\exp(\cG_\mu(\ell,\theta))$.
\medskip

\emph{Case 2: $|j|\le L^{\frac 13+3\epsilon} $.} Simple geometry
  shows that both $(A,X)$ and $(X,B)$ belong to $\cR_n$ and we can use
the induction.  Also, Peierls argument shows that we can safely
assume that $\Gamma_2$ does not reach horizontal coordinate $x_A+(\log
L)^2$, with $x_A$ the horizontal coordinate of $A$, otherwise this
would imply an extremely unlikely large deviation of the length $|\Gamma|$.

Note that the area $A(\Gamma)$ can be written as
\[
A(\Gamma)=A_1(\Gamma_1)+A_2(\Gamma_2)+A_{0},
\] with $A_1(\Gamma_1)$
(resp.  $A_2(\Gamma_2)$)
the signed area of $\Gamma_1$ (resp. $\Gamma_2$) w.r.t.\ the
segment $AX$, (resp. $XB$) while
\[
A_{0}=\frac\ell{2}j\cos(\theta)
\]
is the
signed area w.r.t.\ $AB$ of the triangle $AXB$.

Next, remark that
\[
\Psi_{\bbZ^2}(\Gamma)=\Psi_{\bbZ^2}(\Gamma_1)+\Psi_{\bbZ^2}(\Gamma_2)-\Delta\Psi_{\bbZ^2}(\Gamma_1, \Gamma_2).
\]
Using the decay properties of the potentials $\varphi$ (see Lemma
\ref{lem:dks} point (iii)), we can bound
\begin{eqnarray}
  \label{eq:13}
|  \D\Psi_{\bbZ^2}(\G_1,\G_2)|\le e^{-\beta}|\Gamma_2\cap Q_{X,B}|
\end{eqnarray}
where $Q_{X,B}$ was defined just after~\eqref{eq:ueit!} (with $A$
replaced by $X$).
As a consequence,
\begin{gather}
  \label{eq:4}
  \sum_{\Gamma:|j|\le L^{\frac 13+3\epsilon}}w(\gamma)\le
\sum_{\Gamma:|j|\le L^{\frac 13+3\epsilon}}
\exp\left[\frac{\mu}L(A_1(\Gamma_1)+A_2(\Gamma_2))+\frac{\mu
  \ell}{2L}j\cos(\theta)\right]\\
\times \exp\left[e^{-\beta}(|\Gamma_1\cap Q_{A,X}|+|\Gamma_2\cap Q_{X,B}|) +\Psi_{\bbZ^2}(\Gamma_1) +\Psi_{\bbZ^2}(\Gamma_2)
\right].
\end{gather}
At last we can use the induction:
with $\G_{\rm left}\circ \G_1$ playing the role of $\G_{\rm left}$ we
have (uniformly in $\G_1$)
\[
\sum_{\G_2}\exp\left[\frac{\mu}L A_2(\Gamma_2)+\Psi_{\bbZ^2}(\Gamma_2)+e^{-\beta}|\Gamma_2\cap
  Q_{X,B}|\right]\le
z_n \exp(\cG_\mu(\ell_{XB},\theta_{XB}))
\]
and similarly, with e.g.\ horizontal contour from $X$ to the right
vertical boundary of $Q$ playing the role of $\G_{\rm right}$,
\[
\sum_{\G_1}\exp\left[\frac{\mu}L A_1(\Gamma_1)+\Psi_{\bbZ^2}(\Gamma_1)+e^{-\beta}|\Gamma_1\cap
  Q_{A,X}|\right]\le
z_n \exp(\cG_\mu(\ell_{AX},\theta_{AX})).
\]
To estimate
\[
\Sigma:=z_n^2\sum_{|j|\le L^{\frac 13+3\epsilon}}\exp\left[\frac{\mu
  \ell}{2L}j\cos(\theta)+\cG_\mu(\ell_{AX},\theta_{AX})+\cG_\mu(\ell_{XB},\theta_{XB})
\right]
\]
we proceed as in the estimate of the sum $\Sigma_2$ appearing in
\eqref{Sigma2}. Using the restriction $|j|\le L^{\frac 13+3\epsilon}$ we can
expand up to second order the exponent in e.g.\ $(\theta-\theta_{AX})=O(L^{-\frac 13+3\epsilon})$. The net result is that
\begin{align*}
\Sigma&\le (1+o(1)) z_n^2
\exp\left(\cG_\mu(\ell,\theta)
-\frac{1}{32}\frac{\mu^2\,\ell^3 }{\b (\t(\theta)+\t''(\theta))L^2}
\right)\\
&\times \sum_{j}\exp\left(\frac{\mu
  \ell}{2L}j\cos(\theta)   -2\b
\left(\t(\theta)+\t''(\theta)\right)\frac{j^2\cos(\theta)^2}{\ell}
\right)\\
&\le
c(\b)\sqrt{\ell}z_n^2 \exp\left(\cG_\mu(\ell,\theta)\right)
\end{align*}
where we used a standard Gaussian summation.
In conclusion, using~\eqref{caso1}, we showed that
\[
Z_{A,B}\le c z_n^2 \sqrt{\ell} \exp\left(\cG_\mu(\ell,\theta)\right)\le
z_{n+1} \exp\left(\cG_\mu(\ell,\theta)\right)
\]
thanks to the definition of the constants $\{z_n\}_{n\le n_f}$. The
inductive step is complete.
\end{proof}


\section{Proof of Theorems~\ref{mainthm-2} and~\ref{mainthm-3}}
\label{GR}

In this section
we show that for all $n\in\bbZ_+$, if there exists a macroscopic $(H(L)-n)$-contour $\G_n$ containing the rescaled Wulff body $L\ell_c(\l^{(n)}) \cW_1$, then with  high probability it is unique and it is contained in the annulus $(1+\e_0)L\cL_c(\l^{(n)})\setminus (1-\e_0)L\cL_c(\l^{(n)})$, for any $\e_0>0$. Combined with the results of Section~\ref{US-section}, this will prove Theorem~\ref{mainthm-2}.
Moreover, we prove that along the flat part of $L\cL_c(\l^{(n)})$ the contour $\G_n$ has fluctuations on the scale $L^{\frac 13}$ up to $O(L^{\epsilon})$ corrections. That covers Theorem~\ref{mainthm-3}.


\subsection{Growth of droplets}
\label{growth}

Recall the Definition~\ref{Lstorto} of the sets $\cL(\l,t,r)$ and the sets
$\cL_c(\l)$. Recall also the Definition~\ref{lambdan} of the parameters $\l^{(n)}$.  To fix ideas, the Wulff shape $\cW_1$ appearing below is assumed to be centered at the origin.
%
%
To simplify the exposition, we introduce the following notation.
\begin{definition}
\label{ccH}
Given a subset $A\subset\bbZ^2$,
we call $\cE_n(A)$ the event that there exists an
$(H(L)-n)$-contour $\G_n$ that contains the set $A$.
\end{definition}

\begin{theorem}[Growth of the critical droplet]\label{th:growth0}
Fix $\e\in(0, 1/10)$, and set $\d_L=L^{-\e/8}$. Let $\L$ be the square of side $L$ centered at the origin.
Consider the SOS model on $\L$
with $0$-boundary conditions and floor at height
zero.  For any fixed $n \in \bbZ_+$,  as $L\to \infty$,
if
\begin{equation}\label{evento}
\cE_n(L(1+\d_L) \ell_c(\l^{(n)})\cW_1)\,\;\, \text{holds w.h.p.},
\end{equation}
then w.h.p.\
\begin{enumerate}[a)]
\item $\cE_n(L\cL(\l^{(n)},1+\d_L,-L^{-\frac 23+4\e}))$ holds;
\item there exists a unique macroscopic $(H(L)-n)$-contour.
\end{enumerate}
\end{theorem}
\begin{remark}\label{capp}
Recall from Remark~\ref{key-property} that $L\ell_c(\l^{(n)}) \cW_1$ can fit inside the box $\L$ iff $\l^{(n)}\ge \hat \l$. Using that $\hat \l\sim 2\beta$ (see Remark~\ref{rem:betalargo}), we have that for $n\ge 1$ this condition is always satisfied (for $\b$ large enough) while if $n=0$ we need to require $\l\ge \hat \l$.  However the results of Section~\ref{US-section} show that a macroscopic $H(L)$-contour exists w.h.p. iff $\l>\l_c>\hat \l$.
\end{remark}
\begin{remark}[Growth up to $L^{\frac 13}$ from flat
  boundary]\label{rem:growth 1/3}
An immediate corollary of Theorem~\ref{th:growth0} is that assuming~\eqref{evento},
the unique macroscopic $(H(L)-n)$-contour is at a distance $O(L^{\frac 13+4\e})$ from the target region $L\cL_c(\l^{(n)})$, uniformly along most of the flat boundary of $L\cL_c(\l^{(n)})$. Indeed, $L\cL(\l^{(n)},1+\d_L,-L^{-\frac 23+4\e})$ is uniformly at a distance
$L^{\frac 13+4\e}$ from the critical region $L\cL(\l^{(n)},1+\d_L,0)$, which overlaps with $L\cL_c(\l^{(n)})$ along the flat boundary of $L\cL(\l^{(n)},1+\d_L,0)$. On the other hand, concerning the curved portions of $L\cL_c(\l^{(n)})$, the above theorem does not allow us to infer an approximation error better than $O(\d_L L)$, since
already the region $L\cL(\l^{(n)},1+\d_L,0)$ has radial distance from $L\cL_c(\l^{(n)})$ of that order at a corner.
\end{remark}

The proof of Theorem~\ref{th:growth0} part $a)$ will be based on an inductive argument. The proof of Theorem~\ref{th:growth0} part $b)$ will be a consequence of part $a)$, as we show below.

\begin{proof}[Proof of Theorem~\ref{th:growth0} part $b)$ assuming part $a)$]
Thanks to Theorem~\ref{th:growth0} part $a)$, for any fixed $n$ w.h.p.\, assuming~\eqref{evento}, there exists an outermost  $(H(L)-n)$-contour,  which we denote $\G_n$, containing the set $\Lambda_n:=(1-o(1))L\cL_c(\l^{(n)})$ for a suitable choice of
the error term $o(1)$.  Proposition~\ref{bdgma} and a union bound imply there are no macroscopic negative  contours w.h.p.\ and hence there are no positive macroscopic $(H(L)-n)$-contours nested inside $\G_n$.
Hence the interior of any other macroscopic contour must be contained
inside
$\Lambda \setminus\L_n$ and
\[
|\Lambda\setminus \L_n|=(1+o(1))L^2 \ell^2_c(\l^{(n)})(\ell^2_\t-1) \le  \varepsilon_\beta \,\frac{\beta^2}{(\lambda^{(n)})^2} L^2
\]
where
$\varepsilon_\beta\to 0$ as $\b\to \infty$. Here we used the fact that $\ell_c(\lambda^{(n)})\sim 2\beta/\lambda^{(n)}$ for $\beta\to\infty$ and the fact that
$\lim_{\beta\to \infty}\ell_\t=1$ because, in the same limit, the
Wulff shape becomes a square.


Now a closed contour $\g$ with $\Lambda_\gamma\subseteq  \Lambda \setminus\L_n$  satisfies
\[
|\L_\g|\le \left(\frac{|\g|^2}{16}\right)^{1/2}\left(|\L\setminus
  \Lambda_n|\right)^{1/2} \le \frac{|\g|}{4} \frac\beta{\lambda^{(n)}}\sqrt{\epsilon_\b}\,
L.
\]
Hence the area term $\l^{(n)} |\L_\g|/L$ appearing in the exponential weight
in~\eqref{e:contourFloorBound} is
negligible compared to the length term $\beta|\g|$ and the probability that there exists any
such macroscopic contour is $O(e^{-c(\log L)^2})$ by a simple counting
argument.  In conclusion, w.h.p.\ $\G_n$ is the unique macroscopic $(H(L)-n)$-contours.
\end{proof}

\bigskip
The proof of Theorem~\ref{th:growth0} part $a)$ will be based on the following argument.
\begin{proposition}\label{pro:growth}
Fix $n\in\bbZ_+$.
Let $1\leq m\leq \log L$ and let $\L'\subset\bbZ^2$ be a
region containing
$L\cL(\l^{(n)},1+\d_L,-(m-1)\g_L)$, $\g_L:=L^{-\frac 23+3\e}\log L$, and consider the SOS model on $\L'$
with $(H(L)-n-1)$-boundary conditions and floor at height
zero. Conditionally on the event $\cE_{n}(L(1+\d_L) \ell_c(\l^{(n)})\cW_1)$,
then $\cE_{n}(L\cL(\l^{(n)},1+\d_L,-m\g_L))$ holds w.h.p.
\end{proposition}

We start with 
the case $n=0$. 
When $n=0$ it is assumed that
$\L'$ contains $A:=L\cL(\l,1+\d_L,-(m-1)\g_L)$
and we condition on the event $\cE_{0}(L(1+\d_L) \ell_c(\l)\cW_1)$
that there is an $H(L)$-contour containing the Wulff body $L(1+\d_L)\ell_c(\l)\cW_1$.
We show that w.h.p.\ this initial droplet grows until it invades the whole region
 $L\cL(\l,1+\d_L,-m\g_L)$.
  The proof is divided into two main steps.

\begin{figure}[h]
\psfig{file=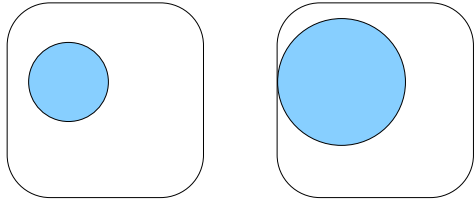,width=0.385\textwidth}
\vspace{-0.1in}
\caption{Growth of the initial droplet as described in Step 1: from $\cW(x,\ell)$ to $\cW(x,\ell_x)$.
}
\label{fig:growth2}
\end{figure}

\subsubsection*{Step 1}
For $x\in A$, $\ell>0$, let $\cW(x,\ell)$ denote the rescaled Wulff shape $L\ell\cW_1$ centered at $x$.
Also, let $\ell_x$ denote the maximal value of $\ell$ such that ${\rm dist}(\cW(x,\ell),A^c)>L^{\frac 13+3\e}$.
The next lemma shows that at any $x\in A$ such that $\ell_x>\ell_c(\l)$ one can let an initial
droplet $\cW(x,\ell)$, $\ell_c(\l)<\ell<\ell_x$, grow until it touches the boundary of $A$ up to $O(L^{\frac 13+3\e})$; see Figure~\ref{fig:growth2}.
\begin{lemma}\label{step1}
Fix $x\in A$ and $\ell_c(\l)(1+\d_L)\leq \ell < \ell_x$. Conditionally on $\cE_0(\cW(x,\ell))$,
then $\cE_0(\cW(x,\ell_x))$ holds w.h.p.\
\end{lemma}
\begin{proof}[Proof of Lemma~\ref{step1}]
By simple recursion, it suffices to show that
conditionally on $\cE_0(\cW(x,\ell))$,
then $\cE_0(\cW(x,\ell'))$ holds w.h.p., with $\ell'=\ell(1+L^{-\frac 23})$, as long as $\ell'\leq \ell_x$.
Next, we shall use the growth gadget of Theorem~\ref{th:napalla} along the boundary of $\cW(x,\ell)$ to show that
conditionally on $\cE_0(\cW(x,\ell))$, then w.h.p.\
there is a circuit $\cC$ surrounding $\cW(x,\ell')$ such that $\eta_y\geq H(L)$ for all $y\in\cC$.
The latter event implies $\cE_0(\cW(x,\ell'))$.

By symmetry, we may restrict our analysis to the north-west corner of the droplet $\cW(x,\ell)$.
Moreover, using symmetry w.r.t.\ reflections along the north-west diagonal we may restrict to
upper half of the north-west corner.
Let $\theta\in[0,\pi/4]$ and consider the chord of $\cW(x,\ell)$ forming an angle $\theta$ with the $x$ axis and whose horizontal projection has length $L^{\frac 23+\e}$. Let $z=(x_z,y_z)$ be the midpoint of this chord
and call $(x_a,y_a)$, and $(x_b,y_b)$ the intersection points of the chord with $\partial \cW(x,\ell)$,
the boundary of  $\cW(x,\ell)$.
From a natural rescaling of the function $\D(d,\theta)$ appearing in Lemma~\ref{th:wulff_height}, one finds  that the vertical distance $\D_0$ from $z$ to $\partial \cW(x,\ell)$ is given by
\begin{equation*}
\D_0=\ell L\D\Big(\frac{L^{\frac 23 + \e}}{\ell L\cos(\theta)},\theta\Big)
=\frac{w_1}{16(\t(\theta)+\t''(\theta))(\cos(\theta))^3}\,\frac{L^{\frac 13+2\e}}{\ell}.
\end{equation*}
Since
$\b w_1/2 = \l \ell_c(\l)$, for any $\l>0$, one can rewrite
\begin{equation}\label{delta0}
\D_0
= \frac{\l\ell_c(\l)/\ell}{8\b(\t(\theta)+\t''(\theta))(\cos(\theta))^3}\,L^{\frac 13+2\e}.
\end{equation}
Consider the rectangle
$Q$ with horizontal side $L^{\frac 23+\e}$ and vertical side $2L^{\frac 23+\e}$ centered at the point $w=(x_w,y_w)$ such that $x_w=x_z$ and $y_w=y_b+L^{\frac 13+3\e}-L^{\frac 23+\e}$, and let $\tilde Q$ denote the enlarged rectangle with the same center, horizontal side $L^{\frac 23+\e}+4(\log L)^2$
and vertical side $2L^{\frac 23+\e}+4(\log L)^2$.  Observe that the assumption that $\ell\leq \ell_x$, or equivalently ${\rm dist}(\cW(x,\ell),A^c)>L^{\frac 13+3\e}$, guarantees that the rectangles $Q,\tilde Q$ are indeed contained in our region $\L'\supset A$.

Notice that,
setting $a=(y_a-y_w)$, $b=(y_b-y_w)$, one has $-\frac12 L^{\frac 23+\e}\leq a\leq b\leq L^{\frac 23+\e}- L^{\frac 13+3\e}$, as required in Theorem~\ref{th:napalla} (1). To ensure that we can indeed apply that statement we now check that w.h.p.\ there exists a  regular circuit $\cC_*$ in $\tilde Q\setminus Q$
with the required properties, namely that one has w.h.p.: 1) heights at least $H(L)-1$ in the upper path along $\cC_*$ connecting $A$ and $B$ and 2) heights at least $H(L)$ in the lower path along $\cC_*$ connecting $A$ and $B$, where $A,B$ are defined in Definition~\ref{bcc}; see Figure \ref{fig:ghiro}.
Point 1) follows  from Lemma~\ref{quasi-rect.1} and the fact that we have b.c.\ $H(L)-1$.
Point 2) follows again from Lemma~\ref{quasi-rect.1} and, via the usual monotonicity and conditioning argument, from the assumption that $\cE_0(\cW(x,\ell))$ holds.
Then, an application of Theorem~\ref{th:napalla} (1), together with monotonicity, 
shows that the point $v=(x_v,y_v)$ with $x_v=x_w$ and $y_v=y_w+K$, with
$$
K=\frac{a+b}2 + \frac{\l L^{\frac 13+2\e}}{8\b(\t(\theta)+\t''(\theta))(\cos(\theta))^3} - c(\b,\theta)L^{\frac 13+\e},
$$
for a suitable constant $c(\b,\theta)>0$, lies w.h.p.\ below a chain $\cC(v)$ connecting $A$ and $B$ with $\eta_y\geq H(L)$ for all $y\in\cC(v)$. Call $\cF(v)$ this event.
Next, observe that the point $v$ lies above $\partial\cW(x,\ell)$ and has a vertical distance $h$ at least $L^{\frac 13+\e}$ from $\partial\cW(x,\ell)$. Indeed, using~\eqref{delta0}
\begin{align*}
h=K-\frac{a+b}{2} - \D_0 = \frac{\l(1-\ell_c(\l)/\ell)}{8\b(\t(\theta)+\t''(\theta))(\cos(\theta))^3}\,L^{\frac 13+2\e}
-c(\b,\theta)L^{\frac 13+\e}\geq L^{\frac 13+\e},
\end{align*}
where we use the assumption $1-\ell_c(\l)/\ell\geq \d_L$ and we take $L$ large enough.
In particular, it follows that $v$ lies outside the enlarged shape $\cW(x,\ell')$, $\ell'=\ell(1+L^{-\frac 23})$ since this is larger than $\cW(x,\ell)$ by an additive $O(L^{\frac 13})$ only.

Repeating the above argument for all $\theta\in [0,\pi/4]$ (of course $O(L)$ values of $\theta$ in this range suffice) and using symmetry to cover the other corners of the droplet, considering the intersection of all events $\cF(v(\theta))$, one  has that w.h.p.\
there exists a chain $\cC$ surrounding $\cW(x,\ell')$ such that $\eta_y\geq H(L)$ for all $y\in\cC$ as desired.
\end{proof}

\subsubsection{Step 2}
 By  assumption we can pretend that $\cE_0(\cW(x_0,\ell))$ holds, where
 $x_0$ is the center of the region $A$, which we identify with the origin.
Thus, by Step 1, one has that $\cE_0(\cW(x_0,\ell_{x_0}))$ holds w.h.p.
Next, we establish that this is enough to invade the whole region $L\cL(\l^{(n)},1+\d_L,-m\g_L)$.
\begin{lemma}\label{step2}
Conditionally on $\cE_0(\cW(x_0,\ell_{x_0}))$,
$\cE_0(L\cL(\l^{(n)},1+\d_L,-m\g_L))$ holds w.h.p.
\end{lemma}

\begin{figure}[h]
\psfig{file=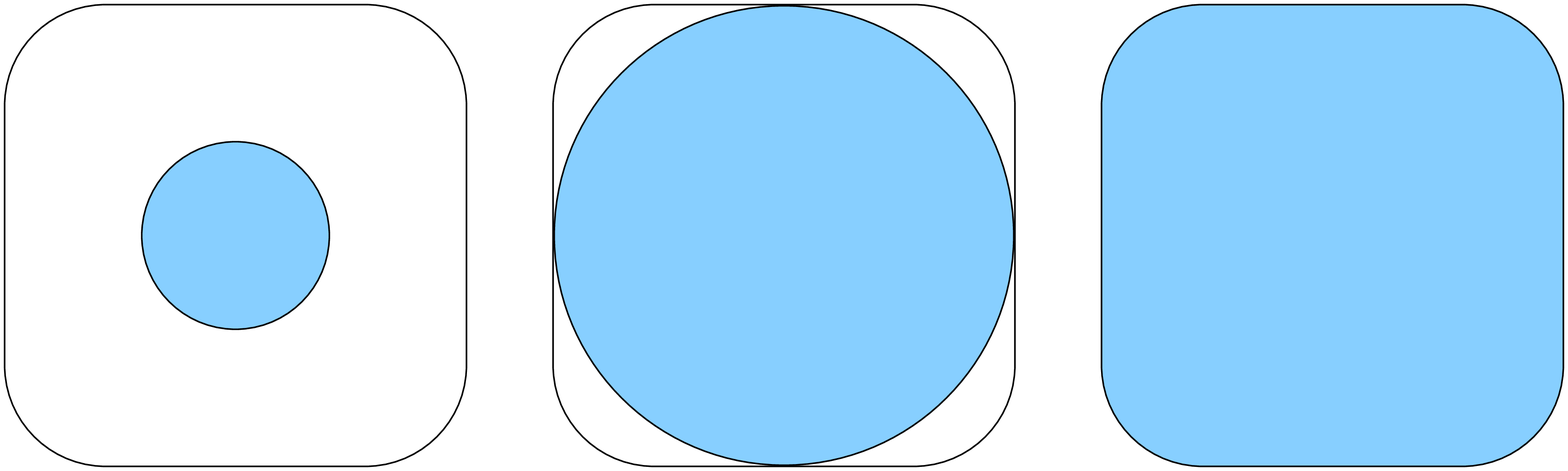,width=0.55\textwidth}
\vspace{-0.1in}
\caption{Growth of
the initial droplet as described in Step 2: from $\cW(x_0,\ell)$ to $\cW(x_0,\ell_{x_0})$ and from $\cW(x_0,\ell_{x_0})$ to $L\cL(\l^{(n)},1+\d_L,-m\g_L)$.
}
\label{fig:growth3}
\end{figure}

Before proving Lemma~\ref{step2}, we need the following deterministic lemma concerning the enlargement of squeezed Wulff shapes. Fix $\l>0$ and $\ell_c(\l)<\ell<1/\ell_\t$. The Wulff body $\ell\cW_1$ is strictly contained in the unit square $Q$; see Section~\ref{sec:Notation} for the notation. Setting $\ell_*=1/\ell_\t$ one has that $\cD_0:=\ell_*\cW_1$ is tangent to all four sides of $Q$, i.e.\ it is the maximal Wulff shape inside $Q$. For any $\z\in \cD_0$ such that  $\ell\cW_1 + \z\subset \cD_0$, define $$t_\z=\max\{t\geq1: \; t\ell\cW_1 + \z\subset Q\},$$
with the convention that $t_\z=0$ if there is no such $t$.
We define $\cD_1=\cup_{\z\in \cD_0}\Big\{t_\z\ell\cW_1 + \z\Big\}$. We then repeat the above enlargement procedure. Namely, given the set $\cD_k$, we define $$
\cD_{k+1}=\cup_{\z\in \cD_{k}}\Big\{t_\z\ell\cW_1 + \z\Big\}\,,$$
where $t_\z=\max\{t\geq1: \; t\ell\cW_1 + \z\subset Q\}$ for $\z\in\cD_k$ with $\ell\cW_1 + \z\subset\cD_k$ and with $t_\z=0$ if $\ell\cW_1 + \z\not\subset\cD_k$. The sequence $\{\cD_{k}\}_k$ consists of nested convex subsets of $Q$.
\begin{lemma}\label{enlarge}
The sequence $\cD_k$ converges to $\cD_{\infty}:=\cL(\l,\ell/\ell_c(\l),1)$. Moreover,
the Hausdorff distance between $\partial\cD_k$ and
$\partial\cD_\infty$ is upper bounded by $c^k$ for some constant $c\in(0,1)$.
\end{lemma}
\begin{proof}
The set $\cD_k$ has four symmetric flat pieces where it is tangent to the sides of $Q$. Let $v_k$ denote the length of one flat piece and write $r_k=(1-v_k)/2$. Moreover, notice that $2r_k$ is the side of the smallest square one can put around the Wulff body $s_k\cW_1$, with $s_k=2r_k/\ell_\t$.  Simple geometric considerations then show that
the sequence $r_k$ satisfies
$$
r_{k+1} = r_k\left(1-\sqrt 2 y/\ell_\t\right) + \frac\ell{\sqrt 2}\,y,\qquad r_0=\frac12,
$$
where $y$ is the radius of the Wulff body $\cW_1$ in the direction $\theta=\pi/4$.
Set $a=1-\sqrt 2 y/\ell_\t$ and note that $a\in(0,1)$.
It follows that $r_k= \frac12 a^k +  \frac\ell{\sqrt 2}\,y\sum_{j=0}^{k-1}a^j$.  As $k\to\infty$, this converges to $\ell\ell_\t/2$ which is the value corresponding to the limiting shape $\cL(\l,\ell/\ell_c(\l),1)$. The Hausdorff distance between
$\partial \cD_k$ and $\partial\cD_{k+1}$
is then of order $a^k$ and the desired conclusion follows.
\end{proof}

\bigskip

\begin{proof}[Proof of Lemma~\ref{step2}]
Consider the sets $\cD_k$, $k=0,1,\dots$ defined above. By assumption we know that $L(1-(m-1)\g_L)\cD_0
\sim\cW(x_0,\ell_{x_0})$ is contained w.h.p.\ in an $H(L)$-contour.
We now prove that conditionally on $\cE_0(L(1-(m-1)\g_L-kL^{-\frac 23+3\e})\cD_k)$, then
$\cE_0(L(1-(m-1)\g_L-(k+1)L^{-\frac 23+3\e})\cD_{k+1})$ holds
w.h.p.\

Fix $\ell=\ell_c(\l_0)(1+\d_L)$, and consider a droplet
$\cW(x,\ell)$ such that $\cW(x,\ell)\subset \cD_k$.
From Lemma~\ref{step1}, we can let $\cW(x,\ell)$ grow up to $\cW(x,\ell_x)$. Repeating this at every $x$ as above yields the desired claim since the parameter $t_\z$ in the definition of $\cD_k$ can be identified with $\ell_{x}/\ell$ for $x=\z L$. This establishes that under the assumptions of Lemma~\ref{step2}, for any $k$, the set
$L(1-(m-1)\g_L-kL^{-\frac 23+3\e})\cD_{k}$ is contained w.h.p.\ in an $H(L)$-contour.
From Lemma~\ref{enlarge}, we know that a number $k=O(\log L)$ of steps suffices to attain a distance of order $1/L$ between $\cD_k$ and $\cD_\infty$, and therefore w.h.p.\ $L(1-m\g_L)\cD_{\infty}$ is contained in an $H(L)$-contour.
This proves Lemma~\ref{step2}.
\end{proof}

\bigskip

\begin{proof}[Proof of Proposition~\ref{pro:growth}]
The above two steps provide a proof in the case $n=0$.
The other cases are obtained with exactly the same argument, provided one uses $\l^{(n)}$ instead of $\l^{(0)}=\l$.
\end{proof}
\begin{proof}[Proof of Theorem~\ref{th:growth0} part $a)$]
Fix $n\in\bbZ_+$. 
Suppose that
the event
\begin{equation}\label{eventon}
\cE_{n+1}(L\cL(\l^{(n)},1+\d_L,-(H(L)-n-1)\g_L))
\end{equation}
holds w.h.p.
On the latter event,
conditioning on the outermost
$(H(L)-n-1)$-contour $\G$ 
and using monotonicity, one can assume that there are b.c.\ at height $H(L)-n-1$ out of some region $\L'$ which contains
the set $L\cL(\l^{(n)},1+\d_L,-(H(L)-n-1)\g_L)$. It follows from Proposition~\ref{pro:growth}
that w.h.p.\ $\cE_{n}(L\cL(\l^{(n)},1+\d_L,-(H(L)-n)\g_L)$ holds. Since $(H(L)-n)\g_L\leq (\log L)^2L^{-\frac 23+3\e}\leq L^{-\frac 23+4\e}$ the desired conclusion follows.

Thus, it suffices to prove that~\eqref{eventon} holds w.h.p.\ assuming~\eqref{evento}.
We use recursion, starting
from the case of the $1$-contour. Here one has $0$ b.c.\ outside the $L\times L$ square $\L$ and floor at $0$. By monotonicity one can lower the floor down to height $-(H(L)-n-1)$. Once this is done the statistics of the $1$-contours coincides with the statistics of the $(H(L)-n)$-contours
with floor
at $0$ and b.c.\ at $(H(L)-n-1)$. By Proposition~\ref{pro:growth}, with $m=1$, one infers that there is a $1$-contour in the original problem that contains $L\cL(\l^{(n)},1+\d_L,-\g_L)$.
Recursively, assume that w.h.p.\ there exists  a $k$-contour containing $L\cL(\l^{(n)},1+\d_L,-k\g_L)$. Conditioning on the outermost such contour, using monotonicity one can assume b.c.\ $k$ on a set $\L'$ that contains $L\cL(\l^{(n)},1+\d_L,-k\g_L)$.
Repeating the above argument (lowering the floor and using Proposition~\ref{pro:growth}) one has that w.h.p.\ there exists  a $(k+1)$-contour containing $L\cL(\l^{(n)},1+\d_L,-(k+1)\g_L)$.
Once we reach the height $k=H(L)-n-1$ the proof is complete.
 \end{proof}
\subsection{Retreat of droplets}
We recall that, from Section~\ref{US-section}, w.h.p a  macroscopic $(H(L)-n)$-contour exists iff $\l^{(n)}> \l_c$. In that case it is unique w.h.p.\ by Theorem~\ref{th:growth0}(b).
\begin{theorem}\label{th:retreat0}
Fix $\epsilon,\hat \epsilon\in(0,1/10)$, let $\L$ be the square of side $L$. 
Consider the SOS model on $\L$
with $0$-boundary conditions and floor at height
zero. Fix $n \in \bbZ_+$ and assume $\l^{(n)}\ge \l_c+\hat\epsilon$. Then w.h.p. as $L\to \infty$ the unique macroscopic
$(H(L)-n)$-contour is contained in
$L\cL(\l^{(n)}, t_L,\d_L)$ (cf.\ Definition~\ref{Lstorto}) with
$t_{L}=1-\d_L$ and $\d_L=L^{-\epsilon/8}$.
\end{theorem}
\begin{proof}\
\vskip 0.2cm
\noindent
{\bf (i)} We begin by treating the base case $n=0$. The case
$n\ge 1$ will then follow by a simple induction.

\begin{definition}
\label{cH}
Given $s\in [0,1]$ we say that $\cH(s)$ holds if w.h.p. there exists a unique
macroscopic $H(L)$-contour $\G_0$ and it is contained in
  $L\cL(\l,s, \d_L)$.
\end{definition}
With this definition the statement of the theorem for $n=0$ follows
from the next two Lemmas.
\begin{lemma}[Base case]
\label{basecase1}
For any $s$ small enough $\cH(s)$ holds.
\end{lemma}
\begin{lemma}[Inductive step]
\label{rinductive}Fix $s\le t_L$. Then $\cH(s)$
implies $\cH(s+ L^{-\frac 23})$.
\end{lemma}
\begin{proof}[Proof of Lemma~\ref{basecase1}]
We actually prove that w.h.p.\ $\G_0$ is contained in $L\cL(\l,s,0)$
for $s$ small enough. Fix $s\in (0,1/4)$ and define $T_i$, $i=1,\dots,4$,
as the ``curved triangle'' delimited by the curved
portion of the boundary of $L\cL(\l,4s,0)$ facing the
$i^{th}$-corner $v_i$ of $\L$ and $\partial \L$. Let $A,B$ be the end points of the curved portion of the boundary of
$T_1$ (both at distance $2s\ell_c\ell_\t L$ from $v_1$)
Let now $E_i$ be the event that inside $T_i\setminus L\cL(\l,2s,0)$ there exists a macroscopic chain where the height of the surface is at least
$H(L)$. If the macroscopic $H(L)$-contour $\G_0$
--- which, under $\pi_\L^0$, exists w.h.p.\ by Proposition~\ref{p:bigHContour} and is unique by Theorem~\ref{th:growth0} --- is not contained in
$L\cL(\l,s,0)$, then necessarily one of the four events $E_i$
occurred. By symmetry, it is therefore enough to show that
$\pi_\L^0(E_1)=O(e^{-c(\log L)^2})$ for any $s$ small enough.

For this purpose let us introduce boundary conditions $\t$ on
$\partial \L$ as follows:
\begin{equation}
  \label{b.c}
\tau_x=
  \begin{cases}
    H(L)& \text{if $x\in \partial \L\setminus \partial T_1$}\\
H(L)-1 & \text{if $x\in \partial \L\cap \partial T_1$}.
  \end{cases}
\end{equation}
By monotonicity we can bound from above $\pi_\L^0(E_1)$ by
$\pi_{\L}^\t(E_1)$.
Call $\G$ the open $H(L)$-contour $\G$ joining $A,B$ and denote by $G$
the event that $\G$ does not get out of $L\cL(\l,2s,0)$. We can once
again appeal to \cite{CLMST}*{Lemma A.2} to get that
\[
\pi_{\L}^\t(E_1\tc G)\le e^{-c(\log L)^2}.
\]
Thus we are left with the proof that, for all $s$ small enough, $G$
occurs w.h.p. As in Lemma~\ref{Law of Gamma} we can write
\begin{equation}
  \label{eq:2}
\pi_{ \L}^\t(\G)\propto \exp\left(-\b |\G| +
\Psi_{ \L}(\G)+\frac{\l}{L}A(\G) +o(L)\right)
\end{equation}
where $A(\G)$ is the signed area of $\G$ w.r.t.\ the segment $AB$
with the obvious choice of the signs. Clearly $A(\G)\le 2s^2 \ell_c^2
\ell_\t^2 L^2\le s^2L^2$ for $\b$ large enough. 
Thus
\begin{align}
  \pi_{\L}^\t(G^c)&\le e^{s^2 (L+o(L))}\ \frac{\sum_{\G\in G^c}e^{-\beta
      |\G|+\Psi_{\L}(\G)}}{\sum_{\G}e^{-\beta
      |\G|+\Psi_{\L}(\G)+\frac{\l}{L}A(\G)}}\nonumber \\
&\le e^{s^2 (L+o(L))}\ \frac{\sum_{\G\in G^c}e^{-(\beta-e^{-\b})
      |\G|+\Psi_{\bbZ^2}(\G)}}{\sum_{\G}e^{-(\beta + e^{-\b})
      |\G|+\Psi_{\bbZ^2}(\G)+\frac{\l}{L}A(\G)}}
\nonumber \\
&= e^{s^2 (L+o(L))}\ \frac{\sum_{\G\in G^c}e^{-(\beta-e^{-\b})
      |\G|+\Psi_{\bbZ^2}(\G)}}{\sum_{\G}e^{-(\beta - e^{-\b})
      |\G|+\Psi_{\bbZ^2}(\G)}}\times \frac {\sum_{\G}e^{-(\beta - e^{-\b})
      |\G|+\Psi_{\bbZ^2}(\G)}}{\sum_{\G}e^{-(\beta + e^{-\b})
      |\G|+\Psi_{\bbZ^2}(\G)+\frac{\l}{L}A(\G)}}
\label{eq:step0}
\end{align}
where we used $|\Psi_{\L}(\G)-\Psi_{\bbZ^2}(\G)|\le e^{-\b}|\G|$.
Using \cite{DKS} the first ratio in the r.h.s.\ of~\eqref{eq:step0} is
bounded from above by $\exp(-c s L)$ with $c$ independent of $\beta$, since the event
$G^c$ implies an excess length of order $sL$ for the contour.
Using Jensen's inequality w.r.t.\ the measure $\nu$ on $\G$ corresponding to
the weight $e^{-(\beta - e^{-\b})
      |\G|+\Psi_{\bbZ^2}(\G)}$, the second ratio is bounded from above
    by
\[
\exp\left(2e^{-\b}\nu(|\G|)+ \frac{\l}{L}\nu(A(\G))\right).
\]
Using once again \cite{DKS} we have $\nu(|\G|)\le 2sL$ and $\nu(A(\G))=O((sL)^{3/2})$.
Hence
$\pi_{\L}^\t(G^c)=O(e^{-c sL/2})$ for any $s$ small enough
independent of $L$.
\end{proof}
\begin{proof}[Proof of Lemma~\ref{rinductive}]
Let us fix some notation. Referring to Figure~\ref{fig:fabio}
and
centering the box $\L$ in the origin,  let
\[
f_s:\bigl(-\frac{(1+\d_L)}{2}L,0\bigr]\mapsto \bigl[-\frac{(1+\d_L)}{2}L,0\bigr]
\] be the
decreasing convex function whose graph is the
South-West quarter of $\partial (L\cL(\l,s,\d_L))$. Let $\hat x(s)$ be the unique solution of
$f_s(x)=x$.
In the sequel we will denote by $x_*(s)$ (resp. $ x^*(s)$) the point
after which $f_s(\cdot)$ is smaller than $-L/2$ (resp. after which $f_s(\cdot)$
is flat and equals $-\frac L2(1+\delta_L)$).
\begin{figure}[h]
  \centering
\begin{overpic}[scale=0.4]
{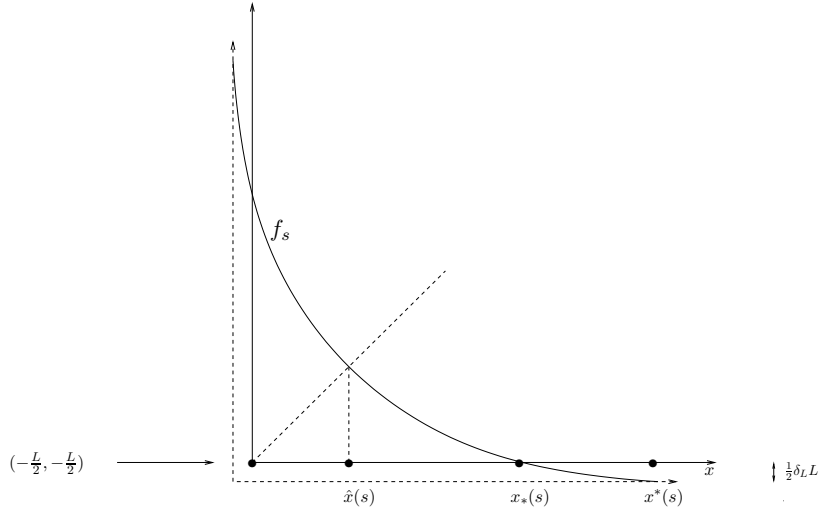}
\put(23,40){\scalebox{.8}{$f_s$}}
\put(19.5,5){\tiny $\bullet$}
\put(59,0){\scalebox{.6}{$x_*(s)$}}
\put(34,0){\scalebox{.6}{$\hat x(s)$}}
\put(59.5,5){\tiny $\bullet$}
\put(34,5){\tiny $\bullet$}
\put(79,0){\scalebox{.6}{$x^*(s)$}}
\put(-16, 5){\scalebox{.6}{$(-\frac L2, -\frac L2)$}}
\put(88,4){\scalebox{.6} {$x$}}
\put(79.5,5){\tiny $\bullet$}
\put(100,4){\scalebox{.5}{$\frac{1}{2} \d_L L $}}
\end{overpic}
  \caption{The graph of the function $f_s$ describing the South-West quarter of $\partial (L\cL(\l,s,\d_L))$.}
\label{fig:fabio}
\end{figure}


For $x\in [\hat x(s),x_*(s')]$ let $x^\pm:= x\pm
\frac 12 L^{\frac 23 +\epsilon}$ and define
\begin{equation}
\label{eq:f}
Z_s(x)=\frac 12 [f_s(x^-)+f_s(x^+)]- \frac{\l}{8\b
  (\tau(\theta)+\tau''(\theta))\cos(\theta)^3}L^{\frac 13 +2\epsilon}
-\sigma(x,\theta)L^\epsilon
\end{equation}
with $\theta=\theta_x\in [0,\pi/4]$
such that $\tan(\theta)=|f'_s(x)|$ and $\sigma^2(x,\theta)=O(L^{\frac
  23 +\epsilon})$ is given
in Theorem~\ref{th:napalla}.

We first observe that, if $s\le t_L$,
\begin{equation}
  \label{eq:Y}
Z_s(x)\ge f_{s'}(x)
\end{equation}
where the r.h.s.\ is larger than $-L/2$ because $x\le x_*(s')$.
Indeed,
\[
x^*(s)-x_*(s')\ge c \sqrt{\d_L} L \quad \text{so that}  \quad x^+ < x^*(s)
\]
and, using Lemma~\ref{th:wulff_height} together with simple trigonometry,
\begin{eqnarray}
  \label{eq:1}
  f_s(x)=\frac 12 [f_s(x^-)+f_s(x^+)]-\frac 1{s(1+\d_L)}
  \left[\frac{\l}{8\b\,(\tau(\theta)+\tau''(\theta))\cos(\theta)^3}L^{\frac13+2\epsilon}\right]
  +o(1).
\end{eqnarray}
Moreover,
\begin{eqnarray*}
  f_{s'}(x)\le f_s(x)+c(s'-s)L=f_s(x)+c \,L^{\frac13}
\end{eqnarray*}
for some positive constant $c$.
Therefore, if $s\le (1-\delta_L)=(1-L^{-\epsilon/8})$,
\begin{align*}
Z_s(x)-f_{s'}(x) &\ge  \frac{\l}{8\b(\tau(\theta)+\tau''(\theta)
  )\cos(\theta)^3}L^{\frac13+2\epsilon}\left(\frac1{s(1+\d_L)}-1\right)-cL^{\frac13}\\
&\ge \frac{\l}{8\b(\tau(\theta)+\tau''(\theta)
  )\cos(\theta)^3}L^{\frac13+2\epsilon} \d^2_L -cL^{\frac13}
\end{align*}
which is positive.

We are now going to apply the results of Theorem \ref{th:napalla} (see Definitions~\ref{bcc}, \ref{reg-circuit} and Figure \ref{fig:ghiro}).
Consider the rectangle $Q_x$
of horizontal side $L^{\frac 23 +\epsilon}$ and vertical side $2L^{\frac 23 +\epsilon}$ centered at~$\left(x,\frac 12
  [f_s(x^-)+f_s(x^+)]\right)$.
    For simplicity we initially  assume that $x\in [\hat x(s),x_*(s')]$
is such that $Q_x\subset \L$. Later on we will explain how to treat the
general case.
Let $\tilde Q_x$ denote the $2(\log L)^2$ neighborhood of
  $Q_x$ and let $G_x$ be the
  event  that there exists a regular circuit $\cC_*\in \tilde
  Q_x\setminus Q_x$ such that the height
  $\eta\restriction_{\cC_*}$ on $\cC_*$ is not
  higher than the height $\xi(\cC_*,j,a,b)$ given in Definition~\ref{bcc} with
$a=f_s(x^-)-(\log L)^2$,  $b=f_s(x^+)-(\log L)^2$, $j=H(L)$. Here the
roles of $j,j-1$ have been interchanged w.r.t.\ the setting of Theorem~\ref{th:napalla}, i.e.\ in the present application the
b.c.\ are $j$ above $A,B$ and $j-1$ below.
Define $E_x$ as the event that there exists  a chain $\cC_x$
of lattice sites satisfying the following conditions:
\begin{enumerate}[(i)]
\item $\cC_x$ connects the points $A,B$; 
\item the point $(x,Z_s(x))$ lies below $\cC_x$;
\item  $\eta_y\le H(L)-1$ for all $y\in \cC_x$.
\end{enumerate}
\begin{claim}
\label{claim:1}
W.h.p. the event $E_x$ occurs. 
\end{claim}

Before proving the claim let us conclude the proof of Lemma ~\ref{rinductive}.
If $E_x$ occurs for all $x\in [\hat x(s),x_*(s')]$, then necessarily there exists a chain $\hat\cC$ (obtained by patching together the individual chains $\cC_x$) joining the vertical lines through the points with horizontal coordinates $\hat x(s)-\frac12 L^{\frac 23 +\epsilon}$ and $x_*(s')+\frac12 L^{\frac 23 +\epsilon}$ and staying above the curve $\cZ_s:=\{(x,Z_s(x));\  x\in [\hat x(s),x_*(s')]\}$ where the surface height is at most $H(L)-1$. Notice that \eqref{eq:Y} implies that
$\cZ_s$ is above the curve $f_{s'}$ on the same interval $[\hat x(s),x_*(s')]$. Using the claim the above event occurs w.h.p.
Using symmetry
w.r.t.\ reflection across the South-West diagonal of $\L$ the corresponding event occurs in the upper half of the South-West corner of $\L$.
By patching together the two chains constructed in this way, we have shown that w.h.p.\ the left vertical boundary of $\L$ and the bottom horizontal   boundary of $\L$ are connected by a chain of sites $\cC$ which stays above the curve $f_{s'}$ such that $\eta_x\leq H(L)-1$, for all $x\in\cC$. Since we are assuming $\cH(s)$ we know that w.h.p. there exists a unique $H(L)$-contour $\G_0$.
The contour $\G_0$ cannot cross $\cC$ so that either $\cC$ is contained in $\L_{\G_0}$, the interior of $\G_0$, or it is contained in $\L\setminus \L_{\G_0}$. The first case can be excluded since it would produce a macroscopic negative contour in $\L$, which has negligible probability by Proposition~\ref{bdgma}. The second case implies that the curve $f_{s'}$ lies outside of $\G_0$.
The same argument can be repeated for the remaining three corners of the box $\L$. That implies that w.h.p.\ the macroscopic $H(L)$-contour $\G_0$ is contained inside $\cL(\l,s',\d_L)$.
\end{proof}

\begin{proof}[Proof of Claim \ref{claim:1}]
%
To compute the probability of the event $G_x$ defined above call $\O_s$ the event that the unique
macroscopic $H(L)$-contour $\G_0$ is contained in
  $L\cL(\l,s,\d_L)$. Since we assume $\cH(s)$,  the event $\O_s$ occurs
  w.h.p. On the other hand, conditionally on
  $\O_s$, $G_x$ occurs w.h.p.\ (cf.\ Lemma~\ref{quasi-rect.2}). In
  conclusion, $G_x$ occurs w.h.p.
We will denote by $\cC_*$ the most external circuit characterizing the
event $G_x$.\footnote{Of all this complicated construction the reader should
just keep in mind the following simplified picture:
$\cC_*=\partial Q_x$ and the height of the surface is at most $H(L)-1$ on the
portion of $\partial Q_x$ below the curve $f_s(\cdot)$ and at least $H(L)$ above it.}

Putting all together we have
\begin{align*}
 \pi_\L^0\left(E_x\right)
&\ge \pi_\L^0\left(E_x\tc G_x\right) - O(e^{-c(\log L)^2}) \\
&\ge \min_{\cC_*}\pi_\L^0\left(E_x\tc \eta\restriction_{\cC_*}\le \xi(\cC_*,j,a,b)\right) - O(e^{-c(\log L)^2})
\end{align*}
with $(j,a,b)$ as above and the minimum is taken over all possible
regular circuits in $\tilde Q_x\setminus Q_x$. Since the event
$E_x$ is decreasing, monotonicity gives
\[
\min_{\cC_*}\pi_\L^0\left(E_x\tc \eta\restriction_{\cC_*}\le
  \xi(\cC_*,j,a,b)\right)\ge
\min_{\cC_*}\pi_\L^0\left(E_x\tc \eta\restriction_{\cC_*}=\xi(\cC_*,j,a,b)\right).
\]
At this stage we are exactly in the setting of
Theorem~\ref{th:napalla} which states that the open $H(L)$-contour inside the
region enclosed by $\cC_*$ and induced by the boundary conditions
$\xi(\cC_*,j,a,b)$ goes above the point $(x,Y_s(x))$ w.h.p. The latter
event implies the event $E_x$ by the very definition of the
$H(L)$-contour.
In conclusion
\[
\min_{\cC_*}\pi_\L^0\left(E_x\tc
  \eta\restriction_{\cC_*}=\xi(\cC_*,j,a,b)\right)=1-O(e^{-c(\log L)^2})
\]
and $E_x$ occurs w.h.p.
\end{proof}

It remains to consider the case where the rectangle $Q_x$ exits the lower
side of $\L$ which happens if $x$ is close to $x_*(s')$. In this case, we repeat the same reasoning, except that in order to estimate from below
\[
\pi_\L^0\left(E_x\tc
  \eta\restriction _{\cC_*}\le \xi(\cC_*,j,a,b)\right),
\]
we use the domain enlarging
procedure
of Remark~\ref{rem:mallargo} and we replace $\Lambda$ with
$\L'=\Lambda\cup Q_x$, again with zero boundary conditions on
$\partial \L'$. The proof then proceeds identically as before and in this
case by construction the regular circuit $\cC_*$ coincides with
$\partial Q_x$ in the portion of $\partial Q_x$ that exits $\L$.

\vskip 0.2cm
\noindent{\bf (ii)} Here we shortly discuss the case $n\ge 1$.
As before we denote by $\G_n$ the (unique w.h.p.) macroscopic
$(H(L)-n)$-contour.
Assume inductively that
\begin{equation}
  \label{eq:annulus}
L \cL(\l_{n-1},t^+_L,-\d_L)\subset \G_{n-1}\subset L\cL(\l_{n-1},
t^-_L,\d_L),\quad \text{w.h.p.}
\end{equation}
where $t_L^\pm=1\pm \d_L$, $\d_L=L^{-\epsilon/8}$ and $\l_0=\l$. Notice
that the first inclusion has been proved in Theorem~\ref{th:growth0}.

For shortness denote
by $\cG_{n-1}$ the set of all possible realizations of $\G_{n-1}$
satisfying~\eqref{eq:annulus} and, for each $\G\in \cG_{n-1}$, denote by $V_\G$ the interior of $\G$.
In the base case $n=1$ Claim~\eqref{eq:annulus}  follows from point
{\bf (i)} together with
Theorem~\ref{th:growth0}. Define $\cH^{(n)}(s)$ exactly as
$\cH(s)$ in Definition~\ref{cH} but with the macroscopic
$H(L)$-contour $\G_0$ replaced by $\G_n$.
In order to get the analog of Lemma~\ref{basecase1} for
$\cH^{(n)}(s)$, i.e., that $\cH^{(n)}(s)$ holds for $s$ small enough, we write
\begin{gather*}
\pi_\L^0\left(\G_n\nsubseteq L\cL(\l^{(n)}, s,0)\right)\le \max_{\G\in \cG_{n-1}}
 \pi_\L^0\left(\G_n\nsubseteq L\cL(\l^{(n)}, s,0) \tc \G_{n-1}=\G\right) +O(e^{-c(\log L)^2}).
\end{gather*}
By monotonicity
\[
\pi_\L^0\left(\G_n\nsubseteq L\cL(\l^{(n)}, s,0) \tc \G_{n-1}=\G\right) \le
\pi_{\L^{\rm ext}_\G}^\xi\left(\G_n\nsubseteq L\cL(\l^{(n)}, s,0)\right)
\]
where $\L^{\rm ext}_\G=\L\setminus V_\G$ and the boundary conditions
$\xi$ are equal to zero on $\partial \L$ and  equal to $H(L)-n$ on
$\partial \L^{\rm ext}_\G \setminus \partial \L$.
We then proceed as in the proof of Lemma~\ref{basecase1} for the new setting $(\L^{\rm
  ext}_\G ,\xi)$
with the following modifications: \\
(i) in the definition of $A,B,\{T_i\}_{i=1}^4$ and $\{E_i\}_{i=1}^4$ the parameter $\l$
is replaced by $\l^{(n)}$;\\
(ii) in the definition~\eqref{b.c} of the
auxiliary boundary conditions $\t$ the height $H(L)$ is
replaced by the height $H(L)-n$. \\
Thus we get
\[
\pi_{\L^{\rm ext}_\G}^\xi\left(\G_n\nsubseteq L\cL(\l^{(n)}, s,0)\right) \le
4\pi_{\L_\G^{\rm ext}}^{\tau}\left(E_1\right),
\]
where the measure $\pi_{\L_\G^{\rm ext}}^{\tau}$ describes the SOS
model on $\L_\G^{\rm ext}$, with
floor at height zero and boundary conditions $\t$
equal
to
$(H(L)-n-1)$ on $\partial \L_\G^{\rm ext}\cap \partial T_1$ and
equal to $(H(L)-n)$ elsewhere.
Under $\pi_{\L_\G^{\rm ext}}^{\tau}$ w.h.p.\ there is no
macroscopic  $(H(L)-n+1)$-contour inside $\L_\G^{\rm ext}$ (the
argument is as in the proof of part (b) of Theorem~\ref{th:growth0}). We can
therefore again write down the
distribution of the open $(H(L)-n)$-contour joining the points
$A,B$ exactly as in~\eqref{eq:2}. The rest of the proof follows step by step
the proof of Lemma~\ref{basecase1}; one uses the fact that the distance between the
``internal'' boundary $\G$ of $\L^{\rm ext}_\G$ and the segment
$AB$  is proportional to $L$ to disregard the possible ``pinning
interaction'' between the open contour and  $\Gamma$.

Similarly one proves the analog of Lemma
\ref{rinductive} for $\cH^{(n)}(s)$. In conclusion~\eqref{eq:annulus}
follows for $\G_n$ and the induction  can proceed.
\end{proof}

\subsection{Conclusion: proof of Theorem~\ref{mainthm-2}}
\label{cpfth2}
Assume that, along a subsequence $L_k$, $\l(L_k)\to\l_\star$. Consider first the case $\l_\star>\l_c$. Then by
Proposition~\ref{p:bigHContour} one has that the event $\cE_0(L(1+\d_L) \ell_c(\l(L))\cW_1)$ holds w.h.p.
(see also Remark \ref{capp}).
Therefore, from Theorem~\ref{th:growth0} one has that for any $\e_0>0$, the unique macroscopic $H(L)$-contour, say $\G_0$,
contains the region $L(1-\e_0)\cL_c(\l_\star)$ w.h.p.
Similarly, from Theorem~\ref{th:retreat0} one has that $\G_0$ is
contained in the region $L(1+\e_0)\cL(\l_\star)$ w.h.p.
Analogous statements hold for the unique macroscopic $(H(L)-n)$-contours $\G_n$ for all fixed values of $n\in\bbZ_+$
provided we replace $\l_\star$ by $\l_\star e^{4\b n}$.
 Since the nested limiting shapes $\cL_c(\l_\star e^{4\b n})$ converge to the unit square $Q$ as $n\to\infty$, the above statements imply the theorem in the case $\l_\star>\l_c$. The case $\l_\star<\l_c$ uses exactly the same argument, except that here one knows that w.h.p.\ there is no macroscopic $H(L)$-contour by Proposition~\ref{prop:scifondo}, and the results of Theorem~\ref{th:growth0} and Theorem~\ref{th:retreat0} can be applied to the macroscopic $(H(L)-n)$-contours with $n\geq 1$ only. 

\subsection{Proof of Theorem~\ref{mainthm-3}}
\label{pfth3}
Thanks to Proposition~\ref{p:bigHContour}, Theorem~\ref{th:growth0} --- see Remark~\ref{rem:growth 1/3} --- we know that the distance
of the unique macroscopic $H(L)$-contour $\G_0$ from the boundary along the flat piece of the limiting shape is w.h.p.\ not larger than $O(L^{\frac13+\epsilon})$ for any fixed $\epsilon >0$. It remains to prove the lower bound.

Let us call  $I$ the interval $I^{(k)}_\epsilon$ of Theorem
\ref{mainthm-3} and  write $L$ instead of $L_k$.  Let $x_i, i=1,\dots, K$ be a mesh of equally spaced
points on $I$, $x_{i+1}-x_{i}=2L^{2/3-\epsilon}$, with $x_1$
(resp. $x_K$) the left-most (resp. right-most) point in $I$. Note that
$K$ is of order $L^{1/3+\epsilon}$. Let $Q_i, 1\le i\le K$ be the
collection of disjoint squares with side of length $L^{2/3-\epsilon}$
centered at $x_i$ , $R_i=Q_i\cap \L$ be the upper half of $Q_i$ and
$\partial^+ R_i=\partial R_i\cap \L$.
\begin{figure}[h]
\psfig{file=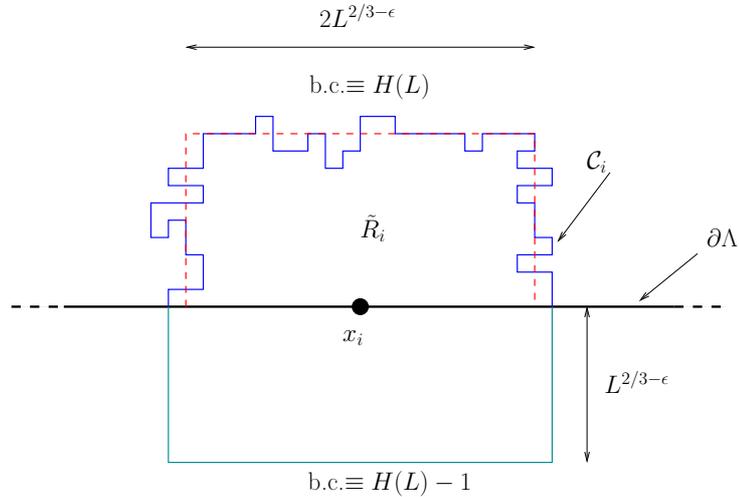,width=0.6\textwidth}
\vspace{-0.1in}
\caption{A point $x_i$, the rectangle $R_i$ (dashed line), the chain $\cC_i$ (wiggled line) enclosing the wiggled rectangle $\tilde R_i$. The thick horizontal line is the boundary of $\L$. The wiggled square $\tilde Q_i$ is obtained by joining to $\tilde R_i$ a rectangle of height
$L^{2/3-\epsilon}$ and width $2L^{2/3+\epsilon}+O((\log L)^2)$, with two corners
coinciding with the endpoints of $\cC_i$. }
\label{fig:mallargo}
\end{figure}

Then, under $\pi^0_\L$, w.h.p.\ the following event $A$ occurs: for
every $1\le i\le K$ there exists a connected chain of sites, within
distance $(\log L)^2$ from $\partial^+ R_i$ and touching $\partial \L$
both on the left and on the right of $x_i$, where the height function
$\eta$ is at most $H(L)$.  On the event $A$, call $\cC_i$ the most
internal such chain and call $\cC=\cup_i \cC_i$.

That $A$ occurs w.h.p.\ is true by
 monotonicity: $A$ is decreasing, so we can lift the boundary conditions
around $\L$ to  $H(L)$
and then apply Lemma~\ref{quasi-rect.2} to deduce that all dominos in $\L$ are
of negative type (see Definition~\ref{def:dominos}, with $j=H(L)$).
The existence of the chains $\cC_i$ then follows easily (a similar argument was
used in the proof of Lemma~\ref{step1}).

We want to show that $\max_{x\in I}\rho(x)\ge L^{1/3-\epsilon/2}$ w.h.p. This
is a decreasing event. Then, condition on the realization of $\cC$ and
by monotonicity lift all heights along $\cC$ to exactly $H(L)$. This
way, the height function in each wiggled rectangle $\tilde R_i$ enclosed by $\cC_i$
is independent. We therefore concentrate on a single $i$.

Again by monotonicity, we can apply the domain enlarging procedure of
Lemma~\ref{rem:mallargo}. For this, we enlarge the domain $\tilde R_i$
outside $\L$ as in Figure~\ref{fig:mallargo}: this way, $\tilde R_i$
has been turned into a wiggled square $\tilde Q_i$ of side
$L^{2/3-\epsilon}+O((\log L^2))$. Then we consider the SOS measure
$\pi^\tau_{\tilde Q_i}$ in $\tilde Q_i$ with floor at zero and
b.c.\ $\tau$ that are $\tau\equiv H(L)$ in $\partial \tilde Q_i \cap
\L$ and $\tau\equiv H(L)-1$ on $\partial \tilde Q_i \setminus \L$.  In
this situation, we have exactly one open $H(L)$-contour $\g$. Its law
can be written by applying Proposition~\ref{griglia} to the partition
function below and above $\gamma$:
\[
\pi^\tau_{\tilde Q_i}(\gamma)\propto\exp\left[-\beta|\gamma|+\Psi_{\tilde Q_i}(\g)-\frac\l LA(\g)+o(1)\right]
\]
with $A(\g)$ the signed area of $\g$ w.r.t.\ the bottom boundary of $\L$
(the minus sign in front of the area is due to the fact that we are looking at the area below $\g$ and not above it).

Call $P(\cdot)$ the law on $\g$ given by
\begin{eqnarray}
  \label{eq:Pii}
P(\g)\propto \exp\left[-\beta|\gamma|+\Psi_{\tilde Q_i}(\g)\right],
\end{eqnarray}
without the area term. From \cite{Hry} we know that for $L\to\infty$ and
rescaling the horizontal (resp. vertical) space direction by $L^{-(2/3-\epsilon)}$
(resp. $L^{-(1/3-\epsilon/2)}$), the law $P(\cdot)$ converges weakly to that of a Brownian
Bridge on $[0,1]$, with a suitable diffusion constant.
Now let $U_i$  be the event that $\gamma$ has maximal height less than
$L^{1/3-\epsilon/2}$ above the mid-point $x_i$ and we want to prove
\begin{eqnarray}
\label{eq:rettangoletto}
\limsup_{L\to\infty}  \pi^\tau_{\tilde Q_i}(U_i)<1-\delta
\end{eqnarray}
with $\delta>0$.

If this is the case, then it follows that w.h.p.\ there is some $1\le
i\le K\sim c L^{1/3+\epsilon}$ such that the complementary event $U_i^c$ happens, and the
proof of the Theorem is concluded.  Actually, note that
\eqref{eq:rettangoletto} implies that, if we consider $L^{\epsilon/2}$
``adjacent'' points $x_i,x_{i+1},\dots,x_{i+L^{\epsilon/2}}$, w.h.p.\ the event $U_j^c$ happens for at least
one such $j$. Since $x_{i+L^{\epsilon/2}}-x_i=O(L^{2/3-\epsilon/2})$, we have proven the stronger version of the
Theorem, as in Remark~\ref{rem:forte}.

It remains to prove~\eqref{eq:rettangoletto}. Write
\begin{align}
  \label{eq:29}
 \pi^\tau_{\tilde Q_i}(U_i^c) =\frac{E \left(U_i^c; e^{-\frac \l L A(\gamma)+o(1)}\right)}{E \left( e^{-\frac \l L A(\gamma)+o(1)}\right)},
\end{align}
with $E$ the expectation w.r.t.\ the law $P$ of~\eqref{eq:Pii}.
The denominator
 is $1+o(1)$ (just use standard techniques \cite{DKS} to see that is very unlikely that $|A(\g)|$ is much larger than $L^{2/3-\epsilon}\times L^{1/3-\epsilon/2}$, as it should be for a random walk).
 As for the numerator, Cauchy-Schwartz gives
\[
E \left(U_i^c; e^{-\frac \l L A(\gamma)}\right)\le \sqrt{P(U_i^c)}\sqrt {E\left(e^{-2\frac \l L A(\gamma)+o(1)}\right)}.
\]
The second term is $1+o(1)$ like the denominator, while  the first factor
is uniformly bounded away from $1$ since, in the Brownian scaling mentioned above, the event $U_i^c$ becomes the event that the Brownian bridge on $[0,1]$ is
lower than $1$ at time $1/2$, an event that clearly does not have full probability.

\appendix

\section{\ }\addtocontents{toc}{\protect\setcounter{tocdepth}{1}}
\subsection{\ }
\label{sec:pbeta-pf}
This section contains the proof of Lemma~\ref{pbeta}, showing that for large enough $\beta$ the $\hat{\pi}$-probability that the height at 0 would exceed $h$ is $(c_\infty+\delta_h)e^{-4\beta h}$ for some $c_\infty=c_\infty(\beta)>0$ tending to one as $\b\to \infty$.
\begin{proof}[\textbf{\emph{Proof of Lemma~\ref{pbeta}}}]
Let $a_h = e^{4\beta h} \hat{\pi}\left(\eta_0 \geq h\right)$ and define
\begin{align*}
  \epsilon_1(\beta,h) = \frac{\hat{\pi}\left(\eta_0 \geq h~,~ S\geq h\right)}{\hat{\pi}\left(\eta_0 \geq h~,~ S\leq h-1\right)}\,,\qquad
  &\epsilon_2(\beta,h) = \frac{\hat{\pi}\left(\eta_0 \geq h-1~,~ S\geq h\right)}{\hat{\pi}\left(\eta_0 \geq h-1~,~ S\leq h-1\right)}\,,
\end{align*}
where $S = \max\{\eta_x : x\sim 0\}$ is the maximum height of a neighbor of the origin. With this notation,
\begin{equation}
  \label{eq-series-ratio}
\frac{a_h}{a_{h-1}} =
e^{4\beta} \cdot \frac{1+\epsilon_1(\beta,h)} {1+\epsilon_2(\beta,h)}
\cdot\frac{\hat{\pi}\left(\eta_0 \geq h ~,~ S\leq h-1\right)}
{\hat{\pi}\left(\eta_0 \geq h-1 ~,~ S\leq h-1\right)} = \frac{1+\epsilon_1(\beta,h)} {1+\epsilon_2(\beta,h)}\,,
\end{equation}
with the last equality due to the fact that $\hat{\pi}(\eta_0 \geq h,S\leq h-1)/ \hat{\pi}(\eta_0\geq h-1,S\leq h-1) = e^{-4\beta}$. Indeed, this fact easily follows from considering the bijective map $T$ which decreases $\eta_0$ by $1$ and deducts a cost of $4\beta$ from the Hamiltonian of every configuration associated with the numerator compared to its image in the denominator.

In order to bound $\epsilon_1$ and $\epsilon_2$, we note that there exists some absolute constant $c_1>0$ independent of $\b$ such that, for any $h\geq 1$ and any large enough $\beta$,
\[ \hat{\pi}\left(\eta_0\geq h~,~ S\geq h\right) \leq c_1 e^{-6\beta h}\,.\]
(This follows from~\cite{CLMST}*{Section~7}, or alternatively from the proof of~\cite{CLMST}*{Proposition~3.9}).
On the other hand, by~\cite{CLMST}*{Proposition~3.9} we know that $\hat{\pi}(\eta_0 \geq h) \geq \frac12 \exp(-4\beta h)$ for any $h\geq 0$ and $\beta\geq 1$. By combining these with an analogous argument for $\epsilon_2$ we deduce that
\[ 0 \leq \epsilon_1(\beta,h) \leq c_2 e^{-2\beta h}\,,\qquad 0\leq \epsilon_2 \leq c_2 \min(e^{-2\beta (h-1)},e^{-4\b})\,,
\]
where $c_2>0$ is some absolute constant independent of $\b$. Revisiting~\eqref{eq-series-ratio} then gives that
\[ 1 - c_2 \min(e^{-2\beta (h-1)},e^{-4\b}) \leq \frac{a_h}{a_{h-1}} \leq 1+ c_2 e^{-2\beta h}\,,\]
which readily implies that $c_\infty = \lim_{h\to\infty}a_h$ exists together with $|c_\infty -a_h|\le c_3 e^{-2\b h}$.
Finally if we write $c_\infty=a_0 \prod_{h=1}^\infty (a_h/a_{h-1})$ and use $\lim_{\beta\to\infty} a_0 = 1$ together with the above bounds we immediately get that
$\lim_{\b\to \infty}c_\infty(\beta)=1$.
\end{proof}

\subsection{\ }
%
\label{grigliatina}

Fix $L\in\bbN$, $\e>0$, and consider $\L\subset\bbZ^2$ with area and external boundary such that
 \begin{equation}
\label{La}
|\L|\leq L^{\frac 43+2\e}\,,\qquad |\partial \L|\leq L^{\frac 23 + 2\e }
\end{equation}
Let $Z_{\L,U}^{h,\pm}$ and $\hat Z^{\pm}_{\L,U}$ be defined as in Proposition~\ref{prop:grigliatona}.
\begin{proposition}\label{griglia}
Fix $\b\geq \b_0$ and $\e\in(0,1/20)$, and assume~\eqref{La}. Set $H(L)=\lfloor\frac1{4\b}\log L\rfloor$, and $h=H(L)-n$, $n=0,1,\dots$.
Then for all fixed $n$, 
 \begin{equation}
\label{griglia1}
Z^{h,\pm}_{\L,U} = \hat Z^{\pm}_{\L,U} \,\exp{\left(-\hat\pi(\eta_0>h)|\L|+ o(1)\right)},
\end{equation}
where 
$\hat\pi$ is the probability obtained as infinite volume limit with zero boundary conditions of the SOS model at inverse temperature $\b$.
\end{proposition}
\begin{proof}
By shifting the heights by $-h$ one can pretend that there are $0$ b.c., that the floor is at $-h$ and that on $U$ the heights are all $\geq 0$ (resp.\ $\leq 0$). Let $\o_\pm$ denote the associated probability measure with no floor, 
so that
\begin{equation}
\label{griglia4}
\frac{Z^{h,\pm}_{\L,U}}{\hat Z^{\pm}_{\L,U}} = \o_\pm\left(\eta_x\geq - h\,,\;x\in\L\right)
\end{equation}
Consider first the lower bound.
Notice that $\o_\pm$ satisfies the FKG property. Thus
\begin{align*}
\o_\pm\left(\eta_x\geq - h\,,\;x\in\L\right)
\geq \prod_{x\in\L}\o_\pm\left(\eta_x\geq - h\right).
\end{align*}
As in \cite{CLMST}*{Section~7, Eq.~(7.19)}, one has for some constant $C>0$, for all $j\in\bbN$:
 \begin{equation}
\label{griglia2}
C^{-1}e^{-4\b j}\leq \o_\pm(\eta_x\geq j)\leq C\, e^{-4\b j}.
\end{equation}
In particular, $\o_\pm(\eta_x<-h)=O(L^{-1})$ for all fixed $n$, and
$$
\o_\pm\left(\eta_x\geq - h\right) = 1- \o_\pm(\eta_x<- h) = \exp{\left[-\o_\pm(\eta_x<- h) + O(L^{-2})\right]}.
$$
For all $x\in\L$ at distance $L^{\e}$ from $\partial \L$, one can write $\hat\pi(\eta_0>h)$ instead of $\o_\pm(\eta_x<- h)$ with an additive error that is $O(L^{-p})$ for any $p>1$, see \cite{CLMST}*{Eq.~(7.47)}.
Using~\eqref{La}, this proves the lower bound
 \begin{equation}
\label{griglia3}
\o_\pm\left(\eta_x\geq - h\,,\;x\in\L\right) \geq \exp{\left(-\hat\pi(\eta_0>h)|\L|+ O(L^{-\frac 13+4\e})\right)}.
\end{equation}
For the upper bound,
we will use essentially
the same argument in the proof of~\cite{CLMST}*{Claim 7.7}. We sketch the main steps below.
From~\eqref{griglia4}, setting $\psi_x=1_{\{\eta_x< -h\}}$, one writes
\begin{equation}
\label{griglia40}
\frac{Z^{h,\pm}_{\L,U}}{ \hat Z^{\pm}_{\L,U}} = \o_\pm\Big(\prod_{x\in\L}(1-\psi_x)\Big)
\end{equation}
Partition $\bbZ^2$
into squares $P$ with side $r=L^u+2L^\kappa$, where $0<\kappa<u$ will be fixed later (we assume for simplicity that $L^u,L^\kappa$ are both integers).
Consider squares $Q$ of side $L^u$ centered inside the squares $P$ in such a way that
each square $Q$ is surrounded within $P$ by a shell of thickness $L^\kappa$.
The set $\cS$ of dual bonds associated to a non-zero height gradient is decomposed into connected components (clusters) $S$. 
We call $\cI(\kappa)$ the collection  of clusters $S$ in $\cS$ such that $|S|\geq L^\kappa$, where $|S|$ denotes the number of edges in $S$.
A point $x\in\L$ can be of four types: 1) those whose distance from the boundary $\partial \L$ is less than $L^{\e}$; 2) those that belong to a shell in some $P\setminus Q$; 3) those belonging to a square $Q\subset P$ such that $P$ intersects
one of the clusters in $\cI(\kappa)$ or its interior; 4) all other  $x\in\L$.

We now spell out the contribution of each type of points to~\eqref{griglia40}. We estimate by $1$ the indicator $1-\psi_x$ for all $x$ of type 1,2, and 3:
$$
\o_\pm\Big(\prod_{x\in\L}(1-\psi_x)\Big)\leq \o_\pm\Big(\prod_{x\in\L'}(1-\psi_x)\Big),
$$
where $\L'$ denotes the (random) set of points of type 4.
Next, we claim that
\begin{equation}
\label{griglia5}
\o_\pm\Big(\prod_{x\in\L'}
(1-\psi_x)\Big)\leq \o_\pm\left(\exp{\left[-\hat\pi(\eta_0>h)|\L'|+o(1)\right]}\right).
\end{equation}
Once this bound is available one can conclude 
by showing that
\begin{equation}
\label{griglia7}
\o_\pm\left(\exp{\left[\hat\pi(\eta_0>h)|\L\setminus\L'|\right]}\right)=1+o(1).
\end{equation}
Let us first prove~\eqref{griglia7}.
Recall that 
$\hat\pi(\eta_0>h)=O(L^{-1})$.
Using~\eqref{La}, one finds that type 1 points contribute $O(L^{-\frac 13+4\e})$ to the exponent in~\eqref{griglia7}.
Moreover, there are $O(L^{\kappa+u})$ points in each shell and
$O(L^{\frac 43+2\e-2u})$ shells, therefore type 2 points contribute $O(L^{\frac 13+2\e-u+\kappa})$.
Exactly as in \cite{CLMST}, using the fact that contours in $\cI(\kappa)$ have area at least $L^\kappa$ and at most $L^{4/3+2\e}=o(L^2)$, one has that type 3 points only contribute $o(1)$ for any fixed $\kappa>0$, cf.\ Eq.~(7.53) there. We see that choosing e.g.\ $u=\frac 13+4\e$, $\kappa=\e$, one has that the total contribution of all points of type 1,2, and 3 is $o(1)$.
This proves~\eqref{griglia7}.

It remains to prove~\eqref{griglia5}.
We follow \cite{CLMST}. Any point of type 4 must have the property that it belongs to some
square $Q$ that is surrounded by a circuit $\cC$ within the shell around $Q$ inside $P$ such that all heights are equal to zero on $\cC$. Then, by conditioning on the circuits $\cC$ one can proceed by expanding separately the different squares $Q$. Fix a square $Q$ and let $\hat\pi_\cC^0$ denote the SOS measure with b.c.\ zero on the circuit $\cC$ surrounding $Q$. Then
Eq.\ (7.56) in \cite{CLMST} yields
  \begin{equation}
\label{griglia6}
\hat\pi_\cC^0\Big(\prod_{x\in Q}(1-\psi_x)\Big)\leq \exp{\left(-\sum_{x\in Q}\hat\pi_\cC^0(\psi_x)+
O(L^{-3/2+2u+ c(\b)})+O(L^{6u-3})\right)},
\end{equation}
where $c(\b)\to0$ as $\b\to\infty$.
Using exponential decay of correlations (\cite{BW}), one can replace
 $\sum_{x\in Q}\hat\pi_\cC^0(\psi_x)$ by $|Q|\hat\pi(\eta_0>h) + O(L^{u+\d-1})$ for any $\d>0$.
Therefore, taking the product over all squares $Q$ containing points of type 4, and taking the average over the realizations of $\L'$ one finds
\eqref{griglia5},
since there are at most $O(L^{\frac43+2\e-2u})$ squares $Q$ in $\L'$ and $u=\frac 13+4\e$, and one can absorb all the errors in the $o(1)$ term if $\b$ is large enough.
This ends the proof of Proposition~\ref{griglia}.
\end{proof}


\subsection{Proof of~\eqref{eq:condizionata}}\label{a44}
Let $\o_\pm$ be defined as in Section~\ref{grigliatina} above. Let also $\bar\o_\pm$ denote the
probability measure $\o_\pm$ conditioned on the event that there are no macroscopic contours. Then,~\eqref{eq:condizionata} becomes equivalent to
 \begin{equation}
\label{grigliata4}
\bar \o_\pm\left(\eta_x\geq - h\,,\;x\in\L\right)\leq \exp{\left(-\hat\pi(\eta_0>h)|\L|+ O(L^{\frac12 +c(\b)})\right)}.
\end{equation}
We proceed as in the proof of Proposition~\ref{griglia}. The proof of Equation~\eqref{griglia5} now
yields
\begin{equation}
\label{grigliata5}
\bar \o_\pm\Big(\prod_{x\in\L}
(1-\psi_x)\Big)\leq \bar \o_\pm\left(\exp{\left[-\hat\pi(\eta_0>h)|\L'|+ CL^{1-u+\d} + CL^{\frac12+c(\b)} + CL^{4u-1} \right]}\right),
\end{equation}
where $C>0$ is a constant, and $\d>0$ is arbitrary.
The error terms above are explained as follows: there are at most $O(L^{2-2u})$ squares $Q$ in $\L_0$ and
for each one of those one has a term $O(L^{u+\d-1})$ coming from the boundary of $Q$, and a term $O(L^{-\frac32 + 2u +c(\b)})+O(L^{6u-3})$
coming from the expansion~\eqref{griglia6}.

Next, we need the statement corresponding to~\eqref{griglia7}. Thanks to the assumption on the absence of large contours, here  there are no points of type 3. Thus the argument behind~\eqref{griglia7} here gives
\begin{equation}
\label{grigliata7}
\bar \o_\pm\left(\exp{\left[\hat\pi(\eta_0>h)|\L\setminus\L'|\right]}\right)\leq \exp{\left[O(L^{-1+\e}|\partial \L|) + O(L^{2-u+\kappa})\right]}.
\end{equation}
The error terms above are explained as follows: the first term is the worst case contribution of boundary terms (points of type 1, i.e.\ those that are at distance from $|\partial \L|$ at most $L^\e$); the second term is due to the points of type 2 which are at most $ O(L^{2-u+\kappa})$.

Finally, we can combine~\eqref{grigliata5} and~\eqref{grigliata7}. Taking $\kappa=\d$ sufficiently small, with e.g.\  $u=3/5$, using $|\partial \L|=O(L^{1+\e})$, one finds that the dominant error term is $O(L^{\frac12 + c(\b)})$. This implies the desired upper bound.

\subsection{\ }
\label{app:ce}
Given $A$ and $B$ on $(\mathbb Z^2)^*$, let
 $\Xi_{A,B}$ the set of open contours from $A$ to $B$ that stay within
 $S_{A,B}$, the infinite strip delimited by the vertical lines going
 through $A$ and $B$.
Contours are self-avoiding paths, with the usual South-West splitting
rule (see e.g.\ Definition~3.3 in~\cite{CLMST}, where closed contours
are defined).
For $\Gamma\in \Xi_{A,B}$, let
\[
w(\Gamma)=\exp\left(-\beta|\Gamma|+\Psi_{S_{A,B}}(\Gamma)\right)
\]
where
$|\Gamma|$ is the geometric length of $\Gamma$. The
``decoration term'' $\Psi_{S_{A,B}}(\gamma)$ was defined in
\eqref{eq:28}
for closed contours: in the present case of an open contour
$\G\in\Xi_{A,B}$,
it is understood that
$\L_\g$is the
subset of $S_{A,B}$ above $\Gamma$ and $\Delta^+_\gamma=\Delta_\g\cap\L_\g$
(resp.\
$\D^-_\g= \Delta_\g\cap(S_{A,B}\setminus \L_\g)$). Cf.\ Definition~\ref{contourdef} for $\D_\g$.
%
%

The following limit exists \cite{DKS}:
\[
\tau(\theta)=-\lim \frac1{\beta \,|A-B|}\log \sum_{\Gamma\in \Xi_{A,B}}w(\Gamma)
\]
where the limit consists in letting $|A-B|$ (the Euclidean distance
between $A$ and $B$)
diverge while the
angle formed by the segment $AB$ with the horizontal axis tends to
$\theta\in (-\pi/2,\pi/2).$

\begin{remark}
As remarked in \cite{DKS}, there is some arbitrariness in the choice
of $\Xi_{A,B}$: for instance, one could replace $\Psi_{S_{A,B}}(\G)$
with
$\Psi_{V}(\G)$ for some set $V$ containing a
$|A-B|^{\small\frac{1}{2}+\epsilon}$-neighborhood
of the segment $AB$, and the resulting surface tension would be unchanged.
\end{remark}

\subsection{\ }
\label{app:1}
We briefly discuss the missing details in the proof of Theorem
\ref{th:napalla} given by the wiggling of the regular circuit
$\cC_*$. Referring to Figure~\ref{fig:ghiri},
\begin{figure}
\psfrag{xa}{$A$}
\psfrag{xb}{$B$}
\psfrag{A}{$\hat A$}\psfrag{B}{$\hat B$}
\psfrag{Q}{$Q$}
\psfrag{l2}{$\bbL_2$}
\psfrag{l1}{$\bbL_1$}
\psfig{file=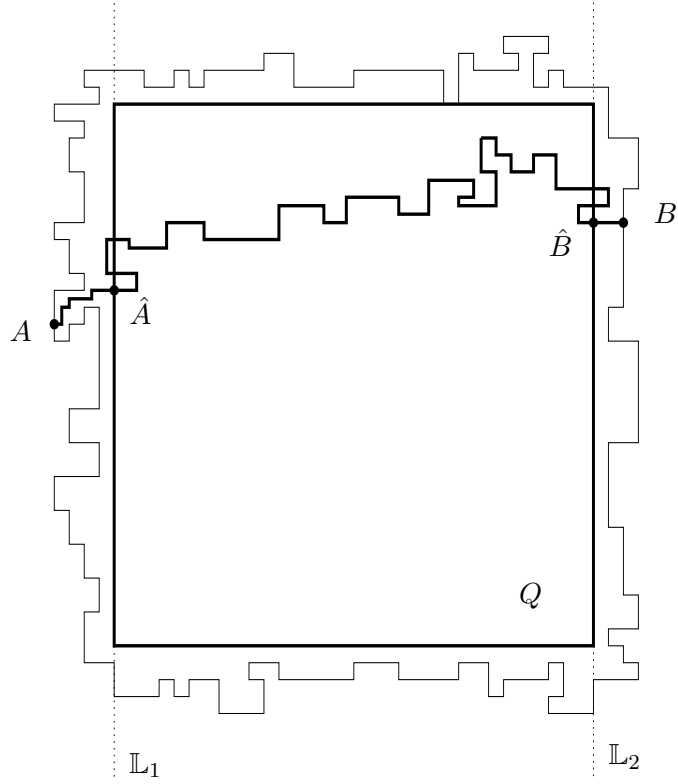,width=0.55\textwidth,height=0.65\textwidth}
\vspace{-0.15in}
\caption{
The rectangle $Q$, the regular circuit
$\cC_*$ (closed wiggled
line), the contour $\Gamma$ (thick wiggled line)
and the points $\hat A,\hat B,A,B$ defined in the text.
}
\vspace{-0.15in}
\label{fig:ghiri}
\end{figure}
let $\hat A,\hat B$ the first and last intersection of $\G$
with the two vertical lines $\bbL_1,\bbL_2$ going through the vertical
sides of $Q$, and let $A,B$ be the points as in Definition \ref{bcc}; see also Figure \ref{fig:ghiro}. The contribution to the
area term $A(\G)/L$ coming from the parts of $\G$ before $\hat A$ and after
$\hat B$ is $o(1)$. Moreover the probability that the height difference
between $\hat A,A$ or $\hat B,B$ is larger that $L^{1/4}$ is smaller than
$\exp(-cL^{1/4})$. In fact the circuit $\cC_*$, being a regular one, cannot
deterministically force such height differences to be larger than
$O((\log L)^2)$. Thus a large (of order $L^{1/4}$) height difference
can only be produced by a large ``spontaneous'' deviation of the
contour $\G$. Without the area term such a deviation has probability
$O(\exp(-cL^{1/4}))$  (see~\eqref{***}). The area term can be taken
care of via an application of the Cauchy-Schwarz inequality: using
again~\eqref{***},
the average of $\exp(2\lambda A(\Gamma)/L)$ is of order
$\exp(L^{3\epsilon})$. Notice that $L^{1/4}$ is negligible w.r.t.\
$Y(n)-(a+b)/2 \geq c L^{\frac 13 +2\epsilon}$ defined in Theorem~\ref{th:napalla}.

Thus, conditioning on the parts $\G_{\rm left},\G_{\rm right}$ of
$\G$ before $\hat A$ and after $\hat B$, with $\hat A$ at distance $O(L^{1/4})$ from
$A$ and similarly for $\hat B$, we have reduced ourselves to a geometry to which
we can apply directly Corollary~\ref{iteration} as in the case where
the circuit $\cC_*$ is the boundary of $Q$.

\section*{Acknowledgments}
We are grateful to Senya Shlosman for valuable discussions on related cluster expansion questions, as well as to Ofer Zeitouni for illuminating discussions on entropic repulsion in the GFF. Major parts of this work were carried out
during visits to ENS Lyon, Universit\`a Roma Tre and the Theory Group of Microsoft Research, Redmond. PC, FM, AS and FLT would like to thank the Theory Group for its hospitality and for creating a stimulating research environment.
%



\end{document}